%% file: 0_main.tex
\documentclass[reqno, 12pt]{amsart}

\usepackage[rgb]{xcolor} 
\definecolor{mygreen}{rgb}{0,0.7,0.3}
\definecolor{myblue}{rgb}{0,0.50,1.20}
\definecolor{myorange}{rgb}{1,0.5,0.1}

\definecolor{fillred}{rgb}{1,0.9,0.9}
\definecolor{fillgreen}{rgb}{0.9,1,0.9}

\definecolor{refkey}{rgb}{0,0.7,0.3}
\definecolor{labelkey}{rgb}{1,0,0}

\usepackage{amsmath}
\usepackage{amsfonts}
\usepackage{amssymb}
\usepackage{amsthm}
\usepackage{mathrsfs}
\usepackage{bbm}
\usepackage{stmaryrd}
\usepackage{graphicx}
\usepackage{tikz}
\usetikzlibrary{calc,decorations.markings}
\usetikzlibrary{arrows.meta}
\usetikzlibrary{positioning}
\usetikzlibrary{cd, intersections, calc, decorations.pathmorphing, arrows, decorations.pathreplacing}

\usepackage{mathtools}

\usepackage{enumerate}

\usepackage{subfiles}
\usepackage[colorlinks, linkcolor=blue, citecolor=red, urlcolor=cyan,pagebackref]{hyperref}

\numberwithin{equation}{section}

\usepackage{cleveref}
\crefname{thm}{Theorem}{Theorems}
\crefname{cor}{Corollary}{Corollaries}
\crefname{lem}{Lemma}{Lemmas}
\crefname{sublem}{Sublemma}{Sublemmas}
\crefname{prop}{Proposition}{Propositions}
\crefname{dfn}{Definition}{Definitions}
\crefname{defi}{Definition}{Definitions}
\crefname{ex}{Example}{Examples}
\crefname{claim}{Claim}{Claims}
\crefname{conj}{Conjecture}{Conjectures}
\crefname{conv}{Notation}{Notations}
\crefname{rem}{Remark}{Remarks}
\crefname{rmk}{Remark}{Remarks}
\crefname{prob}{Problem}{Problems}
\crefname{figure}{Figure}{Figures}
\crefname{section}{Section}{Sections}
\crefname{subsection}{Section}{Sections}
\crefname{appendix}{Appendix}{Appendices}
\crefname{introthm}{Theorem}{Theorems}
\crefname{introcor}{Corollary}{Corollaries}
\crefname{introconj}{Conjecture}{Conjectures}

\newtheorem{thm}{Theorem}[section]
\newtheorem{prop}[thm]{Proposition}
\newtheorem{cor}[thm]{Corollary}
\newtheorem{lem}[thm]{Lemma}

\newtheorem{introthm}{Theorem}

\theoremstyle{definition}
\newtheorem{dfn}[thm]{Definition}

\newtheorem{ex}[thm]{Example}

\theoremstyle{remark}

\newtheorem{rem}[thm]{Remark}

\newcommand{\corrected}
{\textcolor{red}{(corrected)}}

\newcommand*{\chom}{\mathcal{H}\kern -.5pt om}

\renewcommand{\mathbf}{\boldsymbol}

\newcommand{\bZ}{\mathbb{Z}}
\newcommand{\bQ}{\mathbb{Q}}
\newcommand{\bR}{\mathbb{R}}
\newcommand{\bC}{\mathbb{C}}

\newcommand{\bS}{\mathbb{S}}

\newcommand{\bM}{\mathbb{M}}

\newcommand{\bP}{\mathbb{M}_\circ}
\newcommand{\bG}{\mathbb{G}}

\newcommand{\bH}{\mathbb{H}}
\newcommand{\bB}{\mathbb{B}}

\newcommand{\aT}{\mathcal{T}^a}
\newcommand{\xT}{\mathcal{T}^x}
\newcommand{\pT}{\mathcal{T}^p}

\newcommand{\A}{\mathcal{A}}

\newcommand{\cC}{\mathcal{C}}

\newcommand{\cH}{\mathcal{H}}

\newcommand{\cL}{\mathcal{L}}

\newcommand{\cO}{\mathcal{O}}
\def\P{{\mathcal{P}}}

\newcommand{\cT}{\mathcal{T}}

\newcommand{\X}{\mathcal{X}}

\newcommand{\sfx}{\mathsf{x}}
\newcommand{\sfa}{\mathsf{a}}
\newcommand{\ve}{\varepsilon}

\newcommand{\inn}{\mathrm{in}}
\newcommand{\out}{\mathrm{out}}

\newcommand{\Hom}{\mathrm{Hom}}

\newcommand{\tri}{\triangle}
\newcommand{\sgn}{\mathrm{sgn}}

\newcommand{\pos}{\mathbb{R}_{>0}}

\newcommand{\uf}{\mathrm{uf}}

\newcommand{\Teich}{Teichm\"uller}
\newcommand{\spiral}{\hL}

\newcommand{\hL}{\widehat{L}}

\newcommand{\Skein}[1]{\mathscr{S}_{#1}}
\newcommand{\Skeinr}[1]{\overline{\mathscr{S}}_{#1}}

\DeclareMathOperator{\interior}{\mathrm{int}}

\makeatletter
\newcommand{\oset}[3][0ex]{%
  \mathrel{\mathop{#3}\limits^{
    \vbox to#1{\kern-2\ex@
    \hbox{$\scriptstyle#2$}\vss}}}}
\makeatother

\makeatletter
\newcommand{\osetnear}[3][0ex]{%
  \mathrel{\mathop{#3}\limits^{
    \vbox to#1{\kern-.3\ex@
    \hbox{$\scriptstyle#2$}\vss}}}}
\makeatother

\tikzset{pics/.cd,
handle/.style={code={
\draw (-0.72,0) to[bend left] (0.72,0);
\draw (-0.9,0.1) to[bend right] (0.9,0.1);
}}}
 
\newcommand\angleAL[1]{([xshift=-4pt] #1) -- ([xshift=-4pt,yshift=4pt] #1) -- ([yshift=4pt] #1)}
\newcommand\angleBL[1]{([xshift=-4pt] #1) -- ([xshift=-4pt,yshift=-4pt] #1) -- ([yshift=-4pt] #1)}

\tikzset{
  mid arrow/.style={postaction={decorate,decoration={
        markings,
        mark=at position .5 with {\arrow[#1]{stealth}}
      }}},
}

\newcommand{\hgline}[3]{
\pgfmathsetmacro{\thetaone}{mod(#1,360)}
\pgfmathsetmacro{\thetatwo}{mod(#2,360)}
\pgfmathsetmacro{\rho}{#3}
\pgfmathsetmacro{\theta}{(\thetaone+\thetatwo)/2}
\pgfmathsetmacro{\phi}{abs(\thetaone-\thetatwo)/2}
\pgfmathsetmacro{\close}{less(abs(\phi-90),0.0001)}
\ifdim \close pt = 1pt
    \draw (\thetaone:\rho) -- (\thetatwo:\rho);
\else
    \pgfmathsetmacro{\R}{\rho*tan(\phi)}
    \ifdim \R pt < 0pt
        \pgfmathsetmacro{\distance}{sqrt(\rho^2+\R*\R)}
        \draw (\theta:-\distance) circle (\R);
    \else \ifdim \R pt > 0pt
        \pgfmathsetmacro{\distance}{sqrt(\rho^2+\R^2)}
        \draw (\theta:\distance) circle (\R);
        \fi
    \fi
\fi
}

\newcommand{\snake}[2]{\draw[thick,-{Classical TikZ Rightarrow[length=4pt]},decorate,decoration={snake,amplitude=2pt,pre length=1pt,post length=2pt}](#1) -- (#2)}

\tikzset{
    partial ellipse/.style args={#1:#2:#3}{
        insert path={+ (#1:#3) arc (#1:#2:#3)}
    }
}

\newcommand{\bline}[3]{
    \path (#1)++(0,-#3) coordinate(m1);
    \path (#2)++(0,-#3) coordinate(m2);
    \filldraw[gray!30] (m1) -- (#1) -- (#2) -- (m2) --cycle;
    \draw[thick] (#1) -- (#2);
}
\newcommand{\tline}[3]{
    \path (#1)++(0,#3) coordinate(m1);
    \path (#2)++(0,#3) coordinate(m2);
    \filldraw[gray!30] (m1) -- (#1) -- (#2) -- (m2) --cycle;
    \draw[thick] (#1) -- (#2);
}

\tikzset{->-/.style 2 args={
	postaction={decorate},
	decoration={markings, mark=at position #1 with {\arrow[thick, #2]{>}}} 
    },
    ->-/.default={0.5}{}
}
\tikzset{-<-/.style 2 args={
	postaction={decorate},
	decoration={markings, mark=at position #1 with {\arrow[thick, #2]{<}}} 
    },
    -<-/.default={0.5}{}
}

\newcommand{\pinn}[4]{
\draw(#1)++(#2:#3) --++(#2+180:2*#3) node[fill,circle,inner sep=#4]{}; 
}

\pgfdeclarelayer{bg}    
\pgfsetlayers{bg,main}  

\title{Teichm\"uller and lamination spaces with pinnings}

\author[Tsukasa Ishibashi]{Tsukasa Ishibashi}
\address{Tsukasa Ishibashi, Mathematical Institute, Tohoku University, 
6-3 Aoba, Aramaki, Aoba-ku, Sendai, Miyagi 980-8578, Japan.}
\email{tsukasa.ishibashi.a6@tohoku.ac.jp}

\date{\today}

\begin{document}
\maketitle

\begin{abstract}
    We describe the spaces of the positive and tropical points of the moduli space $\P_{PGL_2,\Sigma}$ introduced by Goncharov--Shen \cite{GS19} as certain \Teich\ and lamination spaces, respectively, with additional data of \emph{pinnings}. In the case where the surface $\Sigma$ has no punctures, we obtain the formulae relating various functions on the \Teich\ space with pinnings: $\lambda$-lengths, cross ratio coordinates, and Wilson lines. A topological description of the tropicalized amalgamation map is given in terms of $\P$-laminations. 
    Based on our topological study of these ``$\P$-type" spaces, we investigate the compatibility of the Fock--Goncharov duality maps $\mathbb{I}_\A$, $\mathbb{I}_\X$ constructed by \cite{FG06,FG07,MSW,GS15} under the extended ensemble map. We also discuss the amalgamation of bracelets bases.
\end{abstract}

\tableofcontents

\section{Introduction}
The \Teich\ theory and the theory of cluster varieties \cite{FG09} are deeply connected, producing fruitful applications in the both sides. 
Given a marked surface $\Sigma$, we have two kinds of dual cluster varieties, called the \emph{cluster $K_2$-variety} $\A_\Sigma$ and the \emph{cluster Poisson variety} $\X_\Sigma$ \cite{FG09} defined by certain quivers associated with ideal triangulations of $\Sigma$. 
The positive structure on these spaces allows us to consider their sets of semifield-valued points, for example the positive parts $\A_\Sigma(\pos)$ and $\X_\Sigma(\pos)$, which are real-analytic manifolds. On the other hand, 
there are two extensions of the usual \Teich\ space: the \emph{decorated \Teich\ space} $\cT^a(\Sigma)$ introduced by Penner \cite{Penner87} and the \emph{enhanced \Teich\ space} $\cT^x(\Sigma)$ systematically studied by Fock and Goncharov \cite{Fock,FG07}. See Penner's book \cite{Penner} for details. It is known \cite{FG07,FST} that we have canonical isomorphisms
\begin{align*}
    \cT^a(\Sigma) \cong \A_\Sigma(\pos), \quad \cT^x(\Sigma) \cong \X^\uf_\Sigma(\pos),
\end{align*}
which are equivariant under the natural actions of the mapping class group $MC(\Sigma)$. 
These isomorphisms are provided by special coordinate functions on these \Teich\ spaces, called the $\lambda$-lengths and the cross ratios, respectively. Here, $\X^\uf_\Sigma$ stands for the cluster Poisson variety \underline{without frozen coordinates} -- there is no natural way to define the cross ratio coordinates on $\cT^x(\Sigma)$ associated to the boundary edges. The supplement of frozen coordinates on boundary intervals is the main theme in this paper. 

The varieties $\A_\Sigma$ and $\X^\uf_\Sigma$ are birationally isomorphic to certain moduli spaces $\A_{SL_2,\Sigma}$ and $\X_{PGL_2,\Sigma}$ of local systems on $\Sigma$ \cite{FG06}. Here, the moduli space $\X_{PGL_2,\Sigma}$ also misses frozen coordinates. 
After a decade, in their seminal paper \cite{GS19}, Goncharov--Shen introduced a new moduli space $\P_{PGL_2,\Sigma}$ closely related to $\X_{PGL_2,\Sigma}$, but with additional data called the \emph{pinnings}. The data of pinnings allows one to define frozen coordinates as well, and thus provides a birational isomorphism $\X_\Sigma \cong \P_{PGL_2,\Sigma}$.

\subsection{\Teich\ space with pinnings} 
In this paper, we introduce a variant of the \Teich\ space corresponding to $\P_{PGL_2,\Sigma}$, which we call the \emph{\Teich\ space with pinnings} $\cT^p(\Sigma)$. Although it should be nothing but a certain ``real locus" of the moduli space $\P_{PGL_2,\Sigma}$, what we elaborate in this paper is its description purely in terms of the hyperbolic geometry. Mimicking \cite[Lemma-Definition 3.7]{GS19} in our setting, we introduce the notion of pinnings in four equivalent ways, and define $\cT^p(\Sigma)$ to be the \Teich\ space of marked hyperbolic structures equipped with such data on each boundary interval (\cref{dfn:T^p}). Then we define cross ratio coordinates on $\cT^p(\Sigma)$, and show that they combine to give an $MC(\Sigma)$-equivariant isomorphism (\cref{cor:Teich_X-variety})
\begin{align*}
    \cT^p(\Sigma) \xrightarrow{\sim} \X_\Sigma(\pos).
\end{align*}
We also describe the \emph{gluing map} \cite{GS19} in terms of the hyperbolic structures (\cref{prop:amalgamation}). It also clarifies the appearance of spiralling geodesics in the enhanced \Teich\ space in relation with the Thurston's completeness criterion. 

The decorations induce pinnings. Hence we get an \emph{ (extended) ensemble map}
\begin{align}\label{eq:ensemble_map}
    p_\Sigma: \cT^a(\Sigma) \to \cT^p(\Sigma).
\end{align}
The coordinate expression of the map $p_\Sigma$ (\cref{prop:ensemble}) is exactly the one known in the cluster theory, enhanced by Goncharov--Shen \cite[Section 18]{GS19}. It expresses the cross ratios as Laurent monomials of $\lambda$-lengths. 
If $\Sigma$ has no interior marked points (\emph{i.e.}, punctures), it turns out that $p_\Sigma$ is invertible. Then we obtain the inverse formula which expresses the $\lambda$-lengths in terms of the cross ratios, which seems to be well-known to specialists but new in the literature:

\begin{introthm}[\cref{thm:A to X}]
    Assume that $\Sigma$ has no punctures. 
Then for each edge $\alpha \in e(\tri)$ of an ideal triangulation, we have the inverse formula
\begin{align*}
    A_\alpha = \prod_{\beta \in e(\tri)} (X^\tri_{\beta})^{q_{\alpha\beta}}.
\end{align*}
Here $q_{\alpha\beta}:=-\sfa_{\beta}(\alpha_\bB)$, and $\sfa_{\beta}(\alpha_\bB) \in \frac{1}{2}\bZ_{\geq 0}$ denotes half the geometric intersection number between the curves $\beta$ and the positive $\bB$-shift $\alpha_\bB$ (\cref{def:shift_ideal}) of the ideal arc $\alpha$.
\end{introthm}
As a consequence, we can compute the Poisson brackets of $\lambda$-lengths. We see that the Poisson algebra $C^\infty(\cT^p(\Sigma))$ is a classical analogue of the Muller's skein algebra \cite{Muller}. 

We also investigate the \emph{Wilson lines} introduced in \cite{IO20} in terms of hyperbolic geometry. We obtain the transition formulae between the $\lambda$-length/cross ratio coordinates and the matrix coefficients of Wilson lines (\cref{thm:LR-formula,prop:Wilson_lambda}). 

\subsection{Lamination space with pinnings}
For $\mathbb{A}=\bZ,\bQ$ or $\bR$, let $\mathbb{A}^{\!\mathsf{T}}=(\mathbb{A},\max,+)$ denote the (max-plus) tropical semifield. Then we can consider the sets $\A_\Sigma(\mathbb{A}^{\!\mathsf{T}})$ and $\X^\uf_\Sigma(\mathbb{A}^{\!\mathsf{T}})$ of tropical points, which are known to be canonically isomorphic to certain spaces of measured laminations \cite{FG07}. Here $\X^\uf_\Sigma(\mathbb{A}^{\!\mathsf{T}})$ also misses the frozen coordinates. We introduce the space $\cL^p(\Sigma,\bQ)$ of \emph{rational $\P$-laminations}, and show that a natural extension of the shear coordinates gives an$MC(\Sigma)$-equivariant piecewise-linear isomorphism
\begin{align*}  
    \cL^p(\Sigma,\bQ) \xrightarrow{\sim} \X_\Sigma(\bQ^{\mathsf{T}}).
\end{align*}
Its $\mathfrak{sl}_3$-version has already appeared in the work \cite{IK22}. We introduce a gluing map (\cref{dfn:gluing_lamination}) purely in terms of laminations, and prove that it is a tropical analogue of the Goncharov--Shen's gluing map (\cref{prop:amalgamation}). 

Combining the results in the \Teich\ and lamination sides, we can form a ``$\P$-version" of the Thurston compactification $\overline{\cT^p(\Sigma)}:=\cT^p(\Sigma) \cup \bS \cL^p(\Sigma,\bR)$ (\cref{dfn:Thurston_P}). 
Here $\cL^p(\Sigma,\bR)$ is the completion of the space $\cL^p(\Sigma,\bQ)$ with respect to the shear coordinates. 
Then we obtain the following:

\begin{introthm}[\cref{thm:gluing_Thurston}]
The gluing maps on the \Teich\ and lamination spaces combine to give a continuous map
\begin{align*}
    \overline{q}_{\Sigma,\Sigma'}: \overline{\cT^p(\Sigma)} \to \overline{\cT^p(\Sigma')}
\end{align*}
between the Thurston compactifications. 
\end{introthm}

\subsection{Ensemble compatibility of duality maps}
Fock--Goncharov's duality conjecture is one of the most fascinating conjectures in the theory of cluster varieties. See \cite{FG09,GHKK}; also \cite{Qin21} for a recent review on this subject. It asks a construction of \emph{duality maps}
\begin{align}
    &\mathbb{I}_\X: \X_{\Sigma}(\bZ^{\mathsf{T}}) \to \cO(\A_{\Sigma}), \label{eq:duality_X}\\
    &\mathbb{I}_\A: \A_{\Sigma}(\bZ^{\mathsf{T}}) \to \cO(\X_{\Sigma}) \label{eq:duality_A}
\end{align}
that parametrize linear bases of the function algebras of cluster varieties, satisfying certain axioms formulated in \cite[Section 4]{FG09}. 

A topological construction of the duality map \eqref{eq:duality_X}, nowadays called the \emph{bracelets basis}, is first given by Fock--Goncharov \cite{FG06,FG07} for a general marked surface, and further studied by Musiker--Schiffler--Williams \cite{MSW} in the absence of punctures. A duality map in the direction \eqref{eq:duality_A} is also constructed by Fock--Goncharov \cite{FG06,FG07}, and further enhanced by Goncharov--Shen \cite{GS15} in the ``$\P$-type" setting. Here the work of Goncharov--Shen gives a basis of the function ring of the moduli space $\P_{SL_2,\Sigma}$ (written as $\mathrm{Loc}_{SL_2,S}$ \emph{loc.~sit.}) parametrized by the space $\cL^a(\Sigma,\bZ) \supset \A_\Sigma(\bZ^{\mathsf{T}})$ (whose elements are called \emph{$PGL_2$-laminations} \emph{loc.~sit.}). Essentially as a restriction of their construction, we obtain:


\begin{introthm}[\cref{thm:X_basis}]
Assume that $\Sigma$ is unpunctured, having at least two marked points. Then the functions $\mathbb{I}_\A(L)$, where $L$ runs over all the integral $\A$-laminations, form a linear basis of the function algebra $\cO(\X_{\Sigma})$. 
\end{introthm}
In this paper, we describe the functions $\mathbb{I}_\A(L)$  by assembling the trace functions along loops and certain matrix coefficients of Wilson lines along arcs. 
We give a proof of this theorem based on the description of $\cO(\X_{\Sigma})$ as the classical limit of the \emph{congruent subalgebra} of the reduced stated skein algebra \cite{IKar}. 
A proof similar to that of \cite[Theorem 10.14]{GS15} will be also possible, with the restriction to the representations of $PGL_2$. 

We then turn our attention to the compatibility of the duality maps \eqref{eq:duality_X} and \eqref{eq:duality_A} under the ensemble map \eqref{eq:ensemble_map}. While such a compatibility has been already formulated in \cite[Conjecture 4.1.3]{FG09}, the importance to extend the ensemble map on the frozen variables seems to be only recognized after then. Indeed, our ensemble map \eqref{eq:ensemble_map} is an extended version according to the choice made in \cite{GS19}. 
Our compatibility statement is the following, which is the main theorem of this paper:

\begin{introthm}[Ensemble compatibility of duality maps: \cref{prop:duality_compatible}]
For any unpunctured marked surface $\Sigma$, the following diagram commutes:
\begin{equation}\label{introeq:duality_compatible}
    \begin{tikzcd}
    \A_{\Sigma}(\bZ^{\mathsf{T}}) \ar[d,"\check{p}_\Sigma^{\mathsf{T}}"'] \ar[rr,"\mathbb{I}_\A"] && \cO(\X_{\Sigma}) \ar[d,"p_\Sigma^\ast"] \\
    \X_{\Sigma}(\bZ^{\mathsf{T}}) \ar[rr,"\mathbb{I}_\X"'] && \cO(\A_{\Sigma}),
    \end{tikzcd}
\end{equation}
where we use the Langlands dual ensemble map $ \check{p}_\Sigma^{\mathsf{T}}: \cL^a(\Sigma,\bZ) \to \cL^p(\Sigma,\bZ)$ \eqref{eq:dual_ensemble} on the tropical side, and $\mathbb{I}_\X$ denotes the bracelets basis (\cref{def:skein_lift_X}). 
\end{introthm}
Here it is remarkable that the non-trivial Langlands duality comes into play in order to get the commutative diagram \eqref{introeq:duality_compatible}, even if the exchange matrix is skew-symmetric. Actually, it concerns with the extension of the ensemble map on the frozen coordinates, and the Langlands dual comes from the algebraic consistency of coordinate expressions of $\mathbb{I}_\A$ and $\mathbb{I}_\X$. See \cref{rem:duality_constraint}.
In the end of the paper ,we also investigate the amalgamation of bracelets bases $\mathbb{I}_\X$. See \cref{thm:amal_bracelet} and \cref{rem:amal_weak}. 

\bigskip

\subsection*{Organization of the paper}
In this section below, we summarize our notation on marked surfaces. 

In \cref{sec:Teich}, we investigate the \Teich\ space with pinnings $\cT^p(\Sigma)$. This section is partially intended to be an introduction to the cluster variety for hyperbolic geometers. Basic definitions on cluster varieties in the surface case are summarized in \cref{app:cluster}. Conversely, those who are familiar with cluster variety may safely skip this section by quickly picking up the algebraic results, such as \cref{thm:A to X,prop:Wilson_lambda}. 

We investigate the lamination space with pinnings $\cL^p(\Sigma,\bQ)$ in \cref{sec:lamination} as a tropical counterpart of the previous section, while most constructions are logically independent.

The contents in \cref{sec:duality} are of cluster algebraic nature. Here we choose to discuss inside the algebra $C^\infty(\cT^p(\Sigma))$ containing $\cO(\X_\Sigma)$ to avoid the problem on the square roots of cluster coordinates. 

\subsection*{Acknowledgements}
The author is grateful to Shunsuke Kano for the insightful discussion on the definition of the lamination space $\cL^p(\Sigma,\bQ)$ and the gluing of $\P$-laminations in several stages of this work. The author also thanks Wataru Yuasa and Hiroaki Karuo for the valuable discussions on the stated skein algebras. 
The author is supported by JSPS KAKENHI Grant Number~JP20K22304.


\subsection*{Marked surfaces}
A marked surface $(\Sigma,\bM)$ is a compact oriented surface $\Sigma$ together with a fixed non-empty finite set $\bM \subset \Sigma$ of \emph{marked points}. 
When the choice of $\bM$ is clear from the context, we simply denote a marked surface by $\Sigma$. 
A marked point is called a \emph{puncture} if it lies in the interior of $\Sigma$, and a \emph{special point} otherwise. 
Let $\bP=\bP(\Sigma)$ (resp. $\bM_\partial=\bM_\partial(\Sigma)$) denote the set of punctures (resp. special points), so that $\bM=\bP \sqcup \bM_\partial$. We say that $\Sigma$ is \emph{unpunctured} if $\bM_\circ=\emptyset$.
Let $\Sigma^*:=\Sigma \setminus \bM$. 
We always assume the following conditions:
\begin{enumerate}
    \item[(S1)] Each boundary component (if exists) has at least one marked point.
    \item[(S2)] $-2\chi(\Sigma^*)+|\bM_\partial| >0$.
    \item[(S3)] $(\Sigma,\bM)$ is not a once-punctured disk with a single special point on the boundary.
\end{enumerate}
We call a connected component of the punctured boundary $\partial^\ast \Sigma:=\partial\Sigma\setminus \bM_\partial$ a \emph{boundary interval}. The set of boundary intervals is denote by $\bB=\bB(\Sigma)$. Note that $|\bB|=|\bM_\partial|$. 
By convention, we endow each boundary interval $\alpha \in \bB$ with the orientation induced from $\partial\Sigma$. Let $m^+_\alpha$ (resp. $m^-_\alpha$) denote its initial (resp. terminal) marked point. 



An \emph{ideal arc} in $(\Sigma,\bM)$ is the isotopy class of an immersed arc in $\Sigma$ with endpoints in $\bM$ having no self-intersections except for its endpoints, and not contractible in $\Sigma^\ast$.
An \emph{ideal triangulation} is a triangulation $\tri$ of $\Sigma$ whose set of $0$-cells (vertices) coincides with $\bM$, $1$-cells (edges) being ideal arcs. 
In this paper, we always consider an ideal triangulation without \emph{self-folded triangles} where two of its sides are identified. 
The conditions (S1)--(S3) ensure the existence of such an ideal triangulation. See, for instance, \cite[Lemma 2.13]{FST}. 

For an ideal triangulation $\tri$, denote the set of edges (resp. interior edges, triangles) by $e(\tri)$ (resp. $e_{\interior}(\tri)$, $t(\tri)$). Since the boundary intervals belong to any ideal triangulation, we always have $e(\tri)=e_{\interior}(\tri) \sqcup \bB$. By a computation on the Euler characteristics, we get
\begin{align*}
    &|e(\tri)|=-3\chi(\Sigma^*)+2|\bM_\partial|, \quad |e_{\interior}(\tri)|=-3\chi(\Sigma^*)+|\bM_\partial|, \\
    &|t(\tri)|=-2\chi(\Sigma^*)+|\bM_\partial|.
\end{align*}
Since the main contribution of this paper is on the structures associated with special points/boundary intervals, we do not discuss much details on those around punctures, such as tagged arcs and tagged triangulations. The interested readers are referred to \cite{FST} and \cite[Section 12]{FG06}. 




\input{2_Teichmuller.tex}

\input{3_lamination.tex}
\input{4_duality.tex}
\input{5_appendix}

\end{document}

%% file: 2_Teichmuller.tex
\section{Teichm\"uller spaces with pinnings}\label{sec:Teich}
In this section, we introduce the \emph{\Teich\ space with pinnings $\pT(\Sigma)$}, which will be identified with the set of positive real points of the moduli space $\P_{PGL_2,\Sigma}$. For the basic terminologies in hyperbolic geometry, we refer the reader to \cite{Penner} and the references therein. See also \cite{FST}.

\subsection{The \Teich\ space $\pT(\Sigma)$ and the cross ratio coordinates}
Let $\bH^2=\{z \in \bC \mid \Im z >0\}$ denote the upper-half plane model of the hyperbolic plane, equipped with the metric $dzd\bar{z}/(\Im z)^2$. The isometry group of $\bH^2$ is isomorphic to the Lie group $PSL_2(\bR)$, which acts on $\bH^2$ by the M\"obius transformations. Another model of the hyperbolic plane is the Poincar\'e disk model $\mathbb{D}^2=\{w \in \bC \mid |w|<1\}$ equipped with the metric $4dwd\bar{w}/(1-|w|^2)^2$. 
We tacitly identify these two models via the Cayley transformation
\begin{align*}
    \bH^2 \xrightarrow{\sim} \mathbb{D}^2, \quad z \mapsto \frac{z-\sqrt{-1}}{z+\sqrt{-1}}.
\end{align*}
Here are basic notions in hyperbolic geometry:
\begin{itemize}
    \item Geodesics in $\mathbb{D}^2$ are euclidean circles/lines perpendicular to the boundary of $\mathbb{D}^2$. The stabilizer of a geodesic is conjugate to $\left\{\begin{bmatrix}\lambda & 0 \\ 0 & \lambda^{-1}\end{bmatrix}\ \middle|\ \lambda\in \bR^\ast \right\}$.
    \item Horocycles in $\mathbb{D}^2$ are euclidean circles tangent to the boundary of $\mathbb{D}^2$. The touching point is called its center. The stabilizer of a horocycle is conjugate to $\left\{\begin{bmatrix}1 & t \\ 0 & 1\end{bmatrix}\ \middle|\ t \in \bR \right\}$.
\end{itemize}

A \emph{decoration} of a geodesic $g$ in $\bH^2$ is a pair $(h_1,h_2)$ of horocycles centered at the two endpoints of $g$. Given such horocycles $(h_1,h_2)$, the geodesic $g$ is uniquely determines as it connects their centers. 

\begin{dfn}[lambda-length]
The \emph{lambda-length} \cite{Penner} of a decorated geodesic $(g;h_1,h_2)$ (or the pair $(h_1,h_2)$) is defined to be $\lambda(h_1,h_2):=\exp (\delta/2) \in \bR_{>0}$, where $\pm\delta$ is the signed hyperbolic length of the segment of $g$ between the horocycles $h_1,h_2$, and the sign is $+$ if and only if the horocycles are disjoint. 
\end{dfn}

\begin{lem}
Given an oriented geodesic $g$ in $\bH^2$, there are  bijections between the following four notions:
\begin{enumerate}
    \item A decoration $(h,h')$ of $g$ with lambda-length $1$.
    \item A horocycle $h$ centered at the initial endpoint of $g$.
    \item A point $x$ on $g$.
    \item An ideal triangle having $g$ as one of its sides, lying on the right of $g$.
\end{enumerate}
\end{lem}

\begin{proof}
The equivalence of the former three notions is obvious: the intersection of the horocycle centered at the initial endpoint and the geodesic $g$ determines a point. Given a point $x \in g$, let $g^\perp$ denote the unique geodesic through $x$ and perpendicular to $g$. Orient $g^\perp$ so that the frame $(T_xg^\perp, T_x g)$ is positive and take the ideal triangle spanned by the terminal endpoint $g_R$ of $g^\perp$ and the two endpoints of $g$. See \cref{fig:pinnings}. Conversely, the point $x$ is uniquely determined as the foot of $g_R$. 
\end{proof}

\begin{figure}
\begin{tikzpicture}[scale=1.2]
\draw(0,0) circle(2cm) coordinate(O);
\draw[->-={0.7}{}](-2,0) node[left]{$g_+$} -- (2,0) node[right]{$g_-$};
\draw(1,0) node[above]{$g$};
\draw[red](-1,0) circle(1cm);
\draw[red](-1,1) node[above]{$h$};
\draw(1,0) circle(1cm);
\draw(1,1) node[above]{$h'$};
\draw(0,0) node[above left=0.3em]{$x$};
\pinn{0,0}{-45}{0.13}{0.035cm};
\draw[myblue,dashed](0,0) --node[midway,above right]{$g^\perp$} (0,-2);
\draw[myblue] ([xshift=-4pt] O) -- ([xshift=-4pt,yshift=-4pt] O) -- ([yshift=-4pt] O);
\draw[myblue](0,-2) node[below]{$g_R$};
\draw[myblue] (-2,0)arc[radius=2cm,start angle=90, end angle=0]; 
\draw[myblue] (2,0)arc[radius=2cm,start angle=90, end angle=180];
\end{tikzpicture}
    \caption{The correspondence between the four notions of pinnings.}
    \label{fig:pinnings}
\end{figure}
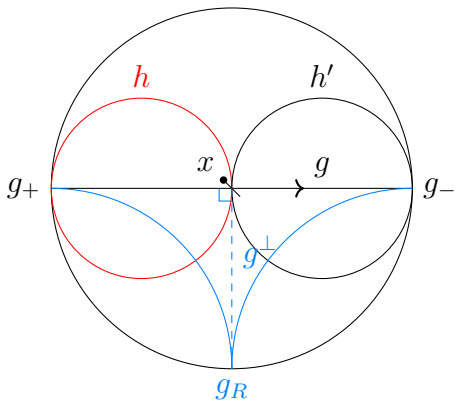

\begin{dfn}
We call one of these equivalent notions a \emph{pinning} over the oriented geodesic $g$. When we speak about a particular one, the notion (1) or (2) is a called a \emph{horocycle pinning}; (3) is called a \emph{point pinning}; (4) is called a \emph{triangle pinning}.
\end{dfn}

\begin{rem}
The equivalence (1) $\Longleftrightarrow$ (2) is an analogue of \cite[Lemma-Definition 3.7]{GS19}. The equivalence (1) $\Longleftrightarrow$ (4) resembles the discussion in \cite[Section 7.1]{GS19}.
\end{rem}

\begin{figure}[ht]
    \centering
\begin{tikzpicture}[scale=0.8]    \draw  (-4.5,0) ellipse (3 and 2);
    \draw (-5.5,-0.5) .. controls (-5.5,-1.35) and (-3.5,-1.35) .. (-3.5,-0.5);
    \draw (-5.4,-0.8) .. controls (-5.4,-0.2) and (-3.6,-0.2) .. (-3.6,-0.8);
    \draw  (-5.5,0.65) ellipse (0.5 and 0.5);
    \node [fill, circle, inner sep=1.3pt] at (-3.5,0.65) {};
    \node [fill, circle, inner sep=1.3pt] at (-5.05,0.85) {};
    \node [fill, circle, inner sep=1.3pt] at (-5.95,0.85) {};
    \node [fill, circle, inner sep=1.3pt] at (-5.5,0.15) {};
    \draw[red] (-5.5,0.15) .. controls (-5.2,-0.35) and (-3.5,0) .. (-3.5,0.65);
    \node [red] at (-4.4,0.25) {$\alpha$};
    \node at (-4.5,-2.5) {$\Sigma$};  
		\draw[->] (-2,2) --node[midway,above]{$f_1$}++ (2,1); 
		\draw[->] (-2,-2) --node[midway,below]{$f_2$}++ (2,-1);
    
\begin{scope}[xshift=-7cm,yshift=3.5cm]    
    \draw  (10.5,0) ellipse (3 and 2);
    \draw [white, ultra thick](8.8,1.65) .. controls (9.1,1.8) and (9.5,1.9) .. (9.8,1.95);
    \draw [white, ultra thick](11.45,1.95) .. controls (11.5,1.9) and (11.8,1.8) .. (11.9,1.7);
    \draw (9.5,-0.5) .. controls (9.5,-1.35) and (11.5,-1.35) .. (11.5,-0.5);
    \draw (9.6,-0.8) .. controls (9.6,-0.2) and (11.4,-0.2) .. (11.4,-0.8);
    \draw (8.35,0.7) .. controls (8.85,0.95) and (8.85,2.3) ..  (8.85,3.15) .. controls (8.85,2.3) and (9.3,1.8) ..node[pos=0.4,inner sep=0](A){} (9.3,2.65) .. controls (9.3,1.8) and (9.75,2.3) ..node[pos=0.6,inner sep=0](B){} (9.75,3.15) .. controls (9.75,2.3) and (8.85,2.3) ..node[pos=0.7,inner sep=0](C){} (8.85,3.15);
		\pinn{A}{45}{0.1}{0.03cm}; 
		\pinn{B}{135}{0.1}{0.03cm}; 
		\pinn{C}{-135}{0.1}{0.03cm};     \draw (10.25,0.7) .. controls (9.75,0.95) and (9.75,2.3) .. (9.75,3.15);
    \draw (10.7,0.7) .. controls (11.2,0.95) and (11.65,1.75) .. (11.65,2.5);
    \draw (12.55,0.7) .. controls (12.05,0.95) and (11.65,1.75) .. (11.65,2.5);
    \draw[dashed] (11.65,1.34) ellipse (0.3 and 0.1);

    \node at (10.5,-2.5) {$X_1$};
    \node [red] at (12,.3) {$f_1(\alpha)$};
    \draw [red](9.3,2.65) .. controls (9.3,0.5) and (10,0.25) .. (10.5,0.25) .. controls (11,0.25) and (11.65,0.5) .. (11.65,1.35) .. controls (11.65,1.7) and (11.65,1.9) .. (11.65,2.5);
\end{scope}
    
\begin{scope}[xshift=-7cm,yshift=-3.5cm]    
    \draw  (10.5,0) ellipse (3 and 2);
    \draw [white, ultra thick](8.8,1.65) .. controls (9.1,1.8) and (9.5,1.9) .. (9.8,1.95);
    \draw [white, ultra thick](11.25,1.95) .. controls (11.5,1.9) and (11.8,1.8) .. (12.05,1.7);
    \draw (9.5,-0.5) .. controls (9.5,-1.35) and (11.5,-1.35) .. (11.5,-0.5);
    \draw (9.6,-0.8) .. controls (9.6,-0.2) and (11.4,-0.2) .. (11.4,-0.8);
    \draw (8.35,0.7) .. controls (8.85,0.95) and (8.85,2.3) ..  (8.85,3.15) .. controls (8.85,2.3) and (9.3,1.8) ..node[pos=0.4,inner sep=0](A){} (9.3,2.65) .. controls (9.3,1.8) and (9.75,2.3) ..node[pos=0.6,inner sep=0](B){} (9.75,3.15) .. controls (9.75,2.3) and (8.85,2.3) ..node[pos=0.7,inner sep=0](C){} (8.85,3.15);
		\pinn{A}{45}{0.1}{0.03cm}; 
		\pinn{B}{135}{0.1}{0.03cm}; 
		\pinn{C}{-135}{0.1}{0.03cm}; 
    \draw (10.25,0.7) .. controls (9.75,0.95) and (9.75,2.3) .. (9.75,3.15);
    \draw (10.7,0.7) .. controls (11.2,0.95) and (11.25,1.95) .. (11.25,2.5);
    \draw (12.55,0.7) .. controls (12.05,0.95) and (12.05,1.95) .. (12.05,2.5);
    \draw(11.65,2.5) ellipse (0.4 and 0.2);

    \node at (10.5,-2.5) {$X_2$};
    \node [red] at (12,.3) {$f_2(\alpha)$};
    \draw [red](9.3,2.65) .. controls (9.3,0.5) and (10,0.25) .. (10.5,0.25) .. controls (11,0.25) and (11.65,0.5) .. (11.65,1.35) .. controls (11.65,1.7) and (11.45,1.9) .. (11.23,2);
    \draw [red](12.05,2.05) .. controls (11.85,1.95) and (11.45,1.95) .. (11.25,2.15);\draw [red](12.05,2.25) .. controls (11.85,2.1) and (11.45,2.1) .. (11.25,2.3);
    \draw [red](12.05,2.4) .. controls (11.85,2.2) and (11.45,2.2) .. (11.25,2.4);
\end{scope}
\end{tikzpicture}
    \caption{Two marked hyperbolic structures with pinnings having different natures at a puncture.}
    \label{fig:cusp_flare}
\end{figure}
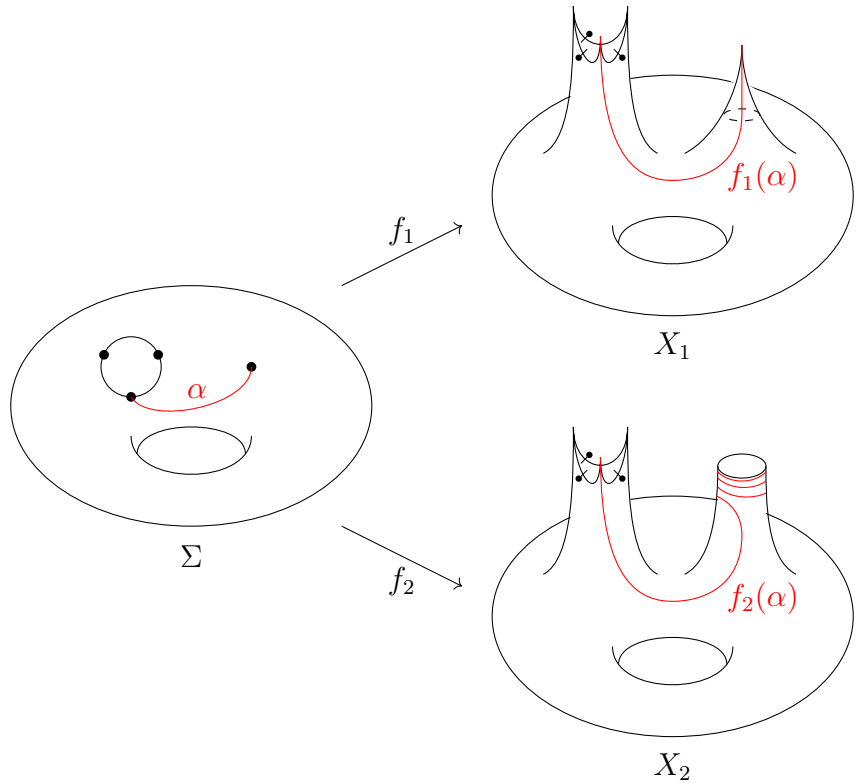

In this paper, a \emph{marked hyperbolic structure} on $\Sigma$ means a pair $(X,f)$, where 
\begin{itemize}
    \item $X$ is a complete hyperbolic surface with finite area and totally geodesic boundary. Let $X^\circ \subset X$ be the complement of the closed geodesic boundary. 
    \item $f: \Sigma^\ast \to X^\circ$ is an orientation-preserving homeomorphism which maps a representative of each ideal arc to a complete geodesic, where each end either enters into a cusp or spirals around a closed geodesic boundary.
\end{itemize}
The hyperbolic surface $X$ can have either cusps or closed geodesic boundary components arising from $m \in \bM_\circ$, and spikes arising from $m \in \bM_\partial$. Boundary intervals give rise to complete geodesics. 
See \cref{fig:cusp_flare}. 
A \emph{hyperbolic structure with pinnings} on $\Sigma$ consists of the following data:
\begin{itemize}
    \item A marked hyperbolic structure $(X,f)$ on $\Sigma$, 
    \item A tuple $p=(p_\alpha)_{\alpha \in \bB}$ of pinnings over the complete geodesics arising from the boundary intervals, oriented positively with respect to $\partial\Sigma$. 
\end{itemize}
Two such data $(X_1,f_1;p_1)$ and $(X_2,f_2;p_2)$ are said to be equivalent if there exists an isometry $h: X_1 \to X_2$ homotopic to $f_2 \circ f_1^{-1}$ relative to the boundary, which sends the pinnings $p_1$ to $p_2$. 

\begin{dfn}\label{dfn:T^p}
The \emph{\Teich\ space with pinnings} $\pT(\Sigma)$ (or the \emph{\Teich\ $\P$-space}) is the set of equivalence classes of hyperbolic structures with pinnings on $\Sigma$. 
\end{dfn}
Forgetting the data of pinnings, we get the \emph{enhanced \Teich\ space} $\mathcal{T}^x(\Sigma)$ (or the \emph{\Teich\ $\X$-space} \cite{FG07}). Let $\pi_\Sigma: \pT(\Sigma) \to \mathcal{T}^x(\Sigma)$ be the natural projection. 

Now we are going to define a coordinate system on $\pT(\Sigma)$ using the cross ratio parameters. 

\begin{dfn}[cross ratio]
For an ideal quadrilateral $Q \subset \bH^2$ with a fixed diagonal $\alpha$, define $r(Q;\alpha) >0$ to be the cross ratio of the four vertices $x_1,x_2,x_3,x_4 \in \partial \bH^2 = \bR \cup \{\infty\}$ of $Q$ in this counter-clockwise order, $x_1$ being one of the endpoints of $\alpha$. Explicitly, we have
\begin{align*}
    r(Q;\alpha) = -\frac{x_1-x_4}{x_3-x_4}\frac{x_3-x_2}{x_1-x_2}.
\end{align*}
Thanks to the cyclic symmetry and the $PSL_2(\bR)$-invariance, it depends only on the isometry class of the quadrilateral $Q$ and its diagonal $\alpha$. 
\end{dfn}

\paragraph{\textbf{Straightening the arcs.}}
Suppose a hyperbolic structure with pinnings $(h,p) \in \pT(\Sigma)$ and an ideal triangulation $\tri$ are given. Here $h=(X,f)$ is a marked hyperbolic structure. For each edge $\alpha \in e(\tri)$, let $\alpha^h:=f(\alpha) \subset X^\circ$ be the corresponding complete geodesic. Similarly, any polygon $P$ in $\tri$ is straightened to a geodesic polygon $P^h \subset X^\circ$. 

Given an ideal triangulation $\tri$ of $\Sigma$, we define a coordinate system $X_\tri=(X_\alpha^\tri)_{\alpha \in e(\tri)}: \pT(\Sigma) \to \bR^\tri_{>0}$ as follows. 
Let $(h,p)$ be a hyperbolic structure with pinnings. 

\begin{itemize}
    \item For an interior edge $\alpha \in e_{\interior}(\tri)$, let $Q$ be the unique quadrilateral of $\tri$ having $\alpha$ as its diagonal. Define 
    \begin{align*}
        X_\alpha^\tri(h,p):=r(\widetilde{Q}^h;\widetilde{\alpha}^h),
    \end{align*}
    where $(\widetilde{Q}^h,\widetilde{\alpha}^h)$ is a lift of $(Q^h,\alpha^h)$.
    \item For a boundary interval $\alpha \in \mathbb{B}$, let $T$ be the unique triangle having $\alpha$ as one of its sides, which necessarily lies on the left of $\alpha$. Choose their lifts $\widetilde{\alpha}^h \subset \widetilde{T}^h$. The datum $p_\alpha$, seen as a triangle pinning over the oriented geodesic $\widetilde{\alpha}^h$, determines a triangle $T(p_\alpha)$ on the right of $\widetilde{\alpha}^h$. Then define 
    \begin{align*}
        X_\alpha^\tri(h,p):=r(\widetilde{T}^h \cup T(p_\alpha);\widetilde{\alpha}^h).
    \end{align*}
\end{itemize}
We call the coordinate system $X_\tri$ the \emph{cross ratio coordinates} associated with $\tri$. The following is essentially due to a combination of Fock--Goncharov \cite{FG07} and Goncharov--Shen \cite{GS19}:

\begin{thm}
For any ideal triangulation $\tri$, the cross ratio coordinate $X_\tri: \pT(\Sigma) \xrightarrow{\sim} \bR^\tri_{>0}$ gives a bijection. For the flip $f_{\kappa}: \tri \to \tri'$ along an interior edge $\kappa \in e_{\interior}(\tri)$, the coordinate transformation $X_{\tri'}\circ X_\tri^{-1}$ is given as shown in \cref{fig:flip}. 
\end{thm}

\begin{proof}
It is known that $X_\tri^{\interior}:=(X_\alpha^\tri)_{\alpha \in e_{\interior}(\tri)}: \xT(\Sigma) \xrightarrow{\sim} \bR^{e_{\interior}(\tri)}_{>0}$ gives a bijection \cite{FG07}. Hence for a given $(h,p) \in \pT(\Sigma)$ with pinnings, the underlying enhanced hyperbolic structure $h \in \xT(\Sigma)$ is determined by the coordinates assigned to the interior edges. 

In order to see that the coordinates assigned to the boundary intervals determine the pinnings, just note that the cross ratio is a complete invariant of a $PSL_2(\bR)$-orbit of four distinct points. In particular, the triangle pinnings are uniquely determined by the underlying hyperbolic structure and the boundary coordinates. The formula for coordinate transformation follows from that for the space $\xT(\Sigma)$ (\cite[Figure 11]{FG07}).
\end{proof}

\begin{figure}[ht]
\[\hspace{1.4cm}
\begin{tikzpicture}[scale=0.8]
\path(0,0) node [fill, circle, inner sep=1.6pt] (x1){};
\path(135:4) node [fill, circle, inner sep=1.6pt] (x2){};
\path(0,4*1.4142) node [fill, circle, inner sep=1.6pt] (x3){};
\path(45:4) node [fill, circle, inner sep=1.6pt] (x4){};
\draw[blue](x1) to node[midway,left,black]{$X_\beta$} (x2) 
to node[midway,left,black]{$X_\alpha$} (x3) 
to node[midway,right,black]{$X_\delta$} (x4) 
to node[midway,right,black]{$X_\gamma$} (x1) 
to node[midway,left,black]{$X_\kappa$} (x3);

\draw[-implies, double distance=2pt](4,2*1.4142) to  (6,2*1.4142);

\begin{scope}[xshift=10cm]
\path(0,0) node [fill, circle, inner sep=1.6pt] (x1){};
\path(135:4) node [fill, circle, inner sep=1.6pt] (x2){};
\path(0,4*1.4142) node [fill, circle, inner sep=1.6pt] (x3){};
\path(45:4) node [fill, circle, inner sep=1.6pt] (x4){};
\draw[blue](x1) to node[midway,left,black]{\scalebox{0.8}{$X_\beta(1+X_\kappa^{-1})^{-1}$}} (x2) 
to node[midway,left,black]{\scalebox{0.8}{$X_\alpha(1+X_\kappa)$}} (x3) 
to node[midway,right,black]{\scalebox{0.8}{$X_\delta(1+X_\kappa^{-1})^{-1}$}} (x4) 
to node[midway,right,black]{\scalebox{0.8}{$X_\gamma(1+X_\kappa)$}} (x1);
\draw[blue] (x2) to node[midway,above,black]{$X_\kappa^{-1}$} (x4);
\end{scope}
\end{tikzpicture}
\]
\caption{The coordinate transformation for the flip along an edge $\kappa$. The formula is the same when some of the surrounding edges are boundary intervals, and still valid when some of the edges are identified as $\alpha=\gamma$ and/or $\beta=\delta$.}
\label{fig:flip}
\end{figure}
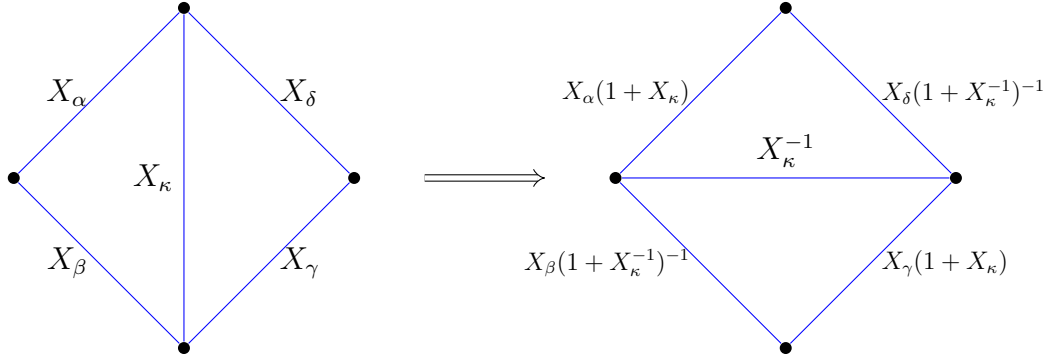
Since the coordinate transformations are real-analytic, we can endow $\pT(\Sigma)$ with a real-analytic structure so that each $X_\tri$ is a real-analytic diffeomorphism. 
Moreover, the formula coincides with the \emph{cluster Poisson transformation} \eqref{eq:X-transf}.
As a consequence, we get:

\begin{cor}\label{cor:Teich_X-variety}
The cross ratio coordinates $X_\tri: \pT(\Sigma) \xrightarrow{\sim} \bR^\tri_{>0}$ associated with ideal triangulations $\tri$ of $\Sigma$ combine to give a canonical $MC(\Sigma)$-equivariant diffeomorphism
\begin{align*}
    \pT(\Sigma) \xrightarrow{\sim} \X_\Sigma(\bR_{>0}).
\end{align*}
\end{cor}
In particular, we have a $MC(\Sigma)$-invariant Poisson bracket $\{-,-\}$ on $C^\infty(\pT(\Sigma))$ such that
\begin{align*}
    \{X_\alpha^\tri,X_\beta^\tri\} = \ve_{\alpha\beta}^\tri X_\alpha^\tri X_\beta^\tri
\end{align*}
for any ideal triangulation $\tri$. Here $(\ve_{\alpha\beta}^\tri)_{\alpha,\beta \in e(\tri)}$ denotes the \emph{exchange matrix} (see \cref{app:cluster}). 

\begin{rem}
It is straightforward to extend the construction of coordinates for a \emph{tagged triangulation}. See, for instance, \cite[Section 9]{AB}. 
\end{rem}

\subsection{Gluing map}\label{subsec:Teich_amalgamation}
We are going to discuss a map between the \Teich\ spaces with pinnings, called the \emph{gluing map}. Let $\Sigma$ be a marked surface (possibly disconnected), and choose distinct boundary intervals $\alpha_L,\alpha_R \in \mathbb{B}$. Let $\Sigma'$ be the marked surface obtained from $\Sigma$ by gluing the edges $\alpha_L$ and $\alpha_R$ together. We define a map $q_{\Sigma,\Sigma'}:\pT(\Sigma) \to \pT(\Sigma')$ as follows. 

Let $(h,p)$ be a hyperbolic structure with pinnings on $\Sigma$. Then the data $p_{\alpha_L}$ and $p_{\alpha_R}$, seen as point pinnings, determine a point on each of the boundary geodesics $\alpha_L^h$ and $\alpha_R^h$. Gluing the hyperbolic surface $(\Sigma,h)$ along these edges so that these points match, we get a new hyperbolic surface $(\Sigma',h')$. 
Since the remaining pinnings naturally induces pinnings over $h'$, we get a pair $(h',p')=q_{\Sigma,\Sigma'}(h,p)$ on $\Sigma'$. 
The resulting map
\begin{align}\label{eq:gluing_Teich}
    q_{\Sigma,\Sigma'}:\pT(\Sigma) \to \pT(\Sigma')
\end{align}
is called the \emph{gluing map}.

\begin{ex}
Let us illustrate the construction in a simple example. Let $\Sigma_L$ (resp. $\Sigma_R$) be an $n_L$-gon (resp. $n_R$-gon), \emph{i.e.}, a disk with $n_L$  (resp. $n_R$) special points, and consider the marked surface $\Sigma:=\Sigma_L \sqcup \Sigma_R$. For $Z \in \{L,R\}$, choose a side $\alpha_Z$ of the polygon $\Sigma_Z$ and glue them together. The resulting surface $\Sigma'$ is an $(n_L+n_R-2)$-gon. 

Given a hyperbolic structure with pinnings $(h,p) \in \pT(\Sigma)$, each polygon $\Sigma_Z$ is realized as an ideal polygon $\widetilde{\Pi}^h_Z \subset \bH^2$. The geodesic lift $\widetilde{\alpha}_Z^h \subset \widetilde{\Pi}^h_Z$ of $\alpha_Z$ is equipped with a point pinning given by the data $p_{\alpha_Z}$. Then there exists a unique hyperbolic isometry $g \in PSL_2(\bR)$ which maps the geodesic $\widetilde{\alpha}_R^h$ to $\widetilde{\alpha}_L^h$ and matches the point pinnings. The resulting polygon $\widetilde{\Pi}^h_L \cup g(\widetilde{\Pi}^h_R)$ gives the hyperbolic structure $h'$ on $\Sigma'$, together with a pinning determined by $p \setminus \{p_{\alpha_L},p_{\alpha_R}\}$.
\end{ex}

\begin{figure}[ht]
\begin{tikzpicture}[scale=1]
\begin{scope}
\draw(0,0) circle(2cm);
\clip(0,0) circle(2cm);
\draw(-2,0) -- (2,0);
\draw[red](-1.5,0) circle(0.5cm);
\pinn{-1,0}{-135}{0.13}{0.03cm}
\draw[red,thick,->,>=latex](-1,0) -- (0,0);
\draw[red,dashed] (-0.5,0) -- (0,-1) node[fill=white,inner sep=2pt,scale=0.9]{$\log X_{\alpha_L}^\tri(h,p)$};
{\color{myblue}
\draw \angleBL{-1,0};
\draw[dashed] (-1,0)arc[radius=1cm,start angle=0, end angle=-75]; 
\hgline{180}{210}{2} 
\hgline{0}{210}{2}
}
\hgline{90}{0}{2}
\hgline{90}{180}{2}
\draw[dashed] (0,2) -- (0,0);
\draw \angleAL{0,0};
\draw(-0.8,1) node[scale=0.9]{$\widetilde{T}_1^h$};
\end{scope}


\begin{scope}[yshift=-4.5cm]
\draw(0,0) circle(2cm);
\clip(0,0) circle(2cm);
\draw(-2,0) -- (2,0);
\draw[red](1,0) circle(1cm);
\pinn{0,0}{-135}{0.13}{0.03cm}
\draw[red,thick,->,>=latex](-1,0) -- (0,0);
\draw[red,dashed] (-0.5,0) -- (0,-1.5) node[fill=white,inner sep=2pt,scale=0.9]{$\log X_{\alpha_R}^\tri(h,p)$};
\draw\angleBL{-1,0};
\draw[dashed] (-1,0)arc[radius=1cm,start angle=0, end angle=-75]; 
\hgline{180}{210}{2} 
\hgline{0}{210}{2}
{\color{myblue}
\hgline{90}{0}{2}
\hgline{90}{180}{2}
\draw[dashed] (0,2) -- (0,0);
\draw \angleAL{0,0};
}
\draw(0.8,-0.5) node[scale=0.9]{$\widetilde{T}_2^h$};
\end{scope}

\snake{2.5,-2.25}{4.5,-2.25} node[midway,above=0.2em]{Glue};

\begin{scope}[xshift=7.5cm,yshift=-2.25cm]
\draw(0,0) circle(2cm);
\clip(0,0) circle(2cm);
\draw(-2,0) -- (2,0);
\pinn{-1,0}{-135}{0.13}{0.03cm}
\draw[red,thick,->,>=latex](-1.5,0) -- (0,0);
\draw[dashed] (-1.5,0)arc[radius=0.5cm,start angle=0, end angle=-75]; 
\hgline{180}{195}{2} 
\hgline{0}{195}{2}

\hgline{90}{0}{2}
\hgline{90}{180}{2}
\draw[dashed] (0,2) -- (0,0);
\draw \angleAL{0,0};
\draw(-0.8,1) node[scale=0.9]{$\widetilde{Q}_{h'}$};
\end{scope}

\begin{scope}[xshift=7.5cm,yshift=-2.25cm]
\draw[red,dashed] (-0.8,0) -- (0,-2.5) node[fill=white,inner sep=2pt,scale=0.9]{$\log X_{\alpha_L}^\tri(h,p)+ \log X_{\alpha_R}^\tri(h,p)$};
\draw \angleBL{-1.5,0};
\end{scope}
\end{tikzpicture}
    \caption{Gluing of two hyperbolic triangles. Here triangle pinnings are shown in blue. The second triangle $\widetilde{T}_2^h$ is mapped by a hyperbolic isometry so that the pointed geodesic $\widetilde{\alpha}_2^h$ is matched with $\widetilde{\alpha}_1^h$ to form an ideal quadrilateral $\widetilde{Q}^{h'}$.}
    \label{fig:gluing_triangle}
\end{figure}
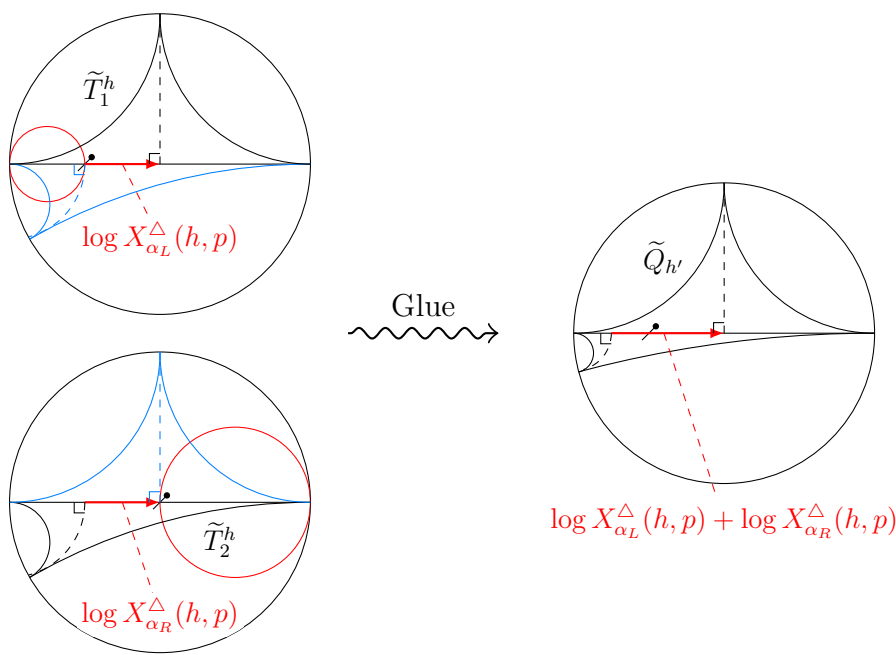

Note that an ideal triangulation $\tri$ on $\Sigma$ naturally induces an ideal triangulation $\tri'$ on $\Sigma'$. Let $\overline{\alpha} \in e_{\interior}(\tri')$ be the interior edge arising from $\alpha_L$ and $\alpha_R$. 

\begin{prop}\label{prop:amalgamation}
We have
\begin{align*}
    q_{\Sigma,\Sigma'}^*X_{\overline{\alpha}}^{\tri'} = X_{\alpha_L}^\tri \cdot X_{\alpha_R}^\tri.
\end{align*}
\end{prop}
In other words, the map $q_{\Sigma,\Sigma'}$ agrees with the cluster amalgamation map $\P_{PGL_2,\Sigma} \to \P_{PGL_2,\Sigma'}$ ( \cite[Definition 2.1]{FG06}).

In the proof, we use another characterization of the cross ratio. Let $Q$ be an ideal quadrilateral with a diagonal $\alpha$ in $\bH^2$, and $x_1,x_2,x_3,x_4$ its vertices in this clockwise order, $x_1$ being one of the endpoints of $\alpha$. Let $g_2$ (resp. $g_4$) be the oriented geodesic perpendicular to $\alpha$ emanating from the vertex $x_2$ (resp. $x_4$). Then the cross ratio $r_{Q;\alpha}$ coincides with the exponential of the signed hyperbolic length of the segment of $\alpha$ bounded by the geodesics $g_2$ and $g_4$, where the sign is positive when one geodesic is seen on the right on the other (\cite[Chapter 1, Corollary 4.14 (c)]{Penner}). 

\begin{proof}
Consider $(h,p) \in \pT(\Sigma)$ and $(h',p'):=q_{\Sigma,\Sigma'}(h,p)$. 
For $Z \in \{L,R\}$, consider a lift $\widetilde{T}_Z^h \subset \bH^2$ of the triangle in $\tri$ having $\alpha_Z$ as one of its sides. In the universal cover of $(\Sigma',h')$, these triangles are glued together and forms a quadrilateral $\widetilde{Q}^{h'}:=\widetilde{T}_L^h \cup g(\widetilde{T}_R^h)$ by using some isometry $g\in PSL_2(\bR)$. See \cref{fig:gluing_triangle}. 
From the definition of the coordinate assigned to the boundary interval $\alpha_Z$, it is given by the exponential of the signed hyperbolic between the point pinning $p_{\alpha_Z} \in \widetilde{\alpha}_Z^h$ and the perpendicular geodesic from the vertex of $\widetilde{T}_Z^h$ other than the endpoints of $\widetilde{\alpha}_Z^h$. 
Then we see that the coordinate $\log X_{\overline{\alpha}}^{\tri'}(h',p')$ coincides with the sum $\log X_{\alpha_L}^\tri(h,p) + \log X_{\alpha_R}^\tri(h,p)$, from which we get the desired assertion.
\end{proof}

\paragraph{\textbf{Relation to the Thurston's completeness criterion.}}
If $\alpha_L,\alpha_R$ are consecutive boundary intervals (say, $\alpha_R$ follows $\alpha_L$ along the boundary orientation), then we get a new puncture $m:=m^-_{\alpha_L}=m^+_{\alpha_R}$ in $\Sigma'$ arising from their common marked point. 
Let us investigate what happens here. 

Let $(h,p) \in \cT^p(\Sigma)$ be a hyperbolic structure with pinnings. Recall that the datum $p_{\alpha_R}$ gives a horocycle pinning, which is a horocycle $C_R$ centered at $m$. Let $\widetilde{C}_0^h \subset \widetilde{\Sigma}^h$ be its lift, a horocyclic arc in the universal cover. 
Similarly, the datum $p_{\alpha_L}$ gives a horocycle pinning $C_L$ centered at $m^+_\alpha$, and a point pinning $x \in \widetilde{\alpha}^h_L$. 

Suppose first that the lambda length $\lambda(C_L,C_R) >1$. In particular, the horocyclic arc $\widetilde{C}_0^h$ does not pass through the point $x$ again. Extending $\widetilde{C}_0^h$ as a horocyclic arc over the geodesic $\widetilde{\alpha}^h_L$, we get a new horocyclic arc $\widetilde{C}_1^h$ in the universal cover of $\Sigma'$, whose projection to $\Sigma'$ is ``closer'' than that of $\widetilde{C}_0^h$ to the puncture $m$. See \cref{fig:horocycle_spiral}. Continuing in this manner, we get an infinite horocyclic arc $\widetilde{C}_\infty^h:=\bigcup_{i=0}^\infty \widetilde{C}_i^h$ ``spiralling'' around $m$, where the hyperbolic distance between the consecutive segments $\widetilde{C}_i^h$ and $\widetilde{C}_{i+1}^h$ is given by the constant $2\log \lambda(C_L,C_R)$. This is exactly the situation discussed in the Thurston's completeness criterion \cite[Proposition 3.4.8]{Thu}, the constant $2\log \lambda(C_L,C_R)$ being the \emph{invariant} $d(v)$. 

In particular, the resulting hyperbolic surface is not complete, and its metric completion gets a new closed geodesic boundary. Indeed, the intersection points between any ray $\ell$ from the ideal vertex $m$ and $\widetilde{C}_\infty^h$ constitute a non-convergent Cauchy sequence. Such sequences give rise to a new $S^1$ after the completion.
The geodesics entering the spike $m$ become spiralling geodesics around the new closed geodesic. 

The case $\lambda(C_L,C_R) <1$ is similar, where the spiralling direction is reversed. In the case $\lambda(C_L,C_R) =1$, the arc $C_R$ is glued into a horocycle $\overline{C}$ around $m$, and the resulting hyperbolic structure is complete around the cusp given by $m$ equipped with the decoration $\overline{C}$. 

\begin{figure}[ht]
    \centering
\begin{tikzpicture}
\fill[gray!20] (3,0) -- (-3,0) -- (-3,-0.2) -- (3,-0.2) --cycle;
\draw[myorange,thick] (-1.3,0) arc (0:180:0.7) node[midway,above,scale=0.9]{$C_L$};
\pinn{-1.3,0}{-135}{0.13}{0.03cm}
\node[scale=0.9] at (-1.2,0.3){$x$};
\draw[red,thick] (-0.9,0) arc (180:0:0.9) node[midway,above,scale=0.9]{$C_R$};
\pinn{0.9,0}{-135}{0.13}{0.03cm}
\draw[red,thick,dashed,->-] (-0.9,0) arc (180:360:0.7);
\draw[thick] (3,0) -- (-3,0);
\draw[dashed,<-,>=latex] (-1.3,-0.1) --++(0,-1) coordinate(A);
\draw[dashed,<-,>=latex] (0.9,-0.1) --++(0,-1) coordinate(B);
\draw[dashed] (A) --node[midway,below,scale=0.9]{glued} (B);
\foreach \i in {0,2,-2} \fill(\i,0) circle(1.5pt);
\node[scale=0.9] at (0,0.3) {$m$};
\snake{3.5,0}{5,0} node[midway,above=0.2em]{Glue};

\begin{scope}[xshift=8.5cm,yshift=-1cm]
\fill[gray!20] (1.5,0) -- (-1.5,0) -- (-1.5,-0.2) -- (1.5,-0.2) --cycle;
\draw[thick] (1.5,0) -- (-1.5,0);
\draw[myorange,thick] (-0.7,0) arc(180:90:0.7) coordinate(X);
\pinn{X}{-135}{0.13}{0.03cm}
\node[scale=0.9] at (0.2,0.5) {$x$};
\draw[red,thick,name path=C1] (X) arc(-90:90:1.1) arc(90:270:0.9) coordinate(Y);
\draw[red,thick,name path=C2] (Y) arc(-90:30:0.7);
\draw (0,0) -- (0,2);
\draw[blue,<->] ($(X)+(-0.1,0)$) coordinate(X') -- ($(Y)+(-0.1,0)$) coordinate(Y');
\draw[blue,dashed] ($(X')!0.5!(Y')$) --++(-2,-0)--++(0,-0.3) node[below,scale=0.9]{$2\log\lambda(C_L,C_R)$};
\draw[red,dashed,name path=ray] (0,2) --++(-45:2) node[above]{$\ell$};
\draw[name intersections={of=C1 and ray,by=F1}];
\draw[name intersections={of=C2 and ray,by=F2}];
\fill[red] (F1) circle(1.5pt);
\fill[red] (F2) circle(1.5pt);
\draw(0,2) ++(45:1.4) node[red]{$\widetilde{C}_\infty^h$};
\fill(0,0) circle(1.5pt);
\fill(0,2) circle(1.5pt);
\end{scope}
\end{tikzpicture}
    \caption{Topological picture of the gluing that produces a new puncture in the case $\lambda(C_L,C_R) >1$. }
    \label{fig:horocycle_spiral}
\end{figure}
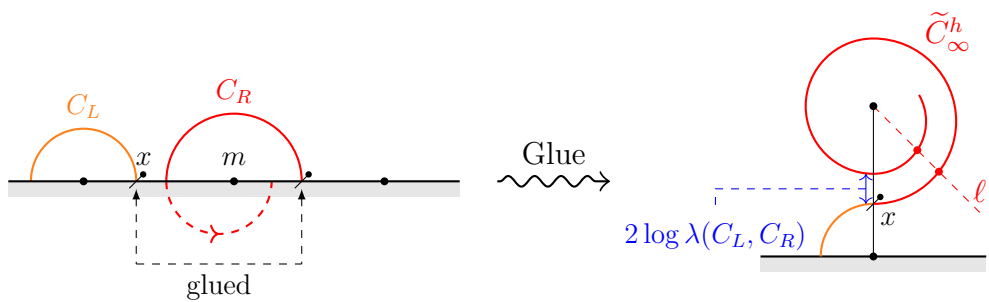

\subsection{Ensemble map}

Recall the \emph{decorated \Teich\ space} introduced by Penner \cite{Penner}. Let $h$ be a marked hyperbolic structure having no closed geodesic boundary. In other words, the monodromy around each $m \in \bM_\circ$ is assumed to be parabolic (unipotent). 
In the universal cover, each marked point gives rise to a $\pi_1(\Sigma)$-invariant collection of spikes. A \emph{decoration} of $h$ is a $\pi_1(\Sigma)$-equivariant collection $d$ of horocycles centered at these points. We call the pair $(h,d)$ a \emph{decorated hyperbolic structure}. An equivalence of decorated hyperbolic structures is similarly defined as in the case of hyperbolic structures with pinnings.  

\begin{dfn}[\cite{Penner}]
The \emph{decorated \Teich\ space} (or the \emph{\Teich\ $\A$-space}) $\aT(\Sigma)$ of $\Sigma$ is the set of equivalence classes of the decorated hyperbolic structures on $\Sigma$.
\end{dfn}
Since each geodesic lift on an ideal arc $\alpha$ in $\Sigma$ is equipped with a decoration, we have a \emph{lambda-length function} $A_\alpha:\aT(\Sigma) \to \bR_{>0}$ associated to $\alpha$. Given an ideal triangulation $\tri$, the collection of lambda-length functions gives a real-analytic coordinate system
\begin{align*}
    A_\tri:=(A_\alpha)_{\alpha \in e(\tri)}: \aT(\Sigma) \xrightarrow{\sim} \bR_{>0}^\tri,
\end{align*}
and they combine to give an $MC(\Sigma)$-equivariant diffeomorphism $\aT(\Sigma) \xrightarrow{\sim} \A_{SL_2,\Sigma}(\pos)$. See \cite[Chapter 2]{Penner} for a detail.  

Now we are going to study a relation between the \Teich\ spaces $\aT(\Sigma)$ and $\pT(\Sigma)$. 
We define the \emph{ensemble map} $p_\Sigma: \aT(\Sigma) \to \pT(\Sigma)$ as follows. Let $(h,d) \in \aT(\Sigma)$ be a decorated hyperbolic structure. 
For each boundary interval $\alpha \in \mathbb{B}$, the decoration $d$ gives a horocycle on each endpoint of $\alpha$. We adopt the one assigned to the initial marked point $m^+_\alpha$ as the horocycle pinning over $\alpha$ (cf.~\cite[Section 12.2]{GS19}). Thus we get a hyperbolic structure with pinnings $(h,p)=p_\Sigma(h,d) \in \pT(\Sigma)$.

When $\Sigma$ has a puncture, the ensemble map is neither injective nor surjective, since it forgets the decorations on punctures and only produces marked hyperbolic structures without closed geodesic boundary.

\begin{prop}\label{prop:ensemble}
If $\Sigma$ is unpunctured, then the ensemble map $p_\Sigma: \aT(\Sigma) \xrightarrow{\sim} \pT(\Sigma)$ is a $C^\omega$-diffeomorphism. For any ideal triangulation $\tri$ of $\Sigma$ and an edge $\kappa \in e(\tri)$, the pull-back $p_\Sigma^*X_\kappa^\tri$ is given as follows.
\begin{enumerate}
    \item If $\kappa \in e_{\interior}(\tri)$, then 
    \begin{align*}
        p_\Sigma^*X_\kappa^\tri = \frac{A_{\alpha} A_{\gamma}}{A_{\beta} A_{\delta}},
    \end{align*}
    where the edges around $\kappa$ is labeled in the same way as in \cref{fig:flip}. 
    \item If $\kappa \in \mathbb{B}$, then 
    \begin{align*}
        p_\Sigma^*X_\kappa^\tri = \frac{A_{\beta}}{A_{\kappa} A_{\alpha}},
    \end{align*}
    where we relabel the edges sharing a triangle with $\kappa$ by $\alpha,\beta$ as in \cref{fig:ensemble_boundary}.
\end{enumerate}
\end{prop}

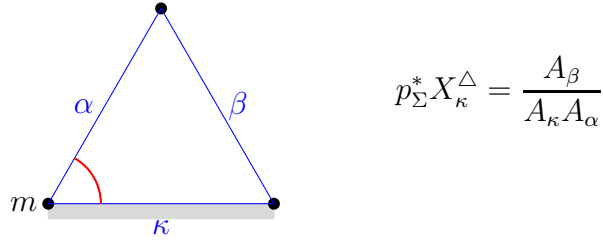
\begin{figure}
\begin{tikzpicture}
\draw[fill, gray!30]  (0,-0.2) rectangle (3,0);
\path(0,0) node [fill, circle, inner sep=1.6pt] {} node[left]{$m$};
\path(3,0) node [fill, circle, inner sep=1.6pt] {};
\path(60:3) node [fill, circle, inner sep=1.6pt] {};
\draw[blue](0,0) --node[midway,below=0.2em]{$\kappa$} (3,0) --node[midway,right]{$\beta$} (60:3) --node[midway,left]{$\alpha$} cycle;
\draw[red,thick] (0.7,0) arc[radius=0.7cm,start angle=0, end angle=60];
\node at (6,1.5) {$\displaystyle p_\Sigma^*X_{\kappa}^\tri = \frac{A_{\beta}}{A_{\kappa} A_{\alpha}}$};
\end{tikzpicture}
    \caption{The pull-back action of the ensemble map on a boundary coordinate. Here $\kappa$ is a boundary interval. 
    }
    \label{fig:ensemble_boundary}
\end{figure}

For the proof, the \emph{hyperboloid model} of the hyperbolic plane is useful. Let us briefly recall it. For a detail, see \cite[Chapter 1]{Penner}.

Let $\bR^{2,1}:=\bR^3$ be the Minkowski space of signature $(2,1)$, having the inner product $\langle x,x'\rangle:=x_1x'_1+x_2x'_2-x_3x'_3$. 
We endow the hyperboloid $\cH:=\{x\in \bR^{2,1} \mid \langle x,x\rangle=-1,~x_3>0\}$ with the induced metric, which turns out to be isometric to $\bH^2$. In this model, 
\begin{description}
    \item[Geodesic] given by the intersection of $\cH$ and a timelike plane, which is the orthogonal complement $n^\perp$ of a spacelike vector $n$ ($\langle n,n\rangle >0$). We will denote such a geodesic simply by $n^\perp$. Two geodesics $n^\perp,(n')^\perp$ are perpendicular to each other if and only if $\langle n,n'\rangle=0$. The distance between a point $x$ and a geodesic $n^\perp$ satisfies $\sinh d(x,n^\perp)=\langle x,n \rangle$. 
    \item[Horocycle] has the form
    \begin{align*}
        h(u):=\left\{ v \in \cH ~\middle|~ \langle u,v\rangle = -\frac{1}{\sqrt{2}} \right\}
    \end{align*}
    for a lightlike vector $u=(u_1,u_2,u_3)$ ($\langle u,u\rangle=0$) with $u_3 >0$. 
    We will identify the horocycle $h(u)$ and the vector $u$. The lambda-length is given by $\lambda(u,u')=\sqrt{-\langle u,u'\rangle}$. 
\end{description}

\begin{proof}
The formula for the first case is well-known. See \cite[Chapter 1, Corollary 4.14 (b)]{Penner}. For the second case, it suffices to think about a hyperbolic triangle equipped with a horocycle at each vertex. 
It is convenient to use the light-cone basis
\begin{align*}
    u=\frac{1}{\sqrt{2}}(-1,0,1), \quad 
    v=\frac{1}{\sqrt{2}}(1,0,1), \quad 
    w=\sqrt{2}(0,-1,1),
\end{align*}
which satisfies $\langle u,v\rangle=\langle v,w\rangle=\langle w,u\rangle=-1$. The lambda-lengths between the rescaled horocycles $u':=(A_{\beta}A_{\kappa}/A_{\alpha})u$, $v':=(A_{\alpha}A_{\kappa}/A_{\beta})v$, $w':=(A_{\beta}A_{\alpha}/A_{\kappa})w$ are given as follows:
\begin{align*}
    \lambda(u',v') = A_{\kappa},\quad \lambda(v',w') = A_{\alpha},\quad \lambda(w',u') = A_{\beta}.
\end{align*}
Now let us consider the hyperbolic triangle $(\bar{u},\bar{v},\bar{w})$ spanned by the centers of the horocycles $u',v',w'$. Let $x$ be the intersection point of the geodesic $\gamma:=[\bar{u},\bar{v}]$ and the horocycle $h(v')$, and $y$ the foot of the geodesic $\delta$ from $\bar{w}$ perpendicular to $\gamma$. See \cref{fig:ensemble_computation}. 
Our aim is to compute the distance between $x$ and $y$ in terms of the lambda-lengths $A_{\kappa},A_{\alpha},A_{\beta}$. 

Since $\gamma$ is clearly the intersection of $\cH$ and the $uv$-plane, its orthonormal vector is $u+v-w$. Then the point $x$ can be computed as a solution of two linear equations (defining $\gamma$ and $h(u')$) and a quadratic equation (defining $\cH$), which is given by
\begin{align*}
    x = \frac{1}{\sqrt{2}}\frac{A_{\beta}}{A_{\kappa}A_{\alpha}}u + \frac{1}{\sqrt{2}}\frac{A_{\kappa}A_{\alpha}}{A_{\beta}}v.
\end{align*}
A short computation also shows that the unit normal vector of $\delta$ is $n=1/\sqrt{2}(u-v)$. Therefore the distance $d(x,y)$ is computed as
\begin{align*}
    \sinh d(x,y) = \sinh d(x,n^\perp) = \langle x,n\rangle = \frac{1}{2}\left( \frac{A_{\beta}}{A_{\kappa}A_{\alpha}} - \frac{A_{\kappa}A_{\alpha}}{A_{\beta}} \right).
\end{align*}
Thus we get $d(x,y) = A_{\beta}/(A_{\kappa}A_{\alpha})$ as desired. 
\end{proof}

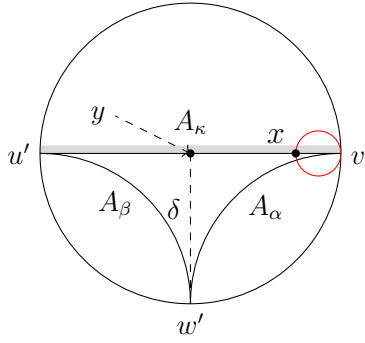
\begin{figure}
\begin{tikzpicture}[scale=1]
\begin{scope}
\draw[fill, gray!30]  (-2,0.1) rectangle (2,0);
\draw(0,0) circle(2cm);
\clip(0,0) circle(2cm);
\hgline{0}{180}{2}
\hgline{0}{270}{2}
\hgline{270}{180}{2}
\draw(0,0) node[above=0.3em,scale=0.9]{$A_{\kappa}$};
\draw(1,-1) node[above,scale=0.9]{$A_{\alpha}$};
\draw(-1,-1) node[above,scale=0.9]{$A_{\beta}$};
\draw[red] (1.7,0) circle(0.3cm);
\path(1.4,0) node [fill, circle, inner sep=1.1pt] {} node[above left]{$x$};
\draw[dashed] (0,-2) --node[midway,above left,scale=0.9]{$\delta$} (0,0);
\path(0,0) node [fill, circle, inner sep=1.1pt] {};
\draw[dashed,<-] (0,0) -- (-1,0.5) node[left,scale=0.9]{$y$};
\end{scope}
\draw(-2,0) node[left,scale=0.9]{$u'$};
\draw(2,0) node[right,scale=0.9]{$v'$};
\draw(0,-2) node[below,scale=0.9]{$w'$};
\end{tikzpicture}
    \caption{Computation of the ensemble map.}
    \label{fig:ensemble_computation}
\end{figure}

We set $m_{\alpha\beta}:=-\delta_{\alpha\beta}$ if $\alpha,\beta \in \bB$, and otherwise $m_{\alpha\beta}:=0$. Then the two formulae in \cref{prop:ensemble} are combined into
\begin{align}\label{eq:ensemble_combined}
    p_\Sigma^\ast X_\kappa^\tri = \prod_{\alpha \in e(\tri)} A_\alpha^{\ve_{\kappa\alpha}^\tri+m_{\kappa\alpha}},
\end{align}
where recall the exchange matrix $\ve^\tri=(\ve_{\alpha\beta}^\tri)$ from \cref{app:cluster}. 
This agrees with the formula given in \cite[Proposition 12.4]{GS19} for the $A_1$ case. 

\begin{rem}The right-hand side of the formula in \cref{prop:ensemble} (2) coincides with the \emph{$h$-length} of the decoration assigned to $m^+_\kappa$ \cite[Chapter 1, Lemma 4.7]{Penner}.     
\end{rem}

\subsection{The $\lambda$-lengths in terms of the cross ratios}
Assuming that $\Sigma$ is unpunctured, we are going to give the inverse formula to \eqref{eq:ensemble_combined}. In this case, we can identify the two \Teich\ spaces $\cT^a(\Sigma)$ and $\cT^p(\Sigma)$ via the ensemble map $p_\Sigma$. Thus we omit the symbol $p_\Sigma^\ast$ in the following. 

\begin{dfn}[positive $\bB$-shift of ideal arcs]\label{def:shift_ideal}
Given an ideal arc $\alpha$ on $\Sigma$, we define its \emph{(positive) $\bB$-shift} to be the simple curve $\alpha_
{\bB}$ having its endpoints on $\partial^\ast \Sigma$ obtained from $\alpha$ by shifting its endpoints to the next boundary interval in the positive direction along $\partial\Sigma$. See \cref{fig:shifting}. 
\end{dfn}

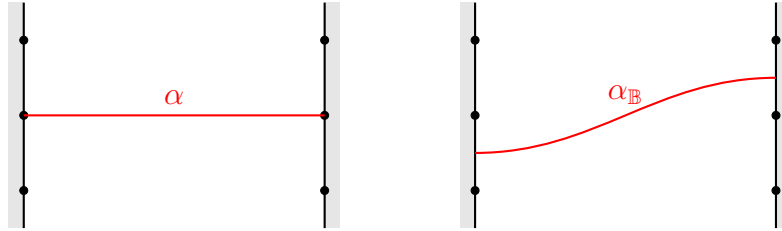
\begin{figure}[ht]
    \centering
\begin{tikzpicture}
\begin{scope}[xshift=0cm]
\fill[gray!20] (0,1.5) -- (-0.2,1.5) -- (-0.2,-1.5) -- (0,-1.5) --cycle;
\fill[gray!20] (4,1.5) -- (4+0.2,1.5) -- (4+0.2,-1.5) -- (4,-1.5) --cycle;
\draw[thick] (0,1.5) -- (0,-1.5);
\draw[thick] (4,-1.5) -- (4,1.5);
\filldraw(0,1) circle(1.5pt); 
\filldraw(0,0) circle(1.5pt);
\filldraw(0,-1) circle(1.5pt);
\filldraw(4,1) circle(1.5pt);
\filldraw(4,0) circle(1.5pt);
\filldraw(4,-1) circle(1.5pt);
\draw[red,thick] (0,0) to[out=0,in=180] node[midway,above]{$\alpha$} (4,0);
\end{scope}
\begin{scope}[xshift=6cm]
\fill[gray!20] (0,1.5) -- (-0.2,1.5) -- (-0.2,-1.5) -- (0,-1.5) --cycle;
\fill[gray!20] (4,1.5) -- (4+0.2,1.5) -- (4+0.2,-1.5) -- (4,-1.5) --cycle;
\draw[thick] (0,1.5) -- (0,-1.5);
\draw[thick] (4,-1.5) -- (4,1.5);
\filldraw(0,1) circle(1.5pt); 
\filldraw(0,0) circle(1.5pt);
\filldraw(0,-1) circle(1.5pt);
\filldraw(4,1) circle(1.5pt);
\filldraw(4,0) circle(1.5pt);
\filldraw(4,-1) circle(1.5pt);
\draw[red,thick] (0,-0.5) to[out=0,in=180] node[midway,above]{$\alpha_\bB$} (4,0.5);
\end{scope}
\end{tikzpicture}
    \caption{The positive $\bB$-shift of an ideal arc.}
    \label{fig:shifting}
\end{figure}

\begin{thm}\label{thm:A to X}
Assume that $\Sigma$ is unpunctured. 
Then for each edge $\alpha \in e(\tri)$ of an ideal triangulation, we have the inverse formula to \eqref{eq:ensemble_combined}:
\begin{align*}
    A_\alpha = \prod_{\beta \in e(\tri)} (X^\tri_{\beta})^{q_{\alpha\beta}}.
\end{align*} 
Here $q_{\alpha\beta}:=-\sfa_{\beta}(\alpha_\bB)$, and $\sfa_{\beta}(\alpha_\bB) \in \frac{1}{2}\bZ_{\geq 0}$ denotes half the geometric intersection number between the two curves $\alpha_\bB$ and $\beta$.
\end{thm}

\begin{proof}
Let us write 
\begin{align*}
    n_{\alpha\beta}:=\sum_{\gamma \in e(\tri)} q_{\alpha\gamma}p^\tri_{\gamma\beta}
\end{align*}
for $\alpha,\beta \in e(\tri)$. Then it suffices to prove the equation $n_{\alpha\beta}=\delta_{\alpha\beta}$ for all $\alpha,\beta \in e(\tri)$. Fix an edge $\alpha \in e(\tri)$, and give $\alpha_\bB$ an arbitrary orientation. Let $\alpha_0,\dots,\alpha_m$ be the edges of $\tri$ traversed by $\alpha_\bB$ in this order, where $\alpha_0,\alpha_m \in \bB$ are end-intervals of $\alpha$, and the other $\alpha_i$ are interior edges. 
Then we get $n_{\alpha\beta}=-1/2\sum_{i=0}^m p^\tri_{\alpha_i,\beta}$, with a notice that we allow $\alpha_i=\alpha_j$ for some $i\neq j$. 

First consider the case where $\alpha$ is an interior edge. Then $\alpha$ is the diagonal of a unique quadrilateral $Q_\alpha$ in $\tri$. 
There is a unique $0 \leq i_0 \leq m$ such that $\alpha_{i_0}=\alpha$ and $\alpha_{i_0\pm 1}$ are the opposite sides of $Q_\alpha$. See the left picture in \cref{fig:shift_interior}. 
Then one can verify the equations
\begin{align*}
    n_{\alpha,\alpha_i} = \begin{cases}
        -\frac{1}{2}(p^\tri_{\alpha_0,\alpha_0}+p^\tri_{\alpha_1,\alpha_0}) =-\frac{1}{2}(-1+1) = 0 & \mbox{for $i=0$}, \\
        -\frac{1}{2}(p^\tri_{\alpha_{i-1},\alpha_{i}}+p^\tri_{\alpha_{i+1},\alpha_{i}}) & \mbox{for $0 < i < m$}, \\
        -\frac{1}{2}(p^\tri_{\alpha_{m-1},\alpha_m}+p^\tri_{\alpha_m,\alpha_m}) =-\frac{1}{2}(1-1) = 0 & \mbox{for $i=m$}.
    \end{cases}
\end{align*}
The middle case produces $n_{\alpha,\alpha_{i_0}}=1$ if $i=i_0$, and otherwise $n_{\alpha,\alpha_i}=0$. 
It is easier to verify $n_{\alpha\beta}=0$ for $\beta \in e(\tri) \setminus \{\alpha_i\}_{i=0}^m$.
Thus $n_{\alpha\beta}=\delta_{\alpha\beta}$ holds in this case. 

\begin{figure}[ht]
    \centering
\begin{tikzpicture}
\bline{-2,-2}{2,-2}{0.2};
\tline{-2,2}{2,2}{0.2};
\draw[blue] (-2,0) -- (0,-2) -- (2,0) -- (0,2) -- cycle;
\draw[blue] (0,2) -- (0,-2);
\foreach \i in {1,2} \draw(-2,0)++(0,\i*0.5) coordinate(A\i);
\draw[blue] (-2,0) -- (A1) -- (0,2);
\draw[blue] (A1) --(A2) -- (0,2);
\foreach \i in {1,2} \draw(2,0)++(0,-\i*0.5) coordinate(B\i);
\draw[blue] (2,0) -- (B1) -- (0,-2);
\draw[blue] (B1) --(B2) -- (0,-2);
\draw[blue,thick,dotted] (0,2)++(-160:1.5) arc(-160:-178:1.5);
\draw[blue,thick,dotted] (0,-2)++(20:1.5) arc(20:2:1.5);
\filldraw (0,2) circle(1.5pt);
\filldraw (0,-2) circle(1.5pt);
\draw[red,thick] (0.5,-2) to[out=90,in=-45] (0,0) to[out=135,in=-90] (-0.5,2);
\node[blue] at (-1,2.4) {$\alpha_0$};
\node[blue] at (1,-2.4) {$\alpha_m$};
\node[blue] at (0.3,0.4) {$\alpha_{i_0}$};
\node[red] at (-0.3,-0.1) {$\alpha_\bB$};

\begin{scope}[xshift=7cm]
\tline{-2,2}{2,2}{0.2};
\foreach \i in {-30,-45,-60,-90,-120,-135,-150} \draw[blue] (0,2) --++(\i:2);
\draw[blue,thick,dotted] (0,2)++(-160:1.5) arc(-160:-178:1.5);
\draw[blue,thick,dotted] (0,2)++(-20:1.5) arc(-20:-2:1.5);
\filldraw (0,2) circle(1.5pt);
\draw[red,thick] (-1,2) arc(-180:0:1);
\node[blue] at (-1,2.4) {$\alpha_0$};
\node[blue] at (1,2.4) {$\alpha_m$};
\node[red] at (-0.3,0.7) {$\alpha_\bB$};
\end{scope}
\end{tikzpicture}
    \caption{Computation of the matrix $n_{\alpha\beta}$. Left: the case $\alpha \in e_{\interior}(\tri)$, Right: the case $\alpha \in \bB$.}
    \label{fig:shift_interior}
\end{figure}

In the case where $\alpha$ is a boundary interval, the curve $\alpha_\bB$ is the corner arc surrounding its terminal marked point $m \in \bM$. Let us give $\alpha_\bB$ an orientation so that it runs around $m$ in the counter-clockwise direction. We have $\alpha=\alpha_m$. 
See the right picture in \cref{fig:shift_interior}. 
Then one can verify that $n_{\alpha,\alpha_i}=0$ for $0 \leq i <m$, and $n_{\alpha,\alpha_m}=-\frac{1}{2}(p^\tri_{\alpha_{m-1},\alpha_m}+p^\tri_{\alpha_m,\alpha_m})=-\frac{1}{2}(-1-1)=1$ in this case. Thus $n_{\alpha\beta}=\delta_{\alpha\beta}$ holds in this case. The assertion is proved.
\end{proof}
In particular, the Poisson bracket of $\lambda$-length functions along compatible arcs are computed as
\begin{align}\label{eq:Poisson_lambda}
    \{A_\alpha,A_\beta\} = \left\{ \prod_{\gamma} (X^\tri_{\gamma})^{q_{\alpha\gamma}}, \prod_{\delta} (X^\tri_{\delta})^{q_{\beta\delta}}\right\} = \left(\sum_{\gamma,\delta} q_{\alpha\gamma}\ve^\tri_{\gamma\delta}q_{\beta\delta}\right) A_\alpha A_\beta,
\end{align}
where we take any ideal triangulation $\tri$ containing $\alpha,\beta$. 
In order to describe it more precisely, recall the Muller's compatibility matrix $\pi_{\alpha\beta}$, defined as follows. For an ideal arc $\alpha$, let $\alpha_+,\alpha_-$ denote its two ends (with an arbitrary labeling). For two ideal arcs $\alpha,\beta$, define 
\begin{align*}
    \pi_{\alpha_\mu,\beta_{\nu}}:=\begin{cases}
    1 & \mbox{if $\alpha_\mu$ is clockwise to $\beta_{\nu}$ at a common marked point}, \\
    -1 & \mbox{if $\alpha_\mu$ is counter-clockwise to $\beta_{\nu}$  at a common marked point}, \\
    0 & \mbox{otherwise},
    \end{cases}
\end{align*}
and set $\pi_{\alpha\beta}:=\sum_{\mu,\nu=+,-} \pi_{\alpha_\mu,\beta_{\nu}}$. 

\begin{lem}
For any unpunctured marked surface, we have 
\begin{align*}
    \sum_{\gamma,\delta \in e(\tri)} q_{\alpha\gamma}\ve^\tri_{\gamma\delta}q_{\beta\delta}=-\frac{1}{4}\pi_{\alpha\beta}.
\end{align*}
In particular, the Poisson bracket \eqref{eq:Poisson_lambda} becomes
\begin{align*}
    \{A_\alpha,A_\beta\} = -\frac{1}{4}\pi_{\alpha\beta} A_\alpha A_\beta.
\end{align*}
\end{lem}

\begin{proof}
Since $\ve^\tri_{\gamma\delta} = p^\tri_{\gamma\delta} -m_{\gamma\delta} = (q^{-1})_{\gamma\delta} - m_{\gamma\delta}$ by the lemma above, we have 
\begin{align*}
    &\sum_{\gamma,\delta \in e(\tri)} q_{\alpha\gamma}\ve^\tri_{\gamma\delta}q_{\beta\delta} \\
    &=q_{\beta\alpha} - \sum_{\gamma \in \bB, \delta \in e(\tri)} q_{\alpha\gamma}m_{\gamma\delta}q_{\beta\delta}\\
    &= q_{\beta\alpha} + \sum_{\gamma \in \bB} q_{\alpha\gamma}q_{\beta\gamma} \\
    &= -\sfa_\alpha(\beta_\bB) + \sum_{\gamma \in \bB} \sfa_{\gamma}(\alpha_\bB)\sfa_{\gamma}(\beta_\bB).
\end{align*}
The last expression is clearly $0$ if $\alpha$ and $\beta$ do not share endpoints. If an end $\alpha_\mu$ of $\alpha$ is clockwise to an end $\beta_\nu$ of $\beta$ at a common marked point, then such a pair $(\alpha_\mu,\beta_\nu)$ contributes to $\sum_{\gamma \in \bB} \sfa_{\gamma}(\alpha_\bB)\sfa_{\gamma}(\beta_\bB)$ by $1/4$, and to $-\sfa_\alpha(\beta_\bB)$ by $-1/2$. In total, its contribution is $-1/4$. If $\alpha_\mu$ is counter-clockwise to $\beta_\nu$, then the contribution of the pair $(\alpha_\mu,\beta_\nu)$ is $0+1/4$, since the shifted end $\beta_{\nu,\bB}$ is disjoint from $\alpha_\mu$ in this case. Thus the assertion is proved. 
\end{proof}

\begin{rem}
\begin{enumerate}
    \item By the lemma above, the Poisson algebra $(C^\infty(\cT^a(\Sigma)),-4\cdot\{\ ,\ \})$ is the classical limit of the Muller's skein algebra $\mathscr{S}_\Sigma^{q}$. 
    \item A similar computation works for higher rank cases as well. In the $\mathfrak{sl}_3$-case, the $A$-variables are expressed as Laurent monomials of $X$-variables with exponents given by $-1/3$ times the Douglas--Sun coordinates \cite{DS20I} of the bounded $\mathfrak{sl}_3$-laminations obtained by $\bB$-shifting the corresponding elementary webs \cite{IYsl3}. The Poisson bracket multiplied by $-6$ gives the classical limit of the skein algebra studied in \cite{IYsl3}.
\end{enumerate}
\end{rem}

\begin{rem}
Via the correspondence between the decorations and pinnings, one can also consider the gluing map $q_{\Sigma,\Sigma'}:\cT^a(\Sigma) \to \cT^a(\Sigma')$ for any marked surface $\Sigma$ in the way illustrated as
\begin{align*}
\begin{tikzpicture}[scale=0.9]
\begin{scope}
\draw (-2,0) -- (0,0) -- (0,3) -- (-2,3);
\draw[red] (0,0) arc(-45:-45+360:0.9);
\draw[red] (1,3) arc(135:135-360:0.6);
\draw[blue] (0,3) arc(45:45+360:0.4);
\draw[mygreen] (1,0) arc(-135:-135+360:0.6);
\foreach \i in {0,1} \foreach \j in {0,3} \fill(\i,\j) circle(1.5pt);
\pinn{0,0.9*1.414}{180}{0.1}{0.03cm}
\draw[dashed] (0.1,0.9*1.414) -- (0.9,3-0.6*1.414);
\end{scope}
\begin{scope}[xshift=1cm]
\draw (2,0) -- (0,0) -- (0,3) -- (2,3);
\pinn{0,3-0.6*1.414}{0}{0.1}{0.03cm}
\end{scope}
\draw[thick,|->] (3.5,1.5) --node[midway,above]{$q_{\Sigma,\Sigma'}$}++(1,0);
\begin{scope}[xshift=7cm]
\draw (-2,0) -- (2,0);
\draw (-2,3) -- (2,3);
\draw[dashed] (0,0) -- (0,3);
\draw[blue] (0,3) arc(45:45+360:0.4);
\draw[mygreen] (0,0) arc(-135:-135+360:0.6);
\foreach \j in {0,3} \fill(0,\j) circle(1.5pt);
\end{scope}
\end{tikzpicture}\ ,
\end{align*}
which is invariant under the $\pos$-action rescaling the $h$-lengths of the red horocycles by $(\lambda,\lambda^{-1})$ for $\lambda \in \pos$. It clearly satisfies $q_{\Sigma,\Sigma'}^\ast A_{\overline{\alpha}}=A_{\alpha_L}\cdot A_{\alpha_R}$.
\end{rem}

\subsection{Wilson lines and $\lambda$-length}
In addition to the usual trace functions of monodromy (\emph{a.k.a.} Wilson loops), the data of pinnings allow us to consider a wider class of functions associated to arcs connecting boundary intervals, which we call the \emph{Wilson lines}. Let us recall the setting from \cite[Section 3.3]{IO20} with a specialization to the $A_1$ case.

An \emph{arc class} is the homotopy class $[c]$ of a path $c$ in $\Sigma$ which runs between two boundary interval $\alpha_\inn$ and $\alpha_\out$, where the homotopies are relative to $\partial^\ast \Sigma$. 
Given an arc class $[c]$ from $\alpha_\inn$ to $\alpha_\out$ and  $(h,p) \in \pT(\Sigma)$, we define an isometry $g_{[c]}(h,p) \in PSL_2(\bR)$ as follows. 

Choose a fundamental polygon $\widetilde{\Pi}^h \subset \bH^2$ of $\Sigma$ so that the unique lift $\widetilde{\alpha}^h_\inn$ of $\alpha_\inn$ contained in $\widetilde{\Pi}^h$ sits in the \lq\lq normalized" position: $\widetilde{\alpha}^h_\inn=\sqrt{-1}\bR_{>0}$ and the point pinning $p_{\alpha_\inn}$ gives $\sqrt{-1} \in \widetilde{\alpha}^h_\inn$. Let $\widetilde{c}^h$ be the lift of $c$ which starts from $\widetilde{\alpha}^h_\inn$, which ends on a certain side $\widetilde{\alpha}^h_\out$ of $\widetilde{\Pi}^h$. The terminal side $\widetilde{\alpha}^h_\out$ must be a lift of $\alpha_\out$, so it is equipped with a point pinning determined by $p_{\alpha_\out}$. 
Define $g=g_{[c]}(h,p) \in PSL_2(\bR)$ to be the unique isometry such that $g(\widetilde{\alpha}^h_\inn)=\widetilde{\alpha}^h_\out$, matching the point pinnings on them. 
In this way, we get a map
\begin{align*}
    g_{[c]}: \pT(\Sigma) \to PSL_2(\bR),
\end{align*}
which we call the \emph{Wilson line} along  $[c]$.

Let $\tri$ be an ideal triangulation of $\Sigma$. Represent an arc class $[c]$ by a curve $c$ so that the intersection with $\tri$ is minimal. Label the edges (resp. triangles) of $\tri$ that $c$ traverses as $\alpha_\inn=\alpha_0,\dots,\alpha_M=\alpha_\out$ (resp. $T_1,\dots,T_M$) in this order. Note that each intersection $c \cap T_\nu$ is either one of the two patterns shown in \cref{f:intersection}. The \emph{turning pattern} of $[c]$ with respect to $\tri$ is the sequence $\tau_\tri([c]) = (\tau_\nu)_{\nu=1}^M \in \{L,R\}^{M}$, where $\tau_\nu = L$ (resp. $\tau_\nu=R$) if $c \cap T_\nu$ is the left (resp. right) pattern in \cref{f:intersection}.

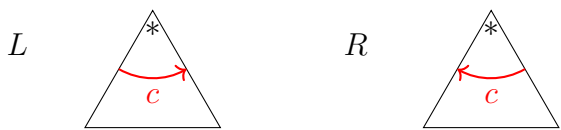
\begin{figure}[hb]
\centering
\begin{tikzpicture}[scale=0.9]
\draw (0,0) coordinate (B1) node[below]{$\ast$};
\draw (240: 2) coordinate (B2);
\draw (300: 2) coordinate (B3);
\draw (B1) -- (B2) -- (B3) --cycle;
\draw[->,thick,color=red] (240:1) arc[start angle=240, end angle=300, radius=1cm] node[midway,below]{$c$};
\draw(-2,-0.5) node{$L$};
\begin{scope}[xshift=5cm]
\draw (0,0) coordinate (B1) node[below]{$\ast$};
\draw (240: 2) coordinate (B2);
\draw (300: 2) coordinate (B3);
\draw (B1) -- (B2) -- (B3) --cycle;
\draw[->,thick,color=red] (300:1) arc[start angle=300, end angle=240, radius=1cm] node[midway,below]{$c$};
\draw(-2,-0.5) node{$R$};
\end{scope}
\end{tikzpicture}
\caption{Two intersection patterns of $c \cap T_\nu$}
\label{f:intersection}
\end{figure}

\begin{thm}[\cite{FG07,Penner,IO20}]\label{thm:LR-formula}
Let $\tri$ be an ideal triangulation of $\Sigma$, and $[c]$ an arc class. Then in terms of the cross ratios $X_\nu:=X_{\alpha_\nu}^\tri(h,p)$ for $\nu=0,\dots,M$, the Wilson line $g_{[c]}$ is expressed as
\begin{align}\label{eq:LR-formula}
    g_{[c]}(h,p) = H(X_0)\mathbb{E}^{\tau_1}H(X_2)\mathbb{E}^{\tau_2}\dots H(X_{M-1})\mathbb{E}^{\tau_{M}}H(X_M),
\end{align}
where 
\begin{align*}
    H(X):=\begin{bmatrix}X^{1/2} & 0 \\ 0 & X^{-1/2} \end{bmatrix}, \quad \mathbb{E}^L:=\begin{bmatrix}1 & 1 \\ 0 & 1 \end{bmatrix}, \quad \mathbb{E}^R:=\begin{bmatrix}1 & 0 \\ 1 & 1 \end{bmatrix} \in PSL_2(\bR).
\end{align*}
\end{thm}


The Wilson line reproduces the lambda-length function, as follows. For a $2\times 2$ matrix $M$, let $\Delta_{ij}(M)$ denote its $(i,j)$-entry for $i,j=1,2$. For $M \in PSL_2(\bR)$, $\Delta_{ij}(M)$ is defined up to sign. 
Observe that an arc class $[c]$ without self-intersections is represented by the $\bB$-shift $c=\alpha_\bB$ (\cref{def:shift_ideal}) of some ideal arc $\alpha$, together with an arbitrary orientation. 

\begin{prop}\label{prop:Wilson_lambda}
Let $[\alpha_\bB]$ be the arc class represented by the $\bB$-shift of an ideal arc $\alpha$.
\begin{enumerate}
    \item For any ideal triangulation $\tri$ of $\Sigma$, we have
    \begin{align*}
    |\Delta_{22}(g_{[\alpha_\bB]})| = \prod_{\beta \in e(\tri)}(X_\beta^\tri)^{-\sfa_\beta(\alpha_\bB)}\cdot F_\alpha^\tri(X_\tri),
    \end{align*}
    where $F_\alpha^\tri(X_\tri)$ is a polynomial of the cross ratio coordinates with respect to $\tri$ having the constant term $1$. 
    \item We have 
\begin{align*}
   |\Delta_{22}(g_{[\alpha_\bB]})| = A_\alpha.
\end{align*}
\end{enumerate}

\end{prop}

\begin{proof}
Let $\tri$ be any ideal triangulation, and $\alpha_0,\dots,\alpha_M$ its edges that the curve $\alpha_\bB$ traverses in this order. 
Let us rescale the diagonal matrices as 
\begin{align*}
    H'(X):=X^{1/2}H(X) = \begin{bmatrix}X & 0 \\ 0 & 1 \end{bmatrix}.
\end{align*}

Then the formula \eqref{eq:LR-formula} becomes
\begin{align*}
    g_{[\alpha_\bB]} = \prod_{\nu=0}^M X_\nu^{-1/2}\cdot H'(X_0)\mathbb{E}^{\tau_1}H'(X_2)\mathbb{E}^{\tau_2}\dots H'(X_{M-1})\mathbb{E}^{\tau_{M}}H'(X_M).
\end{align*}
With a notice that the monomial term $\prod_{\nu=0}^M X_\nu^{-1/2}$ coincides with $\prod_{\beta \in e(\tri)}(X_\beta^\tri)^{-\sfa_\beta(\alpha_\bB)}$, we see that the first assertion holds. 

If $\tri$ contains the ideal arc $\alpha$, then we have $A_\alpha=\prod_{\beta \in e(\tri)}(X_\beta^\tri)^{-\sfa_\beta(\alpha_\bB)}$ by \cref{thm:A to X}. 
In this case, the turning pattern is given by $\tau_1=\cdots=\tau_{i_0}=L$ and $\tau_{i_0+1}=\cdots=\tau_{M}=R$ in the notation of \cref{fig:shift_interior}. In particular we get $F_\alpha^\tri=1$, and hence $|\Delta_{22}(g_{[\alpha_\bB]})| = A_\alpha$ holds.
\end{proof}
\begin{rem}
The expressions of the other entries of $g_{[c]}$ are also given in \cite[(5.1)]{IOS22}.
\end{rem}

%% file: 3_lamination.tex
\section{Lamination spaces with pinnings}\label{sec:lamination}
In this section, we introduce the space of \emph{$\P$-laminations} which will be identified with the set of real tropical points of the moduli space $\P_{PGL_2,\Sigma}$. 

\subsection{The space $\cL^p(\Sigma,\bQ)$ of rational $\P$-laminations and shear coordinates}
Let $\Sigma$ be a marked surface. During this section, by a \emph{curve} we mean an unoriented curve $\gamma$ in $\Sigma$ which is either closed or having endpoints in $\bM_\circ \cup \partial^\ast \Sigma$, and the other part is embedded into $\Sigma^\ast$. Isotopies of curves are considered within this class. 
Such a curve $\gamma$ is said to be
\begin{itemize}
    \item \emph{peripheral} \footnote{It is called \lq\lq special" in \cite{FG07}.} if it is either isotopic to a puncture $m \in \bM_\circ$ or an interval in $\partial \Sigma$ which contains exactly one special point $m \in M_\partial$. 
    \begin{align}\label{eq:peripheral}
    \begin{tikzpicture}[scale=.8]
    \draw[dashed, fill=white] (0,0) circle [radius=1];
    \draw[thick, red, fill=pink!60] (0,0) circle [radius=0.5];
    \filldraw[draw=black,fill=white] (0,0) circle(2.5pt);
    \begin{scope}[xshift=5cm]
    \coordinate (P) at (-0.5,0) {};
    \coordinate (P') at (0.5,0) {};
    \coordinate (C) at (0,0.5) {};
    \draw[thick, red, fill=pink!60] (P) to[out=north, in=west] (C) to[out=east, in=north] (P');
    \draw[dashed] (1,0) arc (0:180:1cm);
    \bline{-1,0}{1,0}{0.2}
    \draw[fill=black] (0,0) circle(2pt);
    \end{scope}
    \end{tikzpicture}
    \end{align}
    In each case, it is called a peripheral curve \emph{around $m$}.
    \item \emph{contractible} if it is isotopic to a point.
\end{itemize}

\begin{dfn}\label{d:P-lamination}
A \emph{rational $\P$-lamination} on $\Sigma$ consists of the following data:
\begin{itemize}
    \item a collection $L=\{(\gamma_j,w_j)\}_j$ of mutually disjoint non-peripheral curves $\gamma_j$ in $\Sigma$ equipped with non-negative rational weights $w_j\geq 0$;
    \item a tuple $\sigma_L=(\sigma_m)_{m \in \bM_\circ} \in \{+,0,-\}^{\bM_\circ}$ of signs assigned to punctures such that $\sigma_m=0$ if and only if there are no curves incident to $m$;
    \item a tuple $\nu=(\nu_\alpha)_{\alpha \in \mathbb{B}} \in \bQ^\mathbb{B}$ of rational numbers assigned to boundary intervals.
\end{itemize}
Such a data is considered modulo the equivalence relation generated by isotopies and the following operations:
\begin{enumerate}
    \item Remove a contractible curve or a curve with weight $0$.
    \item Combine a pair of isotopic curves with weights $u$ and $v$ into a single curve with the weight $u+v$.
\end{enumerate}
\end{dfn}
Let $\mathcal{L}^p(\Sigma,\bQ)$ denote the set of rational $\P$-laminations. 
We call the tuple $\sigma_L=(\sigma_m)_{m \in \bM_\circ}$ the \emph{lamination signature}, and $\nu$ the \emph{pinning}. 
Forgetting the pinnings, we get the projection
\begin{align*}
    \pi_\Sigma^\mathsf{T}: \cL^p(\Sigma,\bQ) \to \cL^x(\Sigma,\bQ), \quad (L,\sigma_L,\nu) \mapsto (L,\sigma_L),
\end{align*}
where $\cL^x(\Sigma,\bQ)$ denotes the space of rational $\X$-laminations of Fock--Goncharov \cite{FG07}. 

A rational $\P$-lamination is said to be \emph{integral} if all the weights of the curves and the pinnings $\nu_\alpha$ are integers. Let $\cL^p(\Sigma,\bZ)\subset \cL^p(\Sigma,\bQ)$ denote the subset of integral $\P$-laminations. 

\smallskip
\paragraph{\textbf{Spiralling diagram.}}
Given a rational $\X$-lamination $(L,\sigma_L) \in \cL^x(\Sigma,\bQ)$, we deform each curve $\gamma_j$ in $L$ incident to a puncture $m$, as follows: 
if the sign $\sigma_m$ is positive (resp. negative), then replace the corresponding end of $\gamma_j$ with an infinite curve $\widehat{\gamma}_j$ that spirals around $m$ in the clockwise (resp. counter-clockwise) direction. See \cref{fig:spiral}. The resulting diagram $\spiral$ is called the \emph{spiralling diagram} of $(L,\sigma_L)$.

\begin{figure}[ht]
    \centering
\begin{tikzpicture}[scale=0.9]
\draw[dashed] (-2.5,-1.5) circle(2cm);
\draw [red](-3,0.45) .. controls (-2.5,0) and (-2.9,-0.8) .. (-2.5,-1.5);
\filldraw[fill=white] (-2.5,-1.5) circle(2pt);
\node[red] at (-2.4,-0.1) {$\gamma_j$};
\node[red] at (-2.3,-1.7) {$+$};
\node at (-2.8,-1.7) {$m$};
\draw (5,-1.5) circle(2pt);
\draw[dashed] (5,-1.5) circle(2cm);
\node[red] at (5.3,-0.1) {$\widehat{\gamma}_j$};

\draw [red](4.5,0.45) .. controls (5,0) and (5.55,-1.05) .. (5.55,-1.5) .. controls (5.55,-1.85) and (5.25,-2) .. (5,-2) .. controls (4.75,-2) and (4.55,-1.8) .. (4.55,-1.5) .. controls (4.55,-1.25) and (4.75,-1.1) .. (5,-1.1) .. controls (5.25,-1.1) and (5.4,-1.25) .. (5.4,-1.5) .. controls (5.4,-1.75) and (5.2,-1.85) .. (5,-1.85) .. controls (4.85,-1.85) and (4.7,-1.7) .. (4.7,-1.5) .. controls (4.7,-1.35) and (4.85,-1.25) .. (5,-1.25) .. controls (5.15,-1.25) and (5.25,-1.35) .. (5.25,-1.5) .. controls (5.25,-1.6) and (5.15,-1.7) .. (5,-1.7) .. controls (4.9,-1.7) and (4.85,-1.6) .. (4.85,-1.5);
\draw [red, thick, dotted](4.85,-1.5) .. controls (4.85,-1.3) and (5.15,-1.3) .. (5.15,-1.5);

\draw [thick,-{Classical TikZ Rightarrow[length=4pt]},decorate,decoration={snake,amplitude=2pt,pre length=2pt,post length=3pt}](0.65,-1.5) -- (2,-1.5);
\end{tikzpicture}
    \caption{Construction of a spiralling diagram. The negative sign similarly produce an end spiralling counter-clockwisely.}
    \label{fig:spiral}
\end{figure}
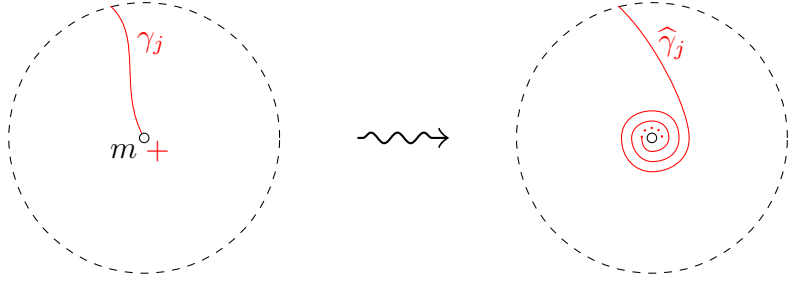
Given an ideal triangulation $\tri$ of $\Sigma$, it is easy to verify that we can move such a spiralling diagram by an isotopy fixing a small neighborhood of $\bM_\circ$ into a position such that its restriction to each triangle of $\tri$ consists only of corner arcs (\emph{i.e.} curves connecting distinct edges). We call such a position a \emph{good position} with respect to $\tri$.

Then we define a coordinate system 
\begin{align*}
    \sfx_\tri=(\sfx_\alpha^\tri)_{\alpha \in e(\tri)}:\mathcal{L}^p(\Sigma,\bQ) \to \bQ^\tri
\end{align*}
associated with an ideal triangulation $\tri$,
as follows. Given $(L,\sigma_L,\nu) \in \cL^p(\Sigma,\bQ)$, let $\hL$ be the spiralling diagram of $(L,\sigma_L)$ in a good position with respect to $\tri$. 
For each edge $\alpha \in e(\tri)$ and a curve $\widehat{\gamma}_j$ in the spiralling diagram, let $(\alpha:\widehat{\gamma}_j) \in \bZ$ be the integer defined as follows: 

\begin{itemize}
    \item if $\alpha$ is an interior edge, then it is the diagonal of a unique quadrilateral $Q_\alpha$ in $\tri$. An intersection between a portion of $\widehat{\gamma}_j$ and $Q_\alpha$ as in the left (resp. right) of \cref{f:intersection sign} contributes as $+1$ (resp. $-1$), and the others $0$. Then $(\alpha:\widehat{\gamma}_j)$ is the sum of these local contributions. 
    \item if $\alpha$ is a boundary interval, then $(\alpha:\widehat{\gamma}_j):=+1$ if $\widehat{\gamma}_j$ contains a corner arc around the initial marked point $m^+_\alpha$ as its portion, and otherwise $0$.  
\end{itemize}

\begin{figure}[ht]
\centering
\begin{tikzpicture}[scale=1.15]
\path(0,0) node [fill, circle, inner sep=1.5pt] (x1){};
\path(135:2) node [fill, circle, inner sep=1.5pt] (x2){};
\path(0,2*1.4142) node [fill, circle, inner sep=1.5pt] (x3){};
\path(45:2) node [fill, circle, inner sep=1.5pt] (x4){};
\draw[blue](x1) to (x2) to (x3) to (x4) to (x1) to node[midway,left]{$\alpha$} node[midway,above right,black]{$\oplus$} (x3);
\draw [red] (2,0.7) to[out=135,in=-45] (0,1.5) to[out=135,in=-45] (-1,3);
\draw [red] (2,0.7) node[below]{$\gamma_j$}; 

\begin{scope}[xshift=5cm]
\path(0,0) node [fill, circle, inner sep=1.5pt] (x1){};
\path(135:2) node [fill, circle, inner sep=1.5pt] (x2){};
\path(0,2*1.4142) node [fill, circle, inner sep=1.5pt] (x3){};
\path(45:2) node [fill, circle, inner sep=1.5pt] (x4){};
\draw[blue] (x1) to (x2) to (x3) to (x4) to (x1) to node[midway,left]{$\alpha$} node[midway,below right,black]{$\ominus$} (x3);
\draw [red] (-1,0.5)  to[out=45,in=215] (0,1.2) to[out=45,in=215] (1,3);
\draw [red] (-1,0.5) node[below]{$\gamma_j$};
\end{scope}
\end{tikzpicture}
\caption{Contributions to $(\alpha:\widehat{\gamma}_j)$.}
\label{f:intersection sign}
\end{figure}
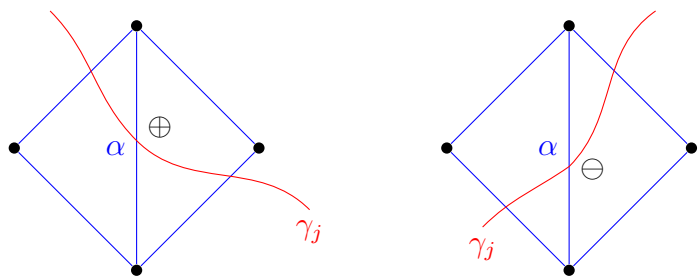

Although $\widehat{\gamma}_j$ may intersect with $\alpha$ infinitely many times in the first case, the number $(\alpha:\widehat{\gamma}_j)$ is always finite. 
Then we define $\sfx_\alpha^\tri(L,\sigma_L,\nu) \in \bQ$ by the following rule:

\begin{itemize}
    \item For an interior edge $\alpha \in e_{\interior}(\tri)$
    define
    \begin{align*}
        \sfx^\tri_\alpha(L,\sigma_L,\nu):=\sum_j w_j (\alpha:\widehat{\gamma}_j).
    \end{align*}
    \item For a boundary interval $\alpha\in \mathbb{B}$, define
    \begin{align}\label{eq:boundary_coord}
        \sfx^\tri_\alpha(L,\sigma_L,\nu):=\nu_\alpha- \sum_j w_j (\alpha:\widehat{\gamma}_j).
    \end{align}
\end{itemize}

We call the coordinate system $\sfx_\tri$ the \emph{(lamination) shear coordinates} associated with $\tri$. 
The following is a slight extension of the result in \cite[Section 3.1]{FG07}.

\begin{thm}\label{prop:p-lamination}
For any ideal triangulation $\tri$ of $\Sigma$, the map
\begin{align*}
    \sfx_\tri: \mathcal{L}^p(\Sigma,\bQ) \xrightarrow{\sim} \bQ^{\tri}
\end{align*}
gives a bijection. For the flip $f_{\alpha}:\tri \to \tri'$ along an interior edge $\alpha \in e_{\interior}(\tri)$, the coordinate transformation $\sfx_{\tri'} \circ \sfx_{\tri}^{-1}$ is given as in \cref{f:tropical x-flip}. Here we assume that both $\tri$ and $\tri'$ do not have self-folded triangles. 
\end{thm}
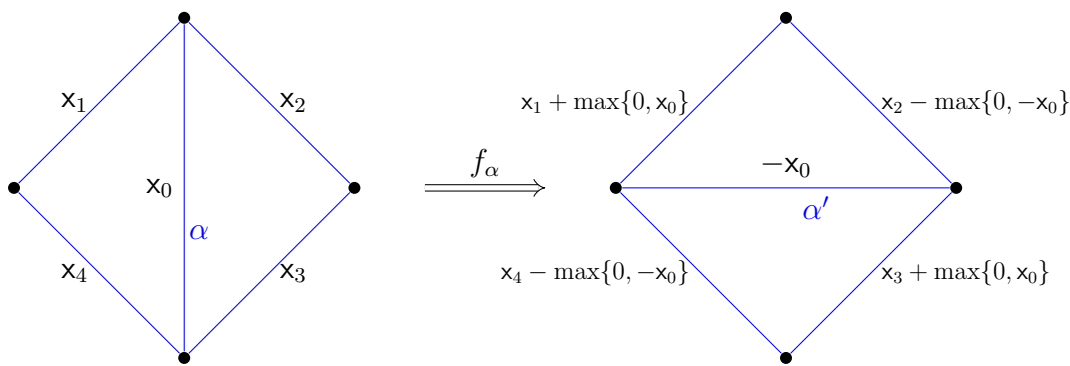
\begin{figure}[ht]
\[\hspace{1.4cm}
\begin{tikzpicture}[scale=0.8]
\path(0,0) node [fill, circle, inner sep=1.6pt] (x1){};
\path(135:4) node [fill, circle, inner sep=1.6pt] (x2){};
\path(0,4*1.4142) node [fill, circle, inner sep=1.6pt] (x3){};
\path(45:4) node [fill, circle, inner sep=1.6pt] (x4){};
\draw[blue](x1) to node[midway,left,black]{$\sfx_4$} (x2) 
to node[midway,left,black]{$\sfx_1$} (x3) 
to node[midway,right,black]{$\sfx_2$} (x4) 
to node[midway,right,black]{$\sfx_3$} (x1) 
to node[midway,left,black]{$\sfx_0$} (x3);

\draw[-implies, double distance=2pt](4,2*1.4142) to node[midway,above]{$f_{\alpha}$} (6,2*1.4142);

\begin{scope}[xshift=10cm]
\path(0,0) node [fill, circle, inner sep=1.6pt] (x1){};
\path(135:4) node [fill, circle, inner sep=1.6pt] (x2){};
\path(0,4*1.4142) node [fill, circle, inner sep=1.6pt] (x3){};
\path(45:4) node [fill, circle, inner sep=1.6pt] (x4){};
\draw[blue](x1) to node[midway,left,black]{\scalebox{0.8}{$\sfx_4-\max\{0,-\sfx_0\}$}} (x2) 
to node[midway,left,black]{\scalebox{0.8}{$\sfx_1+\max\{0,\sfx_0\}$}} (x3) 
to node[midway,right,black]{\scalebox{0.8}{$\sfx_2-\max\{0,-\sfx_0\}$}} (x4) 
to node[midway,right,black]{\scalebox{0.8}{$\sfx_3+\max\{0,\sfx_0\}$}} (x1);
\draw[blue] (x2) to node[midway,above,black]{$-\sfx_0$} (x4);
\end{scope}
\node [blue] at (0.25,2.1) {$\alpha$};
\node [blue] at (10.5,2.5) {$\alpha'$};
\end{tikzpicture}
\]
\caption{The coordinate transformation for a flip. 
}
\label{f:tropical x-flip}
\end{figure}

\begin{proof}
It is known that $\sfx_\tri^{\uf}:=(\sfx_\alpha^\tri)_{\alpha \in e_{\interior}(\tri)}: \cL^x(\Sigma,\bQ) \xrightarrow{\sim} \bQ^{e_{\interior}(\tri)}$ gives a bijection \cite[Section 3.1]{FG07}. In other words, given a vector $\sfx=(\sfx_\alpha) \in \bQ^{e(\tri)}$, one can uniquely reconstruct a rational $\X$-lamination $(L,\sigma_L)$ such that $\sfx_\alpha^\tri(L,\sigma_L)=\sfx_\alpha$ for $\alpha \in e_{\interior}(\tri)$. Then the pinning $\nu$ can be reconstructed from $(L,\sigma_L)$ and the boundary coordinates via the relation \eqref{eq:boundary_coord}. Thus the first statement holds. 

When all the edges in \cref{f:tropical x-flip} are interior edges, the formula is the one given in \cite{FG07}. Consider the case where one of the edges, say $\alpha_1$, is a boundary interval. 
Fix a rational $\P$-lamination $(L,\sigma_L,\nu) \in \cL^p(\Sigma,\bQ)$. For $i \neq j \in \{0,1,2,3,4\}$, let $\mathsf{w}_{ij}^\tri=\mathsf{w}_{ij}^\tri(L,\sigma_L,\nu)$ denote the weighted sum of the leaves which surround the corner bounded by the edges $\alpha_i$ and $\alpha_j$ in $\tri$. Let $\mathsf{w}_{ij}^{\tri'}$ be the similar quantity for the triangulation $\tri'$. Since the pinnings contributes to the frozen coordinates linearly, we may assume that $\nu_\alpha=0$ for all $\alpha \in \bB$ without loss of generality. 

Then from the definitions, $\sfx_{\alpha_1}^\tri= -\mathsf{w}_{01}^\tri$ and $\sfx_{\alpha_1}^{\tri'} = -\mathsf{w}_{12}^{\tri'}$. 
If $\sfx_{\alpha_0}^\tri \geq 0$, then 
$\mathsf{w}_{12}^{\tri'} = \mathsf{w}_{01}^\tri - \sfx_{\alpha_0}^\tri=-(\sfx_{\alpha_1}^\tri+\sfx_{\alpha_0}^\tri)$, 
and hence $\sfx_{\alpha_1}^{\tri'}=\sfx_{\alpha_1}^\tri+\sfx_{\alpha_0}^\tri$. 
If $\sfx_{\alpha_0}^\tri \leq 0$, then $\mathsf{w}_{12}^{\tri'} = \mathsf{w}_{01}^\tri$ and hence $\sfx_{\alpha_1}^{\tri'}=\sfx_{\alpha_1}^\tri$. By a similar argument for the edge $\alpha_2$ and the symmetry, we get the desired formula. 
\end{proof}

The formula in \cref{f:tropical x-flip} is the tropical analogue of the cluster Poisson transformation \eqref{eq:X-transf}. Then we get:

\begin{cor}
The shear coordinates $\sfx_\tri: \mathcal{L}^p(\Sigma,\bQ) \xrightarrow{\sim} \bQ^{e(\tri)}$ associated with ideal triangulations $\tri$ of $\Sigma$ combine to give a canonical $MC(\Sigma)$-equivariant isomorphism $\cL^p(\Sigma,\bQ) \xrightarrow{\sim} \X_{\Sigma}(\bQ^\mathsf{T})$. 
\end{cor}

\paragraph{\textbf{Fock--Goncharov's reconstruction, revisited.}}
For later use in \cref{subsec:lamination_amalgamation}, let us recall the reconstruction procedure of a rational $\X$-laminations from the shear coordinates given in \cite{FG07}. 
Suppose $(\sfx_\alpha)_\alpha \in \bZ^{e_{\interior}(\tri)}$ is given.
On each triangle $T \in t(\tri)$, draw an infinite collection of disjoint corner arcs around each corner (\cref{fig:gluing block}). We are going to glue these local blocks together to form an integral $\X$-lamination. 

\begin{figure}[ht]
    \centering
\begin{tikzpicture}
\draw[blue] (-30:2) coordinate(A) -- (90:2) coordinate(B) -- (210:2) coordinate(C) --cycle;
\begin{scope}
\clip (-30:2) -- (90:2) -- (210:2) --cycle;

\foreach \i in {0.8,1,1.2,1.4}
{
\foreach \x in {0,120,240}
    {
    \draw(-30+\x:2)++(120+\x:\i) coordinate(a);
    \draw[red] (a) arc(120+\x:180+\x:\i);
    \draw[red,thick,dotted] (-30+\x:1.7) -- (-30+\x:1.34);
    }
}
\end{scope}
\end{tikzpicture}
    \caption{The building block for reconstruction from the shear coordinates.}
    \label{fig:gluing block}
\end{figure}
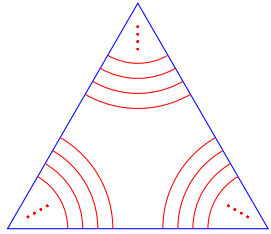

Consider two triangles $T_L$ and $T_R$ that share an interior edge $\alpha$. 
Fatten $\alpha$ into a biangle $B_\alpha$, which is bounded by the boundary intervals $\alpha_L$ and $\alpha_R$ of $T_L$ and $T_R$, respectively. 
For $Z \in \{L,R\}$, let $S_{Z}$ denote the set of endpoints of the infinite corner arcs on $\alpha_Z$. 
We connect the points in $S_L$ and $S_R$ inside the biangle $B_\alpha$ by the following rule. See \cref{fig:FG_gluing}.
\begin{itemize}
    \item For $Z \in \{L,R\}$, choose an orientation-preserving homeomorphism $\phi_{Z}:\bR \to \alpha_Z$ so that $\phi_{Z}(\frac{1}{2}+\bZ)=S_{Z}^\pm$, and $\phi_Z(\bR_{<0}) \cap S_Z$ consists of all the strands coming from the corner arcs around $m^+_{\alpha_Z}$. 
    \item Put the points 
    \begin{align}\label{eq:pins_reconstruction}
        p_L:=\phi_{L}(\sfx_{\alpha})\quad \mbox{and} \quad p_R:=\phi_{R}(0),
    \end{align}
    which we call the \emph{pins}. 
    \item There exists an orientation-reversing homeomorphism $f:\alpha_L \to \alpha_R$ such that $f(\frac{1}{2}+\bZ)=\frac{1}{2}+\bZ$ and $f(p_L)=p_R$. 
    Connect the points $s \in S_L$ to the points $f(s) \in S_R$ by a disjoint collection of curves. 
\end{itemize}

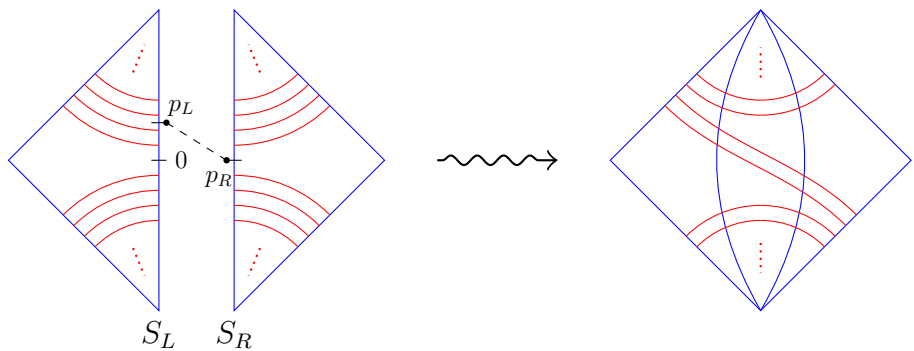
\begin{figure}[ht]
    \centering
\begin{tikzpicture}
\draw[blue] (-2,0) -- (0,2) -- (0,-2) --cycle;
\node[below] at (0,-2) {$S_L$};

\foreach \x in {1.8,1.6,1.4,1.2}
{
\draw[red] (0,-2+\x) arc(90:135:\x); 
\draw[red] (0,2-\x) arc(-90:-135:\x);
}
\draw[red,thick,dotted] ($(0,-2)+(112.5:0.9)$)--++(-67.5:0.4);
\draw[red,thick,dotted] ($(0,2)+(-112.5:0.9)$)--++(67.5:0.4);

\draw(-0.1,-0) -- (0.1,-0) node[right,scale=0.8]{$0$};
\pinn{0,0.5}{180}{0.1}{0.03cm}
\draw(0.1,0.5) coordinate(p1);
\node[scale=0.8] at (0.3,0.7) {$p_L$};

\begin{scope}[xshift=1cm]
\draw[blue] (2,0) -- (0,2) -- (0,-2) --cycle;
\node[below] at (0,-2) {$S_R$};

\foreach \x in {1.8,1.6,1.4,1.2}
{
\draw[red] (0,-2+\x) arc(90:45:\x); 
\draw[red] (0,2-\x) arc(-90:-45:\x);
}
\draw[red,thick,dotted] ($(0,-2)+(67.5:0.9)$)--++(-112.5:0.4);
\draw[red,thick,dotted] ($(0,2)+(-67.5:0.9)$)--++(112.5:0.4);

\pinn{0,0}{0}{0.1}{0.03cm}
\draw(-0.1,0) coordinate(p2);
\node[scale=0.8] at (-0.2,-0.25) {$p_R$};
\end{scope}
\draw[dashed] (p1) -- (p2);
\draw [thick,-{Classical TikZ Rightarrow[length=4pt]},decorate,decoration={snake,amplitude=1.8pt,pre length=2pt,post length=3pt}](3.7,0) --(5.3,0);

\begin{scope}[xshift=8cm]
\draw[blue] (2,0) -- (0,2) -- (-2,0) -- (0,-2) --cycle;
\draw[blue] (0,-2) to[bend left=30pt] (0,2);
\draw[blue] (0,-2) to[bend right=30pt] (0,2);
\foreach \x in {1.4,1.2}
{
\draw[red] (0,-2+\x) arc(90:135:\x); 
\draw[red] (0,-2+\x) arc(90:45:\x); 
\draw[red] (0,2-\x) arc(-90:-135:\x);
\draw[red] (0,2-\x) arc(-90:-45:\x);
}
\draw[red,thick,dotted] ($(0,-2)+(90:0.9)$)--++(-90:0.4);
\draw[red,thick,dotted] ($(0,2)+(-90:0.9)$)--++(90:0.4);

\draw[red] (0,2)++(-135:1.6) ..controls ++(-45:1) and ($(0,-2)+(45:1.8)+(135:1)$).. ($(0,-2)+(45:1.8)$);
\draw[red] (0,2)++(-135:1.8) ..controls ++(-45:1) and ($(0,-2)+(45:1.6)+(135:1)$).. ($(0,-2)+(45:1.6)$);
\end{scope}
\end{tikzpicture}
    \caption{Fock--Goncharov's gluing procedure of $\X$-laminations.}
    \label{fig:FG_gluing}
\end{figure}

Then we get an infinite collection of curves on the quadrilateral $T_L \cup B_\alpha \cup T_R$. Applying this construction to each pair of consecutive triangles, we get an infinite collection of curves on $\Sigma$. Then we do the followings: 
\begin{itemize}
    \item Remove the peripheral curves around each special point of $\Sigma$.
    \item For each puncture $m \in \bM_\circ$ of $\Sigma$, replace each spiralling end around $m$ with a signed end at $m$, while encoding the spiralling directions in signs by reversing the rule in \cref{fig:spiral}. 
\end{itemize}
Then we get an integral $\X$-lamination $(L,\sigma_L)$, which satisfies $\sfx^\tri_\alpha(L,\sigma_L) = \sfx_\alpha$ for $\alpha \in e_{\interior}(\tri)$. 

\begin{rem}\label{rem:pin_shift}
The asymmetry of the pins $p_L$ and $p_R$ is explained as follows. In fact, if we change the pins to $p'_L:=\phi_L(\sfx_{\alpha}-\nu)$ and $p'_R:=\phi_R(\nu)$ for $\nu \in \bZ$, the resulting pairing of points does not change. In particular, we could instead use the pins $p'_L=\phi_L(0)$ and $p'_R=\phi_R(\sfx_\alpha)$, which produces the same result. 
\end{rem}

\subsection{Gluing map}\label{subsec:lamination_amalgamation}
Now let us turn our attention to the tropical analogue of the gluing map. In the setting at the beginning of \cref{subsec:Teich_amalgamation}, we are going to construct a map 
\begin{align*}
    q_{\Sigma,\Sigma'}^\mathsf{T}: \cL^p(\Sigma,\bQ) \to \cL^p(\Sigma',\bQ)
\end{align*}
satisfying the equation $(q_{\Sigma,\Sigma'}^\mathsf{T})^*\sfx_{\overline{\alpha}}^{\tri'} = \sfx_{\alpha_L}^\tri + \sfx_{\alpha_R}^\tri$, which is the tropical analogue of the formula given in \cref{prop:amalgamation}. It is also defined so that equivariant under the $\bQ_{>0}$-action rescaling the weights on the curves, and the action $\sigma_{\alpha_L,\alpha_R}:\bQ \curvearrowright \cL^p(\Sigma,\bQ)$ given by the shift
\begin{align}\label{eq:gluing_action}
    \mu.(\nu_{\alpha_L},\nu_{\alpha_R}):= (\nu_{\alpha_L}+\mu,\nu_{\alpha_R}-\mu)
\end{align}
for $\mu \in \bQ$, keeping the other $\nu_\alpha$, $\alpha\neq \alpha_L,\alpha_R$ intact. 

Let $(L,\sigma_L,\nu) \in \cL^p(\Sigma,\bZ)$ be an integral $\P$-lamination. Represent $L$ by a collection of curves with weight $1$. 
Around each endpoint of $\alpha_L$ and $\alpha_R$, draw an infinite collection of disjoint peripheral curves so that they are disjoint from the curves in $L$. For $Z \in \{L,R\}$, let $S_Z$ denote the set of the endpoints of the curves in $L$ and these additional peripheral curves on the edge $\alpha_Z$. 
Insert a biangle $B$ between $\alpha_L$ and $\alpha_R$, and identify $\Sigma'= \Sigma \cup B$. We connect the points in $S_L$ and $S_R$ inside the biangle $B$ by the following rule:
\begin{itemize}
    \item Choose an orientation-preserving homeomorphism $\psi_{Z}:\bR \to \alpha_Z$ so that $\psi_{Z}(\frac{1}{2}+\bZ)=S_{Z}$, and $\psi_Z(\bR_{<0}) \cap S_Z$ consists of all the endpoints of the additional peripheral curves around the marked point $m^+_{\alpha_Z}$. 
    \item Put the point $p_Z:=\psi_{Z}(\nu_{\alpha_Z}) \in \alpha_Z$, which we call the \emph{pin}.
    \item There exists an orientation-reversing homeomorphism $f:\alpha_L \to \alpha_R$ such that $f(\frac{1}{2}+\bZ)=\frac{1}{2}+\bZ$ and $f(p_L)=p_R$. 
    Connect the points $s \in S_L$ to the points $f(s) \in S_R$ by a disjoint collection of curves. 
\end{itemize}
Then we get an infinite collection of curves on $\Sigma'=\Sigma\cup B$. Here the reader should notice the similarity to the reconstruction procedure given in the previous subsection. 
The marked points of $\alpha_L$ and $\alpha_R$ are identified, and regarded as new marked points in $\Sigma'$. For each of these new marked points, do the similar procedure as before: remove the peripherals around new special points, and replace spiralling ends to signed ends around new punctures. 
Thus we get an integral $\P$-lamination $\hL'=q^\mathsf{T}_{\Sigma,\Sigma'}(\hL) \in \cL^p(\Sigma,\bZ)$. 
The construction is clearly equivariant under the rescaling $\bZ_{>0}$-action. 

\begin{dfn}\label{dfn:gluing_lamination}
By extending the above construction $\bQ_{>0}$-equivariantly, we obtain a map $q^\mathsf{T}_{\Sigma,\Sigma'}: \cL^p(\Sigma,\bQ) \to \cL^p(\Sigma',\bQ)$, which we call the \emph{gluing map} along $\alpha_L$ and $\alpha_R$.
\end{dfn}

The following is easily verified with \cref{rem:pin_shift} in mind:
\begin{lem}\label{lem:shift-invariance}
The gluing map $q^\mathsf{T}_{\Sigma,\Sigma'}$ is invariant under the shift action \eqref{eq:gluing_action}.
\end{lem}

Any ideal triangulation $\tri$ of $\Sigma$ naturally induces a triangulation $\tri'$ of $\Sigma'$, where the edges $\alpha_L$ and $\alpha_R$ are identified and give an interior edge $\overline{\alpha}$ of $\tri$. The other edges are naturally inherited to $\tri'$. 

\begin{thm}\label{thm:amalgamation}
The gluing map $q^\mathsf{T}_{\Sigma,\Sigma'}$ is the tropical analogue of \cref{prop:amalgamation}. Namely, for any ideal triangulation $\tri$ of $\Sigma$ and the induced triangulation $\tri'$ of $\Sigma'$, it satisfies
\begin{align*}
    (q^\mathsf{T}_{\Sigma,\Sigma'})^\ast\sfx^{\tri'}_{\overline{\alpha}} = \sfx^\tri_{\alpha_L} + \sfx^\tri_{\alpha_R},
\end{align*}
and the other coordinates are kept intact: $(q^\mathsf{T}_{\Sigma,\Sigma'})^\ast\sfx^{\tri'}_{\alpha} = \sfx^\tri_{\alpha}$ for $\alpha \neq \overline{\alpha}$.
\end{thm}

\begin{proof}
The last statement is clear from the definition. To see the relation between the coordinates on the edges $\alpha_L$, $\alpha_R$ and $\alpha$, it suffices to consider an integral lamination $(L,\sigma_L,\nu) \in \cL^p(\Sigma,\bZ)$ by $\bQ_{>0}$-equivariance. Write $\sfx_\alpha:=\sfx_\alpha^\tri(L,\sigma_L,\nu)$ for $\alpha \in e(\tri)$
. Recall from \eqref{eq:boundary_coord} that the pinnings are given by
\begin{align*}
    \begin{aligned}
    \nu_{\alpha_L}= \sfx_{\alpha_L} +c_{\alpha_L}, \quad \nu_{\alpha_R}&= \sfx_{\alpha_R} +c_{\alpha_R},
    \end{aligned}
\end{align*} 
where we write $c_{\alpha}:= \sum_j w_j(\alpha:\widehat{\gamma}_j)$ for $\alpha\in \bB$ with $L=\{(\gamma_j,w_j)\}_j$. By \cref{rem:pin_shift}, the result of gluing is the same if we use the pins $\widetilde{p}_Z=\psi_Z(\widetilde{\nu}_{\alpha_Z})$ with 
\begin{align}\label{eq:pins_gluing}
    \begin{aligned}
    \widetilde{\nu}_{\alpha_L}:= (\sfx_{\alpha_L}+\sfx_{\alpha_R}) +c_{\alpha_L}, \quad \widetilde{\nu}_{\alpha_R}:= c_{\alpha_R}.
    \end{aligned}
\end{align}
Comparing to the reconstruction procedure in the previous subsection, we here have \lq\lq original'' corner arcs of $L$ in $T_L$ and $T_R$ before adding infinite collections of peripheral curves in the gluing procedure. Hence the two parametrizations of edges are related by 
\begin{align*}
    \phi_Z(n) = \psi_Z(n+c_{\alpha_Z})
\end{align*}
for $n \in \bZ$ and $Z \in \{L,R\}$. 
See \cref{fig:difference_gluing}. 
Comparing two choices of pins \eqref{eq:pins_reconstruction} and \eqref{eq:pins_gluing} under this relation, we see that $(L',\sigma_{L'},\nu')=q_{\Sigma,\Sigma'}(L,\sigma_L,\nu)$ if and only if $\sfx_{\overline{\alpha}}(L',\sigma_{L'},\nu')=\sfx_{\alpha_L}(L,\sigma_L,\nu) + \sfx_{\alpha_R}(L,\sigma_L,\nu)$. 
\end{proof}

\begin{figure}[ht]
\begin{tikzpicture}
\draw[blue] (0,-2) -- (0,2) --(-2,0) --cycle;
\draw[thick] (0,-2.5) -- (0,2.5);
\fill(0,2) circle(2pt);
\fill(0,-2) circle(2pt);
\foreach \i in {0.8,1.0,1.2,1.4}
\draw[myblue] (0,-2+\i) arc(90:135:\i);
\foreach \i in {1.6,1.8}
\draw[red] (0,-2+\i) arc(90:135:\i);
\foreach \i in {0.8,1.0,1.2,1.4}
\draw[myblue] (0,2-\i) arc(-90:-135:\i);
\draw[myblue,thick,dotted] ($(0,-2)+(112.5:0.6)$)--++(-67.5:0.3);
\draw[myblue,thick,dotted] ($(0,2)+(-112.5:0.6)$)--++(67.5:0.3);

\draw[dashed] (-0.1,-0.5) --++(0.6,0) node[anchor=west]{$\psi_L^+(0)$};
\draw[dashed] (-0.1,0) --++(0.6,0) node[anchor=west]{$\phi_L^+(0)$};
\node at (-1.5,1) {$T_L$};
\node at (0.5,1.2) {$\alpha_L$};
\end{tikzpicture}
    \caption{Comparison of two edge parametrizations. Arcs in the given lamination are shown in red, while the peripheral curves added upon the gluing procedure are shown in blue.}
    \label{fig:difference_gluing}
\end{figure}
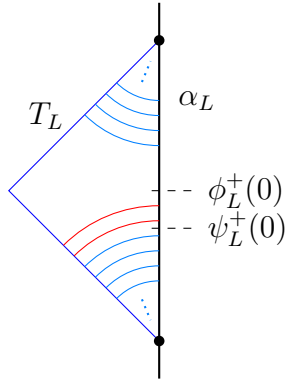

\subsection{Ensemble map}

Recall the following from \cite{FG07}:
\begin{dfn}
A \emph{rational $\A$-lamination} on $\Sigma$ is the isotopy class of a mutually non-isotopic, disjoint collection $\{\gamma_i\}_i$ of curves in $\Sigma$ that do not incident to punctures, together with rational weights $w_i \in \bQ$ such that $w_i \geq 0$ if $\gamma_i$ is non-peripheral. Such a data is considered modulo the equivalence relation generated by isotopies and the following operations:
\begin{enumerate}
    \item Remove a contractible curve or a curve with weight $0$.
    \item Combine a pair of isotopic curves with weights $u$ and $v$ into a single curve with the weight $u+v$.
\end{enumerate}
Let $\cL^a(\Sigma,\bQ)$ denote the set of integral $\A$-laminations, whose element is denoted by $L=\{(\gamma_i,w_i)\}_i$. 
\end{dfn}
For each ideal arc $\alpha$ on $\Sigma$ and a rational $\A$-lamination $L=\{(\gamma_i,w_i)\}_i$, isotope each curve $\gamma_i$ so that the intersection with $\alpha$ is minimal. Then we define
\begin{align*}
    \sfa_\alpha(L):= \sum_i w_i \sfa_\alpha(\gamma_i),
\end{align*}
where $\sfa_\alpha(\gamma_i) \in \frac 1 2 \bZ_{\geq 0}$ denotes half the geometric intersection number of $\alpha$ and $\gamma_i$. Then it is known that, for any ideal triangulation $\tri$ of $\Sigma$, the map
\begin{align*}
    \sfa_\tri:=(\sfa_\alpha)_{\alpha \in e(\tri)}: \cL^a(\Sigma,\bQ) \to \bQ^{e(\tri)}
\end{align*}
gives a bijection. They transform by the tropical analogue of cluster $K_2$-transformation \eqref{eq:A-transf}, and thus combine to give an $MC(\Sigma)$-equivariant isomorphism $\cL^a(\Sigma,\bQ) \xrightarrow{\sim} \A_{\Sigma}(\bQ^\mathsf{T})$. A rational $\A$-lamination $L$ is said to be \emph{integral} if $\sfa_\tri(L) \in \bZ^{e(\tri)}$ for any ideal triangulation $\tri$. Notice that an $\A$-lamination with integral weights may not be integral in this sense. 
Since the coordinate transformations are integral piece-wise linear, it suffices to check this condition for one triangulation. Let $\cL^a(\Sigma,\bZ) \subset \cL^a(\Sigma,\bQ)$ denote the subset of integral $\A$-laminations, which is identified with $\A_\Sigma(\bZ^\mathsf{T})$. 

Let us define the \emph{extended ensemble map}
\begin{align}
    p_\Sigma^\mathsf{T}: \cL^a(\Sigma,\bQ) \to \cL^p(\Sigma,\bQ)
\end{align}
by forgetting the peripheral components, and defining the pinning $\nu_\alpha \in \bZ$ to be minus the weight of the peripheral component around the initial marked point $m^+_\alpha$. 

\begin{prop}\label{prop:ensemble_tropical}
For any ideal triangulation $\tri$ of $\Sigma$, we have
\begin{align}\label{eq:ensemble_tropical}
    (p_\Sigma^\mathsf{T})^\ast \sfx_\kappa^\tri = \sum_{\alpha \in e(\tri)} (\ve_{\kappa\alpha}^\tri+ m_{\kappa\alpha}) \sfa_\alpha
\end{align}
for all $\kappa \in e(\tri)$.
\end{prop}

\begin{proof}
For $\kappa \in e_{\interior}(\tri)$, we have $m_{\kappa\alpha}=0$ and hence the formula is proved in \cite{FG07}. 

For $\kappa \in \bB$, label the edges of the unique triangle $T$ containing $\kappa$ as $\alpha,\beta,\kappa$ in this clockwise order (as in \cref{fig:ensemble_boundary}). Let $m \in \bM$ be the initial marked point of $\kappa$. 
Let $L$ be a rational $\A$-lamination, and denote by $\mathsf{w}_{T,m}(L)$ (resp. $\mathsf{w}_m(L)$) the total weight of the corner arcs of $T\cap L$ (resp. the total weight of the peripheral components of $L$) around $m$. Then we have 
\begin{align*}
    \mathsf{w}_{T,m}(L) = \sfa_\alpha(L) +\sfa_\kappa(L) - \sfa_\beta(L).
\end{align*}
Observe that the integral $\P$-lamination $p_\Sigma^\mathsf{T}(L)$ has the corner arcs with the total weight $\mathsf{w}_{T,m}(L) - \mathsf{w}_m(L)$, equipped with the pinning $\nu_\kappa=-\mathsf{w}_m(L)$. Then by definition of the coordinate $\sfx_\kappa$, we get
\begin{align*}
    \sfx^\tri_\kappa(p_\Sigma^\mathsf{T}(L)) = \nu_\kappa - (\mathsf{w}_{T,m}(L) - \mathsf{w}_m(L)) = -\mathsf{w}_{T,m}(L) =-\sfa_\alpha(L) -\sfa_\kappa(L) + \sfa_\beta(L).
\end{align*}
This is exactly the desired formula. The assertion is proved.
\end{proof}

When $\bM_\circ \neq \emptyset$, the map $p_\Sigma^\mathsf{T}$ is neither injective nor surjective, since it forgets peripheral components and its image does not have components incident to punctures. 

\begin{thm}\label{thm:index_2}
If $\bM_\circ=\emptyset$, then $p_\Sigma^\mathsf{T}: \cL^a(\Sigma,\bQ) \xrightarrow{\sim} \cL^p(\Sigma,\bQ)$ is a bijection. Moreover, its restriction to the subset $\cL^a(\Sigma,\bZ)$ is an embedding of index $2$.
\end{thm}

\begin{proof}
The first assertion is clear since the weights of peripheral components around special points can be recovered from the pinnings. For the second assertion, observe that the inverse formula of \eqref{eq:ensemble_tropical} is given by
\begin{align*}
    \sfa_\alpha= \sum_{\beta \in e(\tri)} q_{\alpha\beta} \sfx_\beta^\tri
\end{align*}
as the linear version of the formula given in \cref{thm:A to X}. The image $p_\Sigma^\mathsf{T}(\cL^a(\Sigma,\bZ))$ is characterized as the subset where the coordinates $\sfa_\alpha$ are integral for all $\alpha\in e(\tri)$, which is obviously a sub-lattice of index $2$. 
\end{proof}

\subsection{Thurston compactification with pinnings}
The coordinate transformation $\sfx_{\tri'} \circ \sfx_{\tri}^{-1}$ given in \cref{prop:p-lamination} is a Lipschitz map with respect to the Euclidean metric on $\bQ^{\tri} \cong \bQ^{-3\chi(\Sigma^\ast)+2|\bM_\partial|}$.
Let $\mathcal{L}^p(\Sigma,\bR)$ be the corresponding metric completion of $\mathcal{L}^p(\Sigma,\bQ)$, which does not depend on a specific coordinate system. Each coordinate system $\sfx_\tri$ is extended to a homeomorphism $\sfx_\tri: \mathcal{L}^p(\Sigma,\bR) \xrightarrow{\sim} \bR^\tri$, being still denoted by the same symbol. We call an element of $\mathcal{L}^p(\Sigma,\bR)$ a \emph{real $\P$-lamination}. We have the following structures:
\begin{itemize}
    \item Since the $\bQ_{>0}$-action on $\cL^p(\Sigma,\bQ)$ rescaling the weights is continuous, we get a continuous $\bR_{>0}$-action on $\cL^p(\Sigma,\bR)$. 
    \item The Shift action \eqref{eq:gluing_action} of the pinnings is also extended to a continuous action $\sigma_{\alpha_L,\alpha_R}:\bR \curvearrowright \cL^p(\Sigma,\bR)$.
    \item Since the coordinate expression of the gluing map given in \cref{thm:amalgamation} is continuous, it is extended to a continuous map
    \begin{align}\label{eq:gluing_Th_boundary}
        q_{\Sigma,\Sigma'}^\mathsf{T}: \cL^p(\Sigma,\bR) \to \cL^p(\Sigma'\bR),
    \end{align}
    which is invariant under $\sigma_{\alpha_L,\alpha_R}$ above.
\end{itemize}
Let us consider the sphere $\bS \cL^p(\Sigma,\bR):=\cL^p(\Sigma,\bR)/\bR_{>0} \cong S^{-3\chi(\Sigma^\ast)+2|\bM_\partial|-1}$.

\begin{dfn}\label{dfn:Thurston_P}
The \emph{Thurston compactification} of the \Teich\ space with pinnings is defined to be
\begin{align*}
    \overline{\cT^p(\Sigma)}:=\cT^p(\Sigma) \cup \bS \cL^p(\Sigma,\bR),
\end{align*}
where the topology is endowed so that a sequence $(g_n)$ in $\cT^p(\Sigma)$ converges to a point $[G] \in \bS\cL^p(\Sigma,\bR)$ if 
\begin{align}\label{eq:Thurston_convergence}
    [\log X_{\alpha_1}^\tri(g_n):\cdots: \log X_{\alpha_N}^\tri(g_n)] \to [\sfx^\tri_{\alpha_1}(G):\cdots: \sfx^\tri_{\alpha_N}(G)], \quad n\to\infty
\end{align}
for any ideal triangulation $\tri$. Here $e(\tri)=\{\alpha_1,\dots,\alpha_N\}$.
\end{dfn}
It is known \cite{FG16,Le16,Ish19} that the condition \eqref{eq:Thurston_convergence} does not depend on the triangulation. In particular, the action of the mapping class group $MC(\Sigma)$ continuously extends to $\overline{\cT^p(\Sigma)}$. The topological space $\overline{\cT^p(\Sigma)}$ is homeomorphic to a closed ball of dimension $-3\chi(\Sigma^\ast)+2|\bM_\partial|$. 

\begin{thm}\label{thm:gluing_Thurston}
The gluing maps \eqref{eq:gluing_Teich} and \eqref{eq:gluing_Th_boundary} combine to give a continuous map
\begin{align*}
    \overline{q}_{\Sigma,\Sigma'}: \overline{\cT^p(\Sigma)} \to \overline{\cT^p(\Sigma')}
\end{align*}
between the Thurston compactifications.
\end{thm}

\begin{proof}
It immediately follows from the coordinate expressions
\begin{align*}
    q_{\Sigma,\Sigma'}^\ast (\log X_{\overline{\alpha}}^{\tri'}) &= \log X_{\alpha_L}^\tri + \log X_{\alpha_R}^\tri, \\
    (q^\mathsf{T}_{\Sigma,\Sigma'})^\ast (\sfx_{\overline{\alpha}}^{\tri'}) &= \sfx_{\alpha_L}^\tri + \sfx_{\alpha_R}^\tri
\end{align*}
and the definition of the topology on the compactification.
\end{proof}

%% file: 4_duality.tex
\section{Duality maps}\label{sec:duality}
The \Teich\ spaces $(\cT^a(\Sigma),\cT^p(\Sigma))$ are ``positive real parts'' of the moduli spaces $(\A^\times_{SL_2,\Sigma},\P_{PGL_2,\Sigma})$ introduced by Fock, Goncharov and Shen. These moduli spaces have natural cluster structures, for which we have algebra isomorphisms $\cO(\A_\Sigma) \cong \cO(\A_{SL_2,\Sigma}^\times)$ and $\cO(\X_\Sigma) \cong \cO(\P_{PGL_2,\Sigma})$ over $\bC$ \cite{Shen20,IOS22}. 
Recall the canonical isomorphisms $\X_{\Sigma}(\bZ^\mathsf{T}) \cong \cL^p(\Sigma,\bZ)$ and $\A_{\Sigma}(\bZ^\mathsf{T}) \cong \cL^a(\Sigma,\bZ)$.
In this section, we study duality maps
\begin{align*}
    &\mathbb{I}_\X: \X_{\Sigma}(\bZ^\mathsf{T}) \to \cO(\A_{\Sigma}), \\
    &\mathbb{I}_\A: \A_{\Sigma}(\bZ^\mathsf{T}) \to \cO(\X_{\Sigma}) 
\end{align*}
based on our investigation on the ``$\P$-type" spaces in the previous sections.

\subsection{Relation with the moduli spaces of $SL_2$-/$PGL_2$-local systems}
In order to precisely state algebraic results, we quickly review the relation between the \Teich\ theory developed in the previous sections with the moduli spaces of $SL_2$-/$PGL_2$-local systems introduced in \cite{FG06,GS19}.

Let $\Sigma$ be a marked surface, and consider the algebraic group $SL_2$ over $\bC$. To the pair $(SL_2,\Sigma)$, associated is the moduli space $\A_{SL_2,\Sigma}$ of \emph{decorated twisted $SL_2$-local systems} on $\Sigma$. It is an algebraic stack over $\bC$. The reader is referred to \cite{FG06} for details. Fock--Goncharov showed that the moduli space $\A_{G,\Sigma}$ has a canonical cluster $K_2$-structure, and its positive real part is canonically identified with the decorated \Teich\ space $\cT^a(\Sigma)$ \cite[Theorem 1.7 (b)]{FG06}\footnote{Indeed, the isomorphism $\cT^a(\Sigma)\xrightarrow{\sim} \A_{SL_2,\Sigma}(\pos)$ is obtained as follows. We can lift the monodromy representation $\rho:\pi_1(\Sigma) \to PSL_2(\bR)$ of a marked hyperbolic structure to a twisted representation $\widetilde{\rho}: \pi_1(T'\Sigma) \to SL_2(\bR)$, as discussed in \cite[Section 1.3]{BW}. See \cite[Section 11]{FG06} for an appropriate way to lift a cyclic configuration of horocycles to a twisted cyclic configuration in the decorated flag variety $\A_{SL_2}(\bR)=\bR^2\setminus\{0\}$.}. From this, we get an algebra embedding
\begin{align}
    \iota_\A: \cO(\A_{SL_2,\Sigma}) \hookrightarrow \cC^\infty(\cT^a(\Sigma)),
\end{align}
where $\cO(\A_{SL_2,\Sigma})$ denotes the $\bC$-algebra of global functions on $\A_{SL_2,\Sigma}$. 
The $\lambda$-length coordinate $A_\alpha$ along an ideal arc $\alpha$ lies in the image of $\iota_\A$, and realized by the cluster $K_2$-coordinate on $\A_{SL_2,\Sigma}$ assigned to $\alpha$. See \cite[Section 11.2]{FG06}. When $\Sigma$ is unpunctured, the function algebra $\cO(\A^\times_{SL_2,\Sigma})$ of a certain open subspace $\A^\times_{SL_2,\Sigma} \subset \A_{SL_2,\Sigma}$ is known to coincide with the associated \emph{cluster algebra} $\mathscr{A}_{\mathfrak{sl}_2,\Sigma}$ (see, for instance, \cite{IOS22}). Hence the cluster $\A$-coordinates, together with the inverses of frozen coordinates, generate the algebra $\cO(\A^\times_{SL_2,\Sigma})$. 

There is a similar results related to $\cT^p(\Sigma)$. Let $PGL_2:=GL_2/\mathbb{G}_m$, the adjoint group of $SL_2$ having the same Lie algebra $\mathfrak{sl}_2$. To the pair $(PGL_2,\Sigma)$, associated is the moduli space $\P_{PGL_2,\Sigma}$ of \emph{framed $PGL_2$-local systems with pinnings} on $\Sigma$. It is introduced in \cite{GS19}, extending the moduli space $\X_{PGL_2,\Sigma}$ studied in \cite{FG06}. See Section 3 \emph{loc.~sit.~}for the $PGL_2$-case. The moduli space $\P_{PGL_2,\Sigma}$ has a cluster Poisson structure, and the pair $(\A^\times_{SL_2,\Sigma},\P_{PGL_2,\Sigma})$ forms a \emph{cluster ensemble} in the sense in \cite{FG09}. In particular, there is the \emph{(extended) ensemble map} $p_\Sigma: \A^\times_{SL_2,\Sigma} \to \P_{PGL_2,\Sigma}$. In terms of the coordinates, it is expressed as
\begin{align*}
    p_\Sigma^\ast X_\kappa^\tri = \prod_{\alpha \in e(\tri)} A_\alpha^{\ve^\tri_{\kappa\alpha}+m_{\kappa\alpha}}
\end{align*}
for all $\kappa \in e(\tri)$. 
When $\Sigma$ is unpunctured, the induced homomorphism
\begin{align*}
    p_\Sigma^\ast: \cO(\P_{PGL_2,\Sigma}) \to \cO(\A^\times_{SL_2,\Sigma})
\end{align*}
is an injective, finite homomorphism of index $2$. 
As a slight extension of \cite[Theorem 1.7 (a)]{FG06}, one can verify that the positive real part of $\P_{PGL_2,\Sigma}$ is identified with the \Teich\ space with pinnings $\cT^p(\Sigma)$. From this, we get an algebra embedding
\begin{align}
    \iota_\X: \cO(\P_{PGL_2,\Sigma}) \hookrightarrow \cC^\infty(\cT^p(\Sigma)).
\end{align}
Although the cross ratios $X_\alpha^\tri$ are not extended to global functions on $\P_{PGL_2,\Sigma}$, they can be defined on an open subspace $\P_{PGL_2,\Sigma}^\tri \subset \P_{PGL_2,\Sigma}$ associated with an ideal triangulation $\tri$. Hence the cross ratios lie in the image of a similar embedding  
\begin{align}
    \iota_\X^\tri: \cO(\P_{PGL_2,\Sigma}^\tri) \hookrightarrow \cC^\infty(\cT^p(\Sigma)),
\end{align}
and realized by the cluster Poisson coordinates. From \cite{IO20}, we have Wilson line morphisms $g_{[c]}: \P_{PGL_2,\Sigma} \to PGL_2$ associated with any arc class $[c]$. 

The following result allows us to study the cluster algebras $\cO(\A_\Sigma)$ and $\cO(\X_\Sigma)$ in terms of these moduli spaces:

\begin{thm}[\cite{Shen20} for $\X_\Sigma$, \cite{IOS22} for $\A_\Sigma$]
We have algebra isomorphisms $\cO(\A_\Sigma) \cong \cO(\A_{SL_2,\Sigma}^\times)$ and $\cO(\X_\Sigma) \cong \cO(\P_{PGL_2,\Sigma})$ over $\bC$.
\end{thm}
Summarizing, we have
\begin{align*}
    \cO(\A_\Sigma)=\cO(\A_{SL_2,\Sigma}^\times) \subset C^\infty(\cT^a(\Sigma)), \quad \cO(\X_\Sigma)=\cO(\P_{PGL_2,\Sigma}) \subset C^\infty(\cT^p(\Sigma)).
\end{align*}
In particular, the relations (for instance those given in \cref{prop:ensemble,thm:A to X,prop:Wilson_lambda}) among the coordinates/Wilson lines mentioned above are valid inside these subalgebras.

\subsection{The basis of $\cO(\X_\Sigma)$ parametrized by the integral $\A$-laminations}
For simplicity, let us restrict our attention to an unpunctured surface $\Sigma$. It is straightforward to extend our construction to the general case, following \cite{FG09}.

\begin{dfn}\label{def:skein_lift_A}
Let $L=\{(\gamma_i,w_i)\} \in \cL^a(\Sigma,\bZ)$ be an integral $\A$-lamination. We define the corresponding function $\mathbb{I}_\A(L) \in C^\infty(\cT^p(\Sigma))$, as follows.
\begin{itemize}
    \item For each weighted non-peripheral loop $(\gamma_i,w_i)$, associate the trace-of-monodromy function
    \begin{align*}
        \mathrm{Tr}_{[\gamma_i]^{w_i}},
    \end{align*}
    where $[\gamma_i]^{w_i} \in \pi_1(\Sigma)$ denotes the $w_i$-th power of a based loop homotopic to $\gamma_i$. 
    \item For each weighted non-peripheral arc $(\gamma_i,w_i)$, associate the function
    \begin{align*}
        \Delta_{22}(g_{[\gamma_i]})^{w_i}.
    \end{align*}
    \item For each weighted peripheral arc $(\gamma_i,w_i)$ around a special point, associate the function
    \begin{align*}
        \Delta_{22}(g_{[\gamma_i]})^{w_i}. 
    \end{align*}
    Here note that $\Delta_{22}(g_{[\gamma_i]})^{-1}=\Delta_{11}(g_{[\gamma_i]})$, $\gamma_i$ being peripheral. 
\end{itemize}
Then the function $\mathbb{I}_\A(L) \in C^\infty(\cT^p(\Sigma))$ is defined to be the product of these elements. 
\end{dfn}
The map $\mathbb{I}_\A:\A_\Sigma(\bZ^{\mathsf{T}}) \to C^\infty(\cT^p(\Sigma))$ is clearly $MC(\Sigma)$-equivariant. 
Notice that the trace functions $\mathrm{Tr}_{[\gamma_i]^{w_i}}$ and the matrix coefficients $\Delta_{kl}(g_{[\gamma_i]})$ themselves do not belong to the subalgebra $\cO(\P_{PGL_2,\Sigma})$, since the Wilson line takes its value in $PGL_2$, rather than $SL_2$. Nevertheless, we have:

\begin{lem}\label{lem:function_even}
For any integral $\A$-lamination, the product $\mathbb{I}_\A(L)$ belongs to $\cO(\P_{PGL_2,\Sigma})$. Namely it is a well-defined global function on the moduli space $\P_{PGL_2,\Sigma}$. The Laurent expression of $\mathbb{I}_\A(L)$ in the cluster coordinates has the unique lowest term $\prod_{\alpha \in e(\tri)} (X_\alpha^\tri)^{-\sfa_\alpha(L)}$ for any ideal triangulation $\tri$.
\end{lem}

\begin{proof}
Fix an ideal triangulation $\tri$ of $\Sigma$, and consider the coordinate expression of $\mathbb{I}_\A(L)$. 
From \cref{prop:Wilson_lambda} (1), the $(2,2)$-entry of Wilson lines are expressed as
\begin{align*}
    |\Delta_{22}(g_{[\gamma_i]})|^{w_i}&=\prod_{\alpha \in e(\tri)} (X^\tri_\alpha)^{-w_i\sfa_\alpha(\gamma_i)}F_i^\tri(X_\tri)^{w_i}
\end{align*}
for some polynomial $F_i^\tri$ in the coordinates $X_\alpha^\tri$ with constant term $1$. 
A similar computation is applied to the monodromy, and hence we get
\begin{align*}
    \mathrm{Tr}_{[\gamma_i]^{w_i}} =\prod_{\alpha \in e(\tri)} (X^\tri_\alpha)^{-w_i\sfa_\alpha(\gamma_i)}F_{i,w_i}^\tri(X_\tri)
\end{align*}
for some polynomial $F^\tri_{i,w_i}$ in the coordinates $X_\alpha^\tri$ with constant term $1$. 
Recall that the integral $\A$-lamination satisfies the integrality condition $\sfa_\alpha(L) \in \bZ$. Hence the product $\mathbb{I}_\A(L)$ only has integral exponents in the coordinates $X_\alpha^\tri$. Since the above argument applies for any ideal triangulation $\tri$, it follows that $\mathbb{I}_\A(L)$ is a universally Laurent polynomial, hence it belongs to $\cO(\X_\Sigma)=\cO(\P_{PGL_2,\Sigma})$. Thus the assertion is proved.
\end{proof}

\begin{rem}
Our construction is essentially the restriction of the construction given in \cite[Section 10.3]{GS15} to the integral $\A$-laminations. Indeed, their function $\Delta_\beta$ is exactly our function $\Delta_{22}(g_{[\beta]})$ if we reinterpret it by identifying the $PGL_2$-version of their moduli space $\mathrm{Loc}_{SL_2,\Sigma}$ with $\P_{PGL_2,\Sigma}$ (cf.~\cite[Remark 3.9]{IOS22} and the proof of \cref{prop:duality_compatible} below). 
\end{rem}

\begin{thm}\label{thm:X_basis}
Assume that $\Sigma$ is unpunctured, having at least two marked points. Then the functions $\mathbb{I}_\A(L)$, where $L$ runs over all the integral $\A$-laminations, form a linear basis of the function algebra $\cO(\X_\Sigma)=\cO(\P_{PGL_2,\Sigma})$. 
\end{thm}


We prove this theorem based on the results on the skein algebras \cite{IY,IKar}.
Let $\Sigma$ be an unpunctured marked surface, and $\Skein{\Sigma}^q(\bB)$ the \emph{stated skein algebra} on $\Sigma$. It consists of $\bZ_q$-linear combinations of framed tangles in $\Sigma \times [0,1]$, whose ends lie in $\partial\Sigma \times [0,1]$ and are equipped with states $\{1,2\}$, modulo certain relations. See \cite{TTQLe18} for a detail, where the states $+,-$ \emph{loc.~sit.}~corresponds to our states $1,2$, respectively.
Let $\mathcal{I}_{\mathrm{bad}} \subset \Skein{\Sigma}^q(\bB)$ denote the ideal generated by \emph{bad arcs}, which are peripheral tangles around a special point with particular states:
\begin{align*}
\begin{tikzpicture}
    \coordinate (P) at (-0.5,0) {};
    \coordinate (P') at (0.5,0) {};
    \coordinate (C) at (0,0.5) {};
    \draw[very thick, red] (P) to[out=north, in=west] (C) to[out=east, in=north] (P');
    \draw(P)++(0,-0.4) node[red,scale=0.8]{$1$};
    \draw(P')++(0,-0.4) node[red,scale=0.8]{$2$};
    \draw[dashed] (1,0) arc (0:180:1cm);
    \bline{-1,0}{1,0}{0.2}
    \draw[fill=black] (0,0) circle(2pt);
\end{tikzpicture}
\end{align*}
The quotient 
\begin{align*}
    \Skeinr{\Sigma}^q(\bB) := \Skein{\Sigma}^q(\bB)/\mathcal{I}_{\mathrm{bad}}
\end{align*}
is called the \emph{reduced stated skein algebra}. We denote its classical specialization $q^{1/2}=1$ by $\Skeinr{\Sigma}^1(\bB)$.
We have the following:

\begin{thm}[{\cite{CL19,IY}}]
We have an isomorphism of $\bC$-algebras
\begin{align}\label{eq:O=S}
    \cO(\A_{SL_2,\Sigma}^\times) \cong \Skeinr{\Sigma}^1(\bB)\otimes \bC, \quad \Delta_{ij}(g_{[c]}) \mapsto \tau([c])_{ij}
\end{align}
where the matrix coefficient $\Delta_{ij}(g_{[c]})$ of the Wilson line along an arc class $[c]$ corresponds to a framed tangle $\tau([c])$ that projects to $[c]$ together with the state $i$ (resp. $j$) on its initial (resp. terminal) end. 
\end{thm}
This theorem follows from \cite[Theorem 8.12]{CL19} by taking the following observation into account: the Wilson line along a peripheral arc class around a special point is a triangular matrix by \cref{thm:LR-formula}. The bad arcs correspond to the vanishing entries of these triangular matrices. 

Now we want to restrict the isomorphism \eqref{eq:O=S} to the subalgebra $\cO(\P_{PGL_2,\Sigma}) \subset \cO(\A_{SL_2,\Sigma}^\times)$ of index $2$. Let $\Skeinr{\Sigma}^q(\bB)_{\mathrm{cong}} \subset \Skeinr{\Sigma}^q(\bB)$ be the \emph{congruent subalgebra} \cite{IKar} of the reduced stated skein algebra, which is generated by \emph{congruent} (or \emph{even}) tangles. Here a (stated) tangle is said to be \emph{congruent} with respect to a given triangulation $\tri$ if its geometric intersection with each edge $\alpha \in e(\tri)$ is even. This condition turns out to be independent of triangulations, and invariant under the isotopy and skein relations. 


\begin{thm}\label{thm:O=S_even}
The isomorphism \eqref{eq:O=S} restricts to an isomorphism
\begin{align}\label{eq:O=S_even}
    \cO(\P_{PGL_2,\Sigma}) \cong \Skeinr{\Sigma}^1(\bB)_{\mathrm{cong}}\otimes \bC.
\end{align}
\end{thm}

\begin{proof}
By \cite[Corollary 3.16]{IO20}, the function algebra $\cO(\P_{PGL_2,\Sigma})$ is generated by the matrix coefficients of Wilson lines. The projection $SL_2 \to PGL_2$ induces an embedding $\cO(PGL_2) \to \cO(SL_2)$, whose image is generated by the elements $\Delta_{ij}\Delta_{kl}$ for $i,j,k,l \in \{1,2\}$. Hence the elements $\Delta_{ij}(g_{[c]})\Delta_{kl}(g_{[c]})$ generate $\cO(\P_{PGL_2,\Sigma})$, which are send to elements $\tau([c])_{ij}\tau([c])_{kl} \in \Skeinr{\Sigma}^1(\bB)_{\mathrm{cong}}$ in the congruent subalgebra. 

Conversely, each element $W \in \Skeinr{\Sigma}^1(\bB)_{\mathrm{cong}}$ corresponds to a certain polynomial $F_W \in \cO(\A_{SL_2,\Sigma}^\times)$ of matrix entries of Wilson lines valued in $SL_2$. Then by the same argument as in the proof of \cref{lem:function_even}, one can verify that $F_W$ actually lies in the subalgebra $\cO(\P_{PGL_2,\Sigma})$, thanks to the congruent condition of $W$. Thus the assertion is proved.
\end{proof}

\begin{proof}[Proof of \cref{thm:X_basis}]
The functions $\mathbb{I}_\A(L) \in \cO(\P_{PGL_2,\Sigma})$ are classical counterparts of the elements in the $\bZ_q$-basis $\mathsf{B}_{\mathrm{cong}}(\Sigma) \subset \Skeinr{\Sigma}^q(\bB)_{\mathrm{cong}}$ constructed in \cite{IKar}. Then the assertion follows from \cref{thm:O=S_even}.
\end{proof}

\begin{rem}
    Without referring to the forthcoming result in \cite{IKar}, the linear independence of the elements $\mathbb{I}_\A(L) \in \cO(\P_{PGL_2,\Sigma})$ also follows from \cref{prop:duality_compatible} below and the linear independence of the bracelets basis.  
\end{rem}

\subsection{The basis of $\cO(\A_\Sigma)$ parametrized by the integral $\P$-laminations}
Let $\Sigma$ be an unpunctured marked surface. In this case, we do not need the data of lamination signature. We basically follow \cite[Definition 12.4]{FG06} with an extra assignment for pinnings. In particular, we lift each loop $\gamma$ to the punctured tangent bundle $T'\Sigma$, and understand the trace function $\mathrm{Tr}_{[\gamma]}$ on $\A_{SL_2,\Sigma}$ as the trace of monodromy of twisted $SL_2$-local systems along this lift. We also use the following shifting operation on the curves. Compare with \cref{def:shift_ideal}.

\begin{dfn}[negative $\bM$-shift of curves]\label{def:shift_curve}
For a curve $\gamma$ in $\Sigma$ having its endpoints on $\partial^\ast\Sigma$, we define its \emph{(negative) $\bM$-shift} to be the ideal arc $\gamma^{\bM}$ obtained from $\gamma$ by shifting its endpoints to the nearest special point in the negative direction along the boundary.
See \cref{fig:shifting_curve}.
\end{dfn}

\begin{figure}[ht]
    \centering
\begin{tikzpicture}
\fill[gray!20] (0,1.5) -- (-0.2,1.5) -- (-0.2,-1.5) -- (0,-1.5) --cycle;
\fill[gray!20] (4,1.5) -- (4+0.2,1.5) -- (4+0.2,-1.5) -- (4,-1.5) --cycle;
\draw[thick] (0,1.5) -- (0,-1.5);
\draw[thick] (4,-1.5) -- (4,1.5);
\filldraw(0,1) circle(1.5pt); 
\filldraw(0,0) circle(1.5pt);
\filldraw(0,-1) circle(1.5pt);
\filldraw(4,1) circle(1.5pt);
\filldraw(4,0) circle(1.5pt);
\filldraw(4,-1) circle(1.5pt);
\draw[red,thick] (0,-0.5) to[out=0,in=180] node[midway,above]{$\gamma$} (4,0.5);
\begin{scope}[xshift=6cm]
\fill[gray!20] (0,1.5) -- (-0.2,1.5) -- (-0.2,-1.5) -- (0,-1.5) --cycle;
\fill[gray!20] (4,1.5) -- (4+0.2,1.5) -- (4+0.2,-1.5) -- (4,-1.5) --cycle;
\draw[thick] (0,1.5) -- (0,-1.5);
\draw[thick] (4,-1.5) -- (4,1.5);
\filldraw(0,1) circle(1.5pt); 
\filldraw(0,0) circle(1.5pt);
\filldraw(0,-1) circle(1.5pt);
\filldraw(4,1) circle(1.5pt);
\filldraw(4,0) circle(1.5pt);
\filldraw(4,-1) circle(1.5pt);
\draw[red,thick] (0,0) to[out=0,in=180] node[midway,above]{$\gamma^{\bM}$} (4,0);
\end{scope}
\end{tikzpicture}
    \caption{The negative $\bM$-shift of a curve.}
    \label{fig:shifting_curve}
\end{figure}
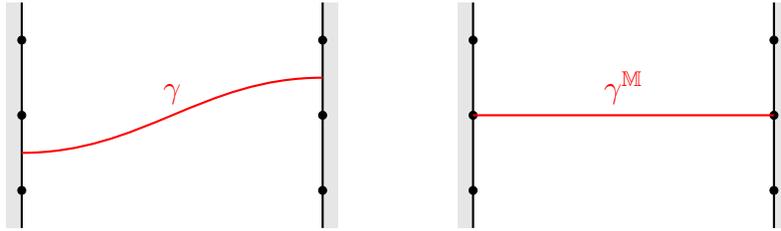
The two shifting operations are related by $(\alpha_\bB)^{\bM}=\alpha$ for an ideal arc $\alpha$, and $(\gamma^\bM)_\bB=\gamma$ for a curve having its endpoints on $\partial^\ast\Sigma$. The following is a slight enhancement of the construction given in \cite[Section 12.3]{FG06} and \cite[Section 7.2]{FG07}:



\begin{dfn}\label{def:skein_lift_X}
Given an integral $\P$-lamination $(L=\{(\gamma_i,w_i)\},\nu) \in \cL^p(\Sigma,\bZ)$, we define the corresponding function $\mathbb{I}_\X(L,\nu)\in \cO(\A^\times_{SL_2,\Sigma})$, as follows.
\begin{itemize}
    \item For each weighted loop $(\gamma_i,w_i)$, associate the trace function
    \begin{align*}
        \mathrm{Tr}_{[\gamma_i]^{w_i}} \in \cO(\A_{SL_2,\Sigma}),
    \end{align*}
    where $[\gamma_i]^{w_i} \in \pi_1(T'\Sigma)$ denotes the $w_i$-th power of a based loop homotopic to (the lift of) $\gamma_i$. 
    \item For each weighted non-peripheral arc $(\gamma_i,w_i)$, associate the function
    \begin{align*}
        (A_{\gamma^\bM})^{w_i} \in \cO(\A_{SL_2,\Sigma}).
    \end{align*}
    \item For each boundary interval $\alpha \in \bB$, associate the function
    \begin{align}\label{eq:duality_pinning}
        A_\alpha^{\nu_\alpha} \in \cO(\A_{SL_2,\Sigma}^\times).
    \end{align}
\end{itemize}
Then the function $\mathbb{I}_\X(L,\nu) \in \cO(\A_{SL_2,\Sigma}^\times)$ is defined to be the product of these elements. 
\end{dfn}
The map $\mathbb{I}_\X:\X_\Sigma(\bZ^{\mathsf{T}}) \to \cO(\A_{SL_2,\Sigma}^\times)$ is clearly $MC(\Sigma)$-equivariant.
Via the isomorphism $\cO(\A_\Sigma) \cong \cO(\A_{SL_2,\Sigma}^\times)$, we have the following:

\begin{thm}[Musiker--Schiffler--Williams {\cite[Theorem 1.1 and Corollary 1.3]{MSW}}]
Suppose that $\Sigma$ has at least two marked points. 
Then the functions $\mathbb{I}_\X(L,\nu)$, where $(L,\nu)$ runs over all the integral $\P$-laminations, form a linear basis of the upper cluster algebra $\cO(\A_{\Sigma})$. 
\end{thm}

\begin{rem}
The construction can be generalized to any marked surface so that it is equivariant under the $\bZ/2\bZ$-action at each puncture which alternates the lamination signature and the tag. See \cite[Section 12.6]{FG06}. 
\end{rem}

\subsection{Ensemble compatibility of duality maps}
We are going to discuss the compatibility of the two constructions of duality maps with respect to the structure of cluster ensemble. It turns out that a non-trivial Langlands duality comes into play.

\paragraph{\textbf{Langlands dual coordinates on $\cL^p(\Sigma,\bQ)$.}}
For an ideal triangulation $\tri$ of $\Sigma$, we define the \emph{Langlands dual coordinates}
\begin{align}\label{eq:coordinate_dual}
    \check{\sfx}_\tri=(\check{\sfx}^\tri_\alpha)_{\alpha \in e(\tri)}: \cL^p(\Sigma,\bQ) \xrightarrow{\sim} \bQ^{e(\tri)},
\end{align}
as follows. For $\alpha \in e_{\interior}(\tri)$, let $\check{\sfx}^\tri_\alpha:=\sfx^\tri_\alpha$. We modify the frozen coordinates $\check{\sfx}^\tri_\alpha$, $\alpha \in \bB$ into
\begin{align*}
    \check{\sfx}^\tri_\alpha(L,\sigma_L,\nu) := \nu_\alpha +\sum_{j} w_j (\alpha:\widehat{\gamma}_j)^\vee,
\end{align*}
where $(\alpha:\widehat{\gamma}_j)^\vee:= +1$ if $\widehat{\gamma}_j$ contains a corner arc around the \underline{terminal} marked point $m^-_\alpha$ as its portion, and otherwise $0$. Compare with \eqref{eq:boundary_coord}. 
We define the \emph{Langlands dual ensemble map}
\begin{align}\label{eq:dual_ensemble}
    \check{p}_\Sigma^{\mathsf{T}}: \cL^a(\Sigma,\bQ) \to \cL^p(\Sigma,\bQ)
\end{align}
by forgetting the peripheral components, and defining the pinning $\nu_\alpha \in \bZ$ to be the weight of the peripheral component around the \underline{terminal} marked point $m^-_\alpha$. 
Then similarly to \cref{prop:p-lamination,prop:ensemble_tropical}, we get:
\begin{thm}
\begin{enumerate}
    \item For any ideal triangulation $\tri$ of $\Sigma$, the map \eqref{eq:coordinate_dual} gives a bijection. The coordinate transformations are again tropical cluster Poisson transformations.
    \item For any ideal triangulation $\tri$ of $\Sigma$, we have
    \begin{align}\label{eq:ensemble_dual}
    (\check{p}_\Sigma^\mathsf{T})^\ast \check{\sfx}_\kappa^\tri = \sum_{\alpha \in e(\tri)} (\ve_{\kappa\alpha}^\tri- m_{\kappa\alpha}) \sfa_\alpha
\end{align}
for all $\kappa \in e(\tri)$.
\end{enumerate}
\end{thm}
Observe that for the square matrix $p^\tri:=(\ve_{\kappa\alpha}^\tri+ m_{\kappa\alpha})_{\kappa,\alpha\in e(\tri)}$, its \emph{Langlands dual} \cite[Section 1.2.10]{FG09} is $(-p^\tri)^\top = (\ve_{\kappa\alpha}^\tri- m_{\kappa\alpha})_{\kappa,\alpha\in e(\tri)}$.

The following property shows that our assignment rule in \cref{def:skein_lift_X} satisfies one of the axioms of Fock--Goncharov duality with respect to this dual coordinates:

\begin{prop}\label{prop:cluster_monomial}
Suppose that an integral $\P$-lamination $(L,\nu)$ satisfies $\check{\sfx}_\alpha:=\check{\sfx}^\tri_\alpha(L,\nu) \geq 0$ for all $\alpha \in e(\tri)$ for some ideal triangulation $\tri$. Then we have
\begin{align*}
    \mathbb{I}_\X(L,\nu) = \prod_{\alpha \in e(\tri)} A_\alpha^{\check{\sfx}_\alpha}.
\end{align*}
\end{prop}

\begin{proof}
Such an integral $\P$-lamination is given by $L=\{(\gamma_\alpha,\check{\sfx}_\alpha)\}_{\alpha \in e_{\interior}(\tri)}$ such that $(\gamma_\alpha)^\bM=\alpha$, together with the pinnings $\nu_\beta:=\check{\sfx}_\beta$ for $\beta \in \bB$. Then we get
\begin{align*}
    \mathbb{I}_\X(L,\nu) = \prod_{\alpha \in e_{\interior}(\tri)} A_\alpha^{\check{\sfx}_\alpha} \cdot \prod_{\beta \in \bB} A_\beta^{\check{\sfx}_\beta} = \prod_{\alpha \in e(\tri)} A_\alpha^{\check{\sfx}_\alpha},
\end{align*}
as desired. 
\end{proof}

\begin{rem}
In the original coordinates $\sfx_\tri$, the $\P$-laminations in the negative cone $\sfx_\alpha^\tri \leq 0$ corresponds to the positive $\bM$-shifts (defined with the opposite direction) of ideal arcs $\alpha \in e_{\interior}(\tri)$ and negative pinnings on boundary intervals. They give rise to cluster monomials $\prod_{\alpha \in e(\tri)}A_\alpha^{-\sfx_\alpha^\tri(L,\nu)}$.  
\end{rem}

The two constructions of duality maps are compatible in the following sense:

\begin{thm}[Ensemble compatibility of duality maps]\label{prop:duality_compatible}
For any unpunctured marked surface $\Sigma$, the following diagram commutes:
\begin{equation}\label{eq:duality_compatible}
    \begin{tikzcd}
    \A_{\Sigma}(\bZ^{\mathsf{T}}) \ar[d,"\check{p}_\Sigma^{\mathsf{T}}"'] \ar[rr,"\mathbb{I}_\A"] && \cO(\X_{\Sigma}) \ar[d,"p_\Sigma^\ast"] \\
    \X_{\Sigma}(\bZ^{\mathsf{T}}) \ar[rr,"\mathbb{I}_\X"'] && \cO(\A_{\Sigma}),
    \end{tikzcd}
\end{equation}
where we use the Langlands dual ensemble map $ \check{p}_\Sigma^{\mathsf{T}}: \cL^a(\Sigma,\bZ) \to \cL^p(\Sigma,\bZ)$ on the tropical side.
\end{thm}

\begin{proof}
Let $L \in \A_\Sigma(\bZ^\mathsf{T})$ be an integral $\A$-lamination. It suffices to consider the case where $L$ consists of a single weighted curve $(\gamma,k)$. 
\begin{itemize}
    \item If $\gamma$ is a non-peripheral loop, the assertion is obvious.
    \item If $\gamma$ is a non-peripheral arc, then $\check{p}_\Sigma^\mathsf{T}(\gamma,k)$ is the same weighted arc. Then we need the equality 
     \begin{align*}
        \Delta_{22}(g_{[\gamma]})^k=(A_{\gamma^\bM})^k.
    \end{align*}
    Since $\gamma=(\gamma^\bM)_\bB$, 
    it is exactly the formula given in \cref{prop:Wilson_lambda} (2). 
    \item If $\gamma$ is a peripheral arc around a special point $m \in \bM$, then $\check{p}_\Sigma^\mathsf{T}(\gamma,k)$ consists of an empty collection of curves together with the pinning $\nu_\alpha=k$ assigned to the boundary interval $\alpha \in \bB$ such that $m=m^-_\alpha$ as its terminal endpoint:
    \begin{align*}
    \begin{tikzpicture}
    \bline{-0.5,0}{2.5,0}{0.2}
    \draw[red,thick] (2,0) arc(0:180:0.5cm) node[midway,above]{$(\gamma,k)$};
    \node at (1.5,-0.4) {$m$};
    \draw[fill=black] (0,0) circle(2pt);
    \draw[fill=black] (1.5,0) circle(2pt);
    \draw[|->,thick] (3,0.5) --node[midway,above]{$\check{p}_\Sigma^{\mathsf{T}}$} (4,0.5);
    \begin{scope}[xshift=5cm]
    \bline{-0.5,0}{2.5,0}{0.2}
    \node at (0.75,0.3) {$\alpha$};
    \node[red] at (0.75,-0.5) {$\nu_\alpha=k$};
    \draw[fill=black] (0,0) circle(2pt);
    \draw[fill=black] (1.5,0) circle(2pt);
    \end{scope}
    \end{tikzpicture}
    \end{align*}
    Then we need the equality 
    \begin{align*}
        \Delta_{22}(g_{[\gamma]})^k=A_\alpha^{k}. 
    \end{align*}
    Since $\gamma=\alpha_\bB$, 
    it follows from \cref{prop:Wilson_lambda} (2). 
\end{itemize}
Thus the assertion follows from the multiplicative nature of the both constructions. 
\end{proof}


\begin{rem}\label{rem:duality_constraint}
It is easily verified that the requirements for the (lowest term) exponents in \cref{lem:function_even,prop:cluster_monomial} are compatible only if the exponents/coefficients of the maps $p_\Sigma$ and $\check{p}^{\mathsf{T}}_\Sigma$ are Langlands dual to each other. In particular, it is an algebraic matter independent of the topological construction.
\end{rem}

\subsection{Amalgamation of bracelets bases}
Let us investigate the behavior of the duality maps $\mathbb{I}_\X$ under the tropical gluing maps studied in \cref{subsec:lamination_amalgamation}. Let us first modify the gluing map $q^\mathsf{T}_{\Sigma,\Sigma'}$ to its Langlands dual so that it is compatible with the coordinates $\check{\sfx}_\tri$. 

Let $\Sigma'$ be obtained by $\Sigma$ by gluing two boundary intervals $\alpha_L,\alpha_R$. 
We define the \emph{Langlands dual gluing map}
\begin{align*}
    \check{q}^\mathsf{T}_{\Sigma,\Sigma'}: \cL^p(\Sigma,\bZ) \to \cL^p(\Sigma',\bZ)    
\end{align*}
similarly to $q^\mathsf{T}_{\Sigma,\Sigma'}$, but replace the parametrization $\psi_Z$ with the one $\check{\psi}_{Z}:\bR \to \alpha_Z$ so that $\check{\psi}_{Z}(\frac{1}{2}+\bZ)=S_{Z}$, and $\check{\psi}_Z(\bR_{>0}) \cap S_Z$ consists of all the endpoints of the additional peripheral curves around the \underline{terminal} marked point $m^-_{\alpha_Z}$ for $Z \in \{L,R\}$. Then the same property as in \cref{thm:amalgamation} with the dual coordinates $\check{\sfx}_\tri$ holds. 

On the moduli side, we have the restriction morphism $\mathrm{Res}_{\Sigma',\Sigma}: \A_{SL_2,\Sigma'} \to \A_{SL_2,\Sigma}$.
It induces an algebra homomorphism
\begin{align*}
    \mathrm{Res}^\ast_{\Sigma',\Sigma}: \cO(\A^\times_{SL_2,\Sigma}) \to \cO(\A^\times_{SL_2,\Sigma'})[A_{\overline{\alpha}}^{-1}],
\end{align*}
which satisfies
\begin{align}\label{eq:res_A-variable}
    \mathrm{Res}^\ast_{\Sigma',\Sigma}(A_{\alpha_L}) = \mathrm{Res}^\ast_{\Sigma',\Sigma}(A_{\alpha_R}) = A_{\overline{\alpha}}. 
\end{align}
Let us consider the diagram
\begin{equation}\label{eq:duality_amalgamation}
    \begin{tikzcd}
    \X_{\Sigma}(\bZ^{\mathsf{T}}) \ar[d,"\check{q}^\mathsf{T}_{\Sigma,\Sigma'}"'] \ar[rr,"\mathbb{I}_\X"] && \cO(\A^\times_{\Sigma}) \ar[d,"\mathrm{Res}^\ast_{\Sigma',\Sigma}"] \\
    \X_{\Sigma'}(\bZ^{\mathsf{T}}) \ar[rr,"\mathbb{I}_\X"'] && \cO(\A^{\times}_{\Sigma'})[A_{\overline{\alpha}}^{-1}].
    \end{tikzcd}
\end{equation}

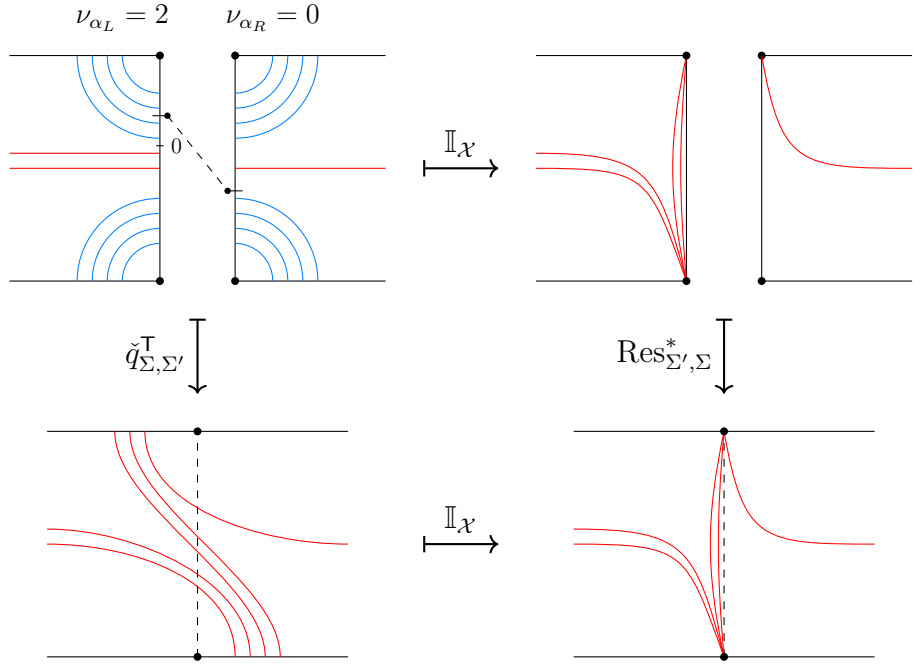
\begin{figure}[ht]
    \centering
\begin{tikzpicture}
\begin{scope}
\draw (-2,0) -- (0,0) -- (0,3) -- (-2,3);
\foreach \i in {0.5,0.7,0.9,1.1}
{
\draw[myblue] (0,0)++(0,\i) arc(90:180:\i);
\draw[myblue] (0,3)++(0,-\i) arc(-90:-180:\i);
}
\foreach \i in {0,1} \foreach \j in {0,3} \fill(\i,\j) circle(1.5pt);
\draw[red] (-2,1.5) -- (0,1.5);
\draw[red] (-2,1.7) -- (0,1.7);
\draw (-0.05,1.8) -- (0.05,1.8) node[right,scale=0.7]{$0$};
\pinn{0,2.2}{180}{0.1}{0.03cm}
\draw[dashed] (0.1,2.2) -- (0.9,1.2);
\node[scale=0.9] at (-0.5,3.5) {$\nu_{\alpha_L}=2$};
\end{scope}
\begin{scope}[xshift=1cm]
\draw (2,0) -- (0,0) -- (0,3) -- (2,3);
\foreach \i in {0.5,0.7,0.9,1.1}
{
\draw[myblue] (0,0)++(0,\i) arc(90:0:\i);
\draw[myblue] (0,3)++(0,-\i) arc(-90:0:\i);
}
\draw[red] (2,1.5) -- (0,1.5);
\pinn{0,1.2}{0}{0.1}{0.03cm}
\node[scale=0.9] at (0.5,3.5) {$\nu_{\alpha_R}=0$};
\end{scope}

\draw[thick,|->] (0.5,-0.5) --node[midway,left]{$\check{q}^{\mathsf{T}}_{\Sigma,\Sigma'}$}++(0,-1);
\draw[thick,|->] (3.5,1.5) --node[midway,above]{$\mathbb{I}_\X$}++(1,0);

\begin{scope}[xshift=7cm]
\draw (-2,0) -- (0,0) -- (0,3) -- (-2,3);
\draw[red] (-2,1.5) ..controls++(1.5,0) and (110:1.5).. (0,0);
\draw[red] (-2,1.7) ..controls++(1.5,0) and (105:1.5).. (0,0);
\draw[red] (0,0) to[bend left=5pt] (0,3);
\draw[red] (0,0) to[bend left=12pt] (0,3);
\foreach \i in {0,1} \foreach \j in {0,3} \fill(\i,\j) circle(1.5pt);
\draw[thick,|->] (0.5,-0.5) --node[midway,left]{$\mathrm{Res}^\ast_{\Sigma',\Sigma}$}++(0,-1);
\end{scope}
\begin{scope}[xshift=8cm]
\draw (2,0) -- (0,0) -- (0,3) -- (2,3);
\draw[red] (2,1.5) ..controls++(-1.5,0) and ($(0,3)+(-80:1.5)$).. (0,3);
\end{scope}

\begin{scope}[yshift=-5cm,xshift=0.5cm]
\draw (-2,0) -- (2,0);
\draw (-2,3) -- (2,3);
\draw[dashed] (0,0) -- (0,3);
\draw[red] (-2,1.5) ..controls++(1.2,0) and ($(0.5,0)+(0,1)$).. (0.5,0);
\draw[red] (-2,1.7) ..controls++(1.2,0) and ($(0.7,0)+(0,1)$).. (0.7,0);
\draw[red] (2,1.5) ..controls++(-1.2,0) and ($(-0.7,3)+(0,-1)$).. (-0.7,3);
\draw[red] (1.1,0) ..controls++(0,1) and ($(-0.9,3)+(0,-1)$).. (-0.9,3); 
\draw[red] (0.9,0) ..controls++(0,1) and ($(-1.1,3)+(0,-1)$).. (-1.1,3); 
\foreach \j in {0,3} \fill(0,\j) circle(1.5pt);
\draw[thick,|->] (3,1.5) --node[midway,above]{$\mathbb{I}_\X$}++(1,0);
\end{scope}

\begin{scope}[yshift=-5cm,xshift=7.5cm]
\draw (-2,0) -- (2,0);
\draw (-2,3) -- (2,3);
\draw[dashed] (0,0) -- (0,3);
\draw[red] (-2,1.5) ..controls++(1.5,0) and (110:1.5).. (0,0);
\draw[red] (-2,1.7) ..controls++(1.5,0) and (105:1.5).. (0,0);
\draw[red] (2,1.5) ..controls++(-1.5,0) and ($(0,3)+(-80:1.5)$).. (0,3);
\draw[red] (0,0) to[bend left=5pt] (0,3);
\draw[red] (0,0) to[bend left=12pt] (0,3);
\foreach \j in {0,3} \fill(0,\j) circle(1.5pt);
\end{scope}
\end{tikzpicture}
    \caption{Amalgamation of bracelets bases: an example with $\nu_{\alpha_L}+\nu_{\alpha_R}\geq 0$.}
    \label{fig:amal_example1}
\end{figure}

\begin{thm}\label{thm:amal_bracelet}
For any integral $\P$-lamination $(L,\nu) \in \X_\Sigma(\bZ^\mathsf{T})$ 
such that $\nu_{\alpha_L}+\nu_{\alpha_R} \geq 0$, we have $\mathrm{Res}^\ast_{\Sigma',\Sigma}(\mathbb{I}_\X(L,\nu))=\mathbb{I}_\X(\check{q}^\mathsf{T}_{\Sigma,\Sigma'}(L,\nu))$. 
\end{thm}

\begin{proof}
Let $(L',\nu'):=\check{q}^\mathsf{T}_{\Sigma,\Sigma'}(L,\nu) \in \X_{\Sigma'}(\bZ^\mathsf{T})$ and $n:=\nu_{\alpha_L}+\nu_{\alpha_R}$. 
Represent $L$ and $L'$ by curves with weight $1$. 
Observe that
\begin{enumerate}
\item[(1)] the curves in $L$ having endpoints on $\alpha_L$ and $\alpha_R$ give rise to curves in $L'$ ``turning right". In particular, their negative $\bM$-shifts are the same before/after the gluing;
\item[(2)] the new curves in $L'$ arising via the gluing give rise to $n$ parallel copies of the ideal ideal $\overline{\alpha}$.
\end{enumerate}
See \cref{fig:amal_example1} for an illustrating example. Hence we have
\begin{align*}
    \mathrm{Res}^\ast_{\Sigma',\Sigma}(\mathbb{I}_\X(L,\nu)) = \mathrm{Res}^\ast_{\Sigma',\Sigma}\left(A_{\alpha_L}^{\nu_{\alpha_L}}\cdot A_{\alpha_R}^{\nu_{\alpha_R}}\cdot \mathbb{I}_\X(L,\nu') \right) = A_{\overline{\alpha}}^n \cdot \mathbb{I}_\X(L,\nu') = \mathbb{I}_\X(L',\nu').
\end{align*}
Here $(L,\nu')$ denotes the data obtained from $(L,\nu)$ by deleting the pinnings $\nu_{\alpha_L}$ and $\nu_{\alpha_R}$, for which $\mathrm{Res}^\ast_{\Sigma',\Sigma}(\mathbb{I}_\X(L,\nu'))=\mathbb{I}_\X(L,\nu')$ holds from the observation (1). We also used \eqref{eq:res_A-variable} and the observation (2) in the second and third equality, respectively. The first assertion is proved.

\end{proof}

\begin{ex}\label{ex:amal_dominance}
Here is a square example with $\nu_{\alpha_L}+\nu_{\alpha_R}< 0$.
\begin{align*}
\begin{tikzpicture}
\begin{scope}
\draw (0,0) -- (0,3) -- (-2,1.5) -- cycle;
\foreach \i in {0,1} \foreach \j in {0,3} \fill(\i,\j) circle(1.5pt);
\foreach \i in {-2,3} \fill(\i,1.5) circle(1.5pt);
\node[red,scale=0.9] at (-0.3,1.5) {$-1$};
\node[scale=0.9] at (0.3,1) {$\alpha_L$};
\node[scale=0.9] at (-1,2.5) {$\beta$};
\node[scale=0.9] at (-1,0.5) {$\gamma$};
\end{scope}
\begin{scope}[xshift=1cm]
\draw (0,0) -- (0,3) -- (2,1.5)--cycle;
\node[red,scale=0.9] at (0.3,1.5) {$0$};
\node[scale=0.9] at (-0.3,2) {$\alpha_R$};
\node[scale=0.9] at (1,2.5) {$\epsilon$};
\node[scale=0.9] at (1,0.5) {$\delta$};
\end{scope}
\draw[thick,|->] (3.5,1.5) --node[midway,above]{$\check{q}^{\mathsf{T}}_{\Sigma,\Sigma'}$}++(1,0);
\begin{scope}[xshift=7cm]
\draw (-2,1.5) -- (0,0) -- (2,1.5) -- (0,3) --cycle;
\draw[dashed] (0,0) -- (0,3);
\draw[dashed] (-2,1.5) -- (2,1.5);
\draw[red] (-1,0.75) ..controls++(50:1) and ($(1,3-0.75)+(-110:1)$).. (1,3-0.75); 
\node[scale=0.9] at (-0.3,2.3) {$\overline{\alpha}$};
\node[scale=0.9] at (1,1.3) {$\overline{\alpha}'$};
\node[scale=0.9] at (-1,2.5) {$\beta$};
\node[scale=0.9] at (-1,0.5) {$\gamma$};
\node[scale=0.9] at (1,2.5) {$\epsilon$};
\node[scale=0.9] at (1,0.5) {$\delta$};
\foreach \j in {0,3} \fill(0,\j) circle(1.5pt);
\foreach \i in {-2,2} \fill(\i,1.5) circle(1.5pt);
\end{scope}
\end{tikzpicture} 
\end{align*}
Let us consider $(L,\nu)$ as shown in the left, the empty lamination with the pinning $\nu_{\alpha_L}=-1$ and $\nu_{\alpha_R}=0$, which produces a lamination $(L',\nu')$ shown in the right after the gluing. Then we have $\mathbb{I}_\X(L,\nu)=A_{\alpha_L}^{-1}$, while
\begin{align*}
    \mathbb{I}_\X(L',\nu')=A_{\overline{\alpha}'} = \frac{A_\beta A_\delta +A_\gamma A_\epsilon}{A_{\overline{\alpha}}} = \frac{A_\gamma A_\epsilon}{A_{\overline{\alpha}}}\cdot(1+ p_\Sigma^\ast X_{\overline{\alpha}}).
\end{align*}
Observe that ignoring the frozen term $A_\gamma A_\epsilon$, the function $\mathrm{Res}^\ast_{\Sigma',\Sigma}(\mathbb{I}_\X(L,\nu))=A_{\overline{\alpha}}^{-1}$ coincides with one of the terms in $\mathbb{I}_\X(L',\nu')$.
\end{ex}
\begin{rem}\label{rem:amal_weak}
In general, one can verify a weak statement that $\mathrm{Res}^\ast_{\Sigma',\Sigma}(\mathbb{I}_\X(L,\nu))$ corresponds to the highest term in $\mathbb{I}_\X(\check{q}^\mathsf{T}_{\Sigma,\Sigma'}(L,\nu))$ with respect to the \emph{dominance order} \cite{Qin21}, up to certain frozen variables of $\Sigma'$. 
In the quantum setting based on the skein algebra \cite{Muller}, the term $A_{\overline{\alpha}}^{|n|}\cdot \mathrm{Res}^\ast_{\Sigma',\Sigma}(\mathbb{I}_\X(L,\nu))$ appears in the expansion of the product $A_{\overline{\alpha}}^{|n|}\cdot \mathbb{I}_\X(\check{q}^\mathsf{T}_{\Sigma,\Sigma'}(L,\nu))$ 
in the graphical basis as the term of highest $q$-exponent, up to frozens.
\end{rem}

%% file: 5_appendix.tex
\appendix
\section{Cluster varieties associated to a marked surface}\label{app:cluster}
The reader is referred to \cite{FG09} for details. 
Let $\Sigma$ be a marked surface. For each ideal (or tagged) triangulation $\tri$ of $\Sigma$, let $\X_\tri=(\bG_m)^{e(\tri)}$, $\A_\tri=(\bG_m)^{e(\tri)}$ be two split algebraic tori, where $\bG_m:=\mathop{\mathrm{Spec}} \bC[t,t^{-1}]=\bC^\ast$ denotes the multiplicative group scheme over $\bC$. Let $(X_\alpha^\tri)_{\alpha \in e(\tri)}$, $(A_\alpha^\tri)_{\alpha \in e(\tri)}$ denote the standard coordinate systems on $\X_\tri$ and $\A_\tri$, respectively. These tori are accompanied with the \emph{exchange matrix} $\ve^\tri=(\ve_{\alpha\beta}^\tri)_{\alpha,\beta \in e(\tri)}$ \cite{FST}, defined as follows: for each triangle $T$ of $\tri$, let
\begin{align*}
    \ve_{\alpha\beta}(T):= \begin{cases}
        1 & \mbox{if $T$ has $\alpha$ and $\beta$ as its consecutive edges in the clockwise order}, \\
        -1 & \mbox{if the same holds with the counter-clockwise order}, \\
        0 & \mbox{otherwise}.
    \end{cases}
\end{align*}
Then we set $\ve_{\alpha\beta}^\tri:=\sum_T \ve_{\alpha\beta}(T)$,
where $T$ runs over all non-self-folded triangles of $\tri$. When $\tri$ has a self-folded triangle or it is tagged, $\ve_{\alpha\beta}^\tri$ is appropriately modified. See \cite{FST}.  
Then the \emph{cluster Poisson/$K_2$-varieties} \cite{FG09} are defined to be
\begin{align*}
    \X_\Sigma:= \bigcup_{\tri} \X_\tri, \quad \A_\Sigma:= \bigcup_{\tri} \A_\tri,
\end{align*}
where the gluing data is given by the birational isomorphisms
\begin{align}
    &\mu_{\kappa}^x: \X_\tri \to \X_{\tri'}, \quad
    (\mu_{\kappa}^x)^\ast X'_\alpha= 
    \begin{cases}
      X_\kappa^{-1} & (\alpha=\kappa'), \\
      X_\alpha (1+ X_\kappa^{-\mathrm{sgn}(\ve_{\alpha\kappa})})^{-\ve_{\alpha\kappa}} & (\alpha \neq \kappa'),
  \end{cases} \label{eq:X-transf} \\
  &\mu_{\kappa}^a: \A_\tri \to \A_{\tri'}, \quad
  (\mu_{\kappa}^a)^\ast A'_\alpha= 
  \begin{cases}A_\kappa^{-1} \big(\prod_{\beta\in e(\tri)} A_\beta^{[\ve_{\kappa\beta}]_+}  + \prod_{\beta\in e(\tri)} A_\beta^{[-\ve_{\kappa\beta}]_+} \big) & (\alpha = \kappa'),
  \\ 
  A_\alpha & (\alpha \neq \kappa'),
  \end{cases} \label{eq:A-transf}
\end{align}
for each flip $\tri'=\tri \setminus \{\kappa\} \cup \{\kappa'\}$ along $\kappa \in e_{\interior}(\tri)$. Here $\sgn (x) \in \{+,0,-\}$ denotes the sign, and $[x]_+:=\max\{0,x\}$ for $x \in \bR$. 
We abbreviated as $X_\alpha:=X_\alpha^\tri$, $X'_\alpha:=X_{\alpha}^{\tri'}$, $\ve_{\alpha\beta}:=\ve_{\alpha\beta}^\tri$, and so on. The maps \eqref{eq:X-transf}, \eqref{eq:A-transf} are called the \emph{cluster Poisson/$K_2$}-transformations, respectively. From the definition, their function algebras are given by
\begin{align*}
    \cO(\X_\Sigma) = \bigcap_{\tri} \bC[(X_\alpha^\tri)^{\pm 1} \mid \alpha \in e(\tri)], \quad \cO(\A_\Sigma) = \bigcap_{\tri} \bC[(A_\alpha^\tri)^{\pm 1} \mid \alpha \in e(\tri)].
\end{align*}
In other words, these algebras consists of \emph{universally Laurent polynomials}. The function algebras $\cO(\X_\Sigma)$ and $\cO(\A_\Sigma)$ are also called the \emph{cluster Poisson algebra} and the \emph{upper cluster algebra}, respectively.  

\paragraph{\textbf{Ensemble maps}}
The exchange matrix $\ve^\tri$ induces the monomial map
\begin{align*}
    p_{\uf}: \A_\tri \to \X_\tri^{\uf}, \quad p_{\uf}^\ast X_\kappa^\tri = \prod_{\alpha \in e(\tri)} (A^\tri_\alpha)^{\ve^\tri_{\kappa\alpha}} \quad (\kappa \in e_{\interior}(\tri)),
\end{align*}
which we call the \emph{ensemble map}. 
Here $\X_\tri^{\uf}:=(\bG_m)^{e_{\interior}(\tri)}$ denotes the cluster torus without frozen coordinates.\footnote{This restriction comes from the fact that the entries $\ve_{ij}$ for $i,j$ frozen are allowed to be rational in a general cluster variety.} It commutes with cluster transformations, and hence induces a morphism $p_{\uf}: \A_\Sigma \to \X_\Sigma^{\uf}$. 
We have a freedom to choose its extension
\begin{align*}
    p_{\Sigma;M}: \A_\tri \to \X_\tri, \quad p_{\Sigma;M}^\ast X_\kappa^\tri = \prod_{\alpha \in e(\tri)} (A^\tri_\alpha)^{\ve^\tri_{\kappa\alpha}+m_{\kappa\alpha}}
\end{align*}
by specifying a constant matrix $M=(m_{\kappa\alpha})_{\kappa,\alpha \in e(\tri)}$ such that $m_{\kappa\alpha}=0$ unless $(\kappa,\alpha) \in \bB \times \bB$ (cf.~\cite[Appendix A]{GHKK} and \cite[Section 18]{GS19}). It also commutes with cluster transformations, and induces a morphism $p_{\Sigma;M}: \A_\Sigma \to \X_\Sigma$. In this paper, following \cite{GS19}, we choose $m_{\kappa\alpha}=\mp \delta_{\kappa\alpha}$ for $(\kappa,\alpha) \in \bB \times \bB$. 

\paragraph{\textbf{Tropicalizations.}}
Let $\mathbb{P}$ be a semifield. Any positive rational map $f: T_1 \to T_2$ between split algebraic tori naturally induces a map $f(\mathbb{P}):T_1(\mathbb{P}) \to T_2(\mathbb{P})$, where $T_i(\mathbb{P}):=\Hom(\bG_m,T_i)\otimes_\bZ \mathbb{P}$, $i=1,2$ are sets of $\mathbb{P}$-valued points. Gluing the coordinate tori by the cluster transformations $\mu_\kappa^x(\mathbb{P})$ and $\mu_\kappa^a(\mathbb{P})$, we get the sets $\X_\Sigma(\mathbb{P})$ and $\A_\Sigma(\mathbb{P})$ of $\mathbb{P}$-points. For example:
\begin{itemize}
    \item if $\mathbb{P}=\pos$ is the semifield of positive real numbers with the usual operations, then $\X_\Sigma(\pos)$ is obtained by gluing $\X_\tri(\pos)=\mathbb{R}^{e(\tri)}_{>0}$ with the same formula as \eqref{eq:X-transf}. 
    \item if $\mathbb{P}=\bZ^{\mathsf{T}}$ is the (max-plus) tropical semifield with the addition $\max$ and multiplication $+$, then $\X_\Sigma(\bZ^{\mathsf{T}})$ is obtained by gluing $\X_\tri(\bZ^{\mathsf{T}})=\bZ^{e(\tri)}$ with the formula obtained from \eqref{eq:X-transf} by replacing the operations as $+ \mapsto \max$, $\times \mapsto +$, which is called the \emph{tropical analogue}. 
\end{itemize}
\paragraph{\textbf{The mapping class group action.}}
Let $MC(\Sigma)$ denote the mapping class group of $\Sigma$. 
Each mapping class $\phi \in MC(\Sigma)$ acts on $\X_\Sigma$ so that $X_\alpha^\tri(\phi(g)) = X_{\phi^{-1}(\alpha)}^{\phi^{-1}(\tri)}(g)$ for all $\tri$ and $\alpha \in e(\tri)$. It acts on $\A_\Sigma$ in a similar manner, and commutes with the (extended) ensemble maps. These actions are positive, and hence descends to the actions on the sets $\X_\Sigma(\mathbb{P})$ and $\A_\Sigma(\mathbb{P})$ of $\mathbb{P}$-points.

%% file: 0_main.bbl
\begin{thebibliography}{GHKK18}

\bibitem[AB20]{AB}
D.~G.~L.~Allegretti and T. Bridgeland,
\emph{The monodromy of meromorphic projective structures},
Trans. Amer. Math. Soc. \textbf{373} (2020), no. 9, 6321--6367.


\bibitem[BW11]{BW}
F. Bonahon and H. Wong,
{\em Quantum traces for representations of surface groups in $\mathrm{SL}_{2}(\mathbb{C})$},
Geom. Topol. \textbf{15} (2011), no.~3, 1569--1615.


\bibitem[CL22]{CL19}
F. Costantino and T. T. Q. L\^{e},
{\em Stated skein algebras of surfaces},
J. Eur. Math. Soc. (JEMS) \textbf{24} (2022), no. 12, 4063--4142.

\bibitem[DS20I]{DS20I}
D. C. Douglas and Z. Sun,
{\em Tropical {F}ock-{G}oncharov coordinates for $SL_3$-webs on surfaces {I}: construction},
arXiv:2011.01768.

\bibitem[Fo97]{Fock}
V.~V.~Fock, 
\emph{Dual Teichm\"uller spaces},
arXiv:dg-ga/9702018.




\bibitem[FG06]{FG06} 
V. V. Fock and A. B. Goncharov, 
{\em Moduli spaces of local systems and higher Teichm\"uller theory},
Publ. Math. Inst. Hautes \'Etudes Sci., \textbf{103} (2006), 1--211.

\bibitem[FG07]{FG07}
V. V. Fock and A. B. Goncharov, 
{\em Dual Teichm\"uller and lamination spaces},
Handbook of Teichm\"uller theory,  Vol. I, 647--684; IRMA Lect. Math. Theor. Phys., \textbf{11}, Eur. Math. Soc., Z\"urich, 2007. 

\bibitem[FG09]{FG09}
V. V. Fock and A. B. Goncharov, 
\emph{Cluster ensembles, quantization and the dilogarithm},
Ann. Sci. \'Ec. Norm. Sup\'er., \textbf{42} (2009), 865--930.

\bibitem[FG16]{FG16}
V. V. Fock and A. B. Goncharov, 
\emph{Cluster {P}oisson varieties at infinity},
Selecta Math. (N.S.) \textbf{22} (2016), 2569--2589.

\bibitem[FST08]{FST}
S. Fomin, M. Shapiro and D. Thurston,
{\em Cluster algebras and triangulated surfaces. {I}. Cluster complexes},
Acta Math. \textbf{201} (2008), 83--146.

\bibitem[GS15]{GS15}
A. B. Goncharov and L. Shen,
{\em Geometry of canonical bases and mirror symmetry},
Invent. Math. \textbf{202} (2015), 487--633. 


\bibitem[GS19]{GS19}
A. B. Goncharov and L. Shen,
\emph{Quantum geometry of moduli spaces of local systems and representation theory},
arXiv:1904.10491v3.


\bibitem[GHKK18]{GHKK}
M. Gross, P. Hacking, S. Keel, and M. Kontsevich,
\emph{Canonical bases for cluster algebras},
J. Amer. Math. Soc. \textbf{31} (2018), no. 2, 497--608. 

\bibitem[Ish19]{Ish19}
T. Ishibashi,
\emph{On a Nielsen-Thurston classification theory for cluster modular groups},
Ann. Inst. Fourier (Grenoble) \textbf{69} (2019), no. 2, 515--560.

\bibitem[IK22]{IK22}
T. Ishibashi and S. Kano,
\emph{Unbounded $\mathfrak{sl}_3$-laminations and their shear coordinates},
arXiv:2204.08947.

\bibitem[IK]{IKar}
T. Ishibashi and H. Karuo,
\emph{Duality maps for quantum cluster varieties from marked surfaces and their ensemble compatibility},
in preparation.

\bibitem[IO20]{IO20}
T. Ishibashi and Hironori Oya,
{\em Wilson lines and their Laurent positivity}, 
arXiv:2011.14260.

\bibitem[IOS22]{IOS22}
T. Ishibashi, H. Oya and L. Shen,
{\em $\mathscr{A}=\mathscr{U}$ for cluster algebras from moduli spaces of $G$-local systems},
arxiv:2202.03168.

\bibitem[IY21]{IYsl3}
T. Ishibashi and W. Yuasa,
{\em Skein and cluster algebras of unpunctured surfaces for $\mathfrak{sl}_3$},
arXiv:2101.00643; to appear in Math. Z.

\bibitem[IY]{IY}
T. Ishibashi and W. Yuasa,
\emph{State-clasp correspondence for skein algebras},
in preparation.

\bibitem[Le16]{Le16}
I. Le,
 {\em Higher laminations and affine buildings},
Geom. Topol. \textbf{20} (2016), no. 3, 1673-–1735.

\bibitem[L\^e18]{TTQLe18}
T. T. Q. L\^{e},
{\em Triangular decomposition of skein algebras},
Quantum Topol. \textbf{9} (2018), no.~3, 591--632.

\bibitem[Mul16]{Muller}
G. Muller,
{\em Skein and cluster algebras of marked surfaces},
Quantum Topol. \textbf{7} (2016), no. 3, 435--503. 

\bibitem[MSW13]{MSW}
G. Musiker, R. Schiffler and L. Williams,
\emph{Bases for cluster algebras from surfaces},
Compos. Math. \textbf{149} (2013), 217--263.

\bibitem[Pen87]{Penner87}
R.~C. Penner, 
\emph{The decorated Teichm\"uller space of punctured surfaces},
Comm. Math. Phys. \textbf{113} (1987), no. 2, 299--339.

\bibitem[Pen]{Penner}
R.~C. Penner, 
\emph{Decorated {T}eichm\"uller theory}, 
QGM Master Class Series, European Mathematical Society (EMS), Z\"urich, 2012.

\bibitem[Qin21]{Qin21}
F. Qin,
{\em Cluster algebras and their bases},
arXiv:2108.09279.

\bibitem[She22]{Shen20}
L. Shen,
{\em Duals of semisimple Poisson-Lie groups and cluster theory of moduli spaces of G-local systems}, 
Int. Math. Res. Not. IMRN 2022, no. 18, 14295--14318.

\bibitem[Thu]{Thu}
W. P. Thurston,
\emph{Three-dimensional geometry and topology,  Vol. 1},
Princeton University Press, Princeton, NJ, 1997.

\end{thebibliography}
